\documentclass[12pt,a4paper]{article}
\usepackage[textwidth=16cm, textheight=23cm,marginratio=1:1]{geometry}
\usepackage[T1]{fontenc}            
\usepackage[utf8]{inputenc}
\usepackage{amsmath}
\usepackage{amsfonts}
\usepackage{mathtools}
\usepackage{amssymb}
\usepackage{comment}
\usepackage{tikz-cd}
\usepackage{pstricks,xy,pst-node}
\usepackage{graphicx}
\xyoption{all}
\usepackage{amsthm}
\usepackage{booktabs} 
\usepackage{mathrsfs} 
\usepackage{url}
\usepackage{caption} 
\usepackage{enumitem}
\usepackage{sidecap}

\usepackage[breaklinks,bookmarks=false, hidelinks]{hyperref} 
\hypersetup{
	colorlinks,
	linkcolor={red}
}

\makeatletter
\let\c@table\c@figure 
\let\ftype@table\ftype@figure 
\makeatother

\numberwithin{equation}{section}

\newtheorem{theorem}[equation]{Theorem}
\newtheorem{proposition}[equation]{Proposition}
\newtheorem{corollary}[equation]{Corollary}

\newtheorem{lemma}[equation]{Lemma}
\newtheorem{conjecture}[equation]{Conjecture}

\theoremstyle{definition}
\newtheorem{definition}[equation]{Definition}
\newtheorem{remark}[equation]{Remark}
\newtheorem{notation}[equation]{Notation}
\newtheorem{setting}[equation]{Setting}
\theoremstyle{remark}
\newtheorem{example}[equation]{Example}

\DeclareMathOperator{\Ann}{Ann}

\DeclareMathOperator{\Hom}{Hom}

\DeclareMathOperator{\Sym}{Sym}

\DeclareMathOperator{\codim}{codim}

\DeclareMathOperator{\HHH}{H}

\DeclareMathOperator{\modulo}{mod}

\newcommand{\Hilb}{{\ccH}ilb}
\DeclareMathOperator{\id}{id}

\newcommand{\univ}{\operatorname{univ}}

\newcommand{\PP}{\mathbb{P}}

\def\AAA{{\mathbb A}}
\def\CC{{\mathbb C}}

\def\PP{{\mathbb P}}

\def\RR{{\mathbb R}}
\def\ZZ{{\mathbb Z}}

\def\kk{{\Bbbk}}
\newcommand{\cactus}[2]{\mathfrak{K}_{#1}\left( #2 \right)}
\newcommand{\cactussch}[2]{\mathfrak{K}^{\operatorname{sch}}_{#1}\left( #2 \right)}
\newcommand{\cactusflat}[2]{\mathfrak{K}^{\operatorname{flat}}_{#1}\left( #2 \right)}

\newcommand{\ccE}{{\mathcal{E}}}
\newcommand{\ccF}{{\mathcal{F}}}
\newcommand{\ccG}{{\mathcal{G}}}
\newcommand{\ccH}{{\mathcal{H}}}
\newcommand{\ccI}{{\mathcal{I}}}
\newcommand{\ccJ}{{\mathcal{J}}}

\newcommand{\ccM}{{\mathcal{M}}}

\newcommand{\ccO}{{\mathcal{O}}}
\newcommand{\ccP}{{\mathcal{P}}}
\newcommand{\ccQ}{{\mathcal{Q}}}
\newcommand{\ccR}{{\mathcal{R}}}
\newcommand{\ccS}{{\mathcal{S}}}

\newcommand{\ccW}{{\mathcal{W}}}

\newcommand{\gotm}{\mathfrak{m}}
\newcommand{\gotn}{\mathfrak{n}}
\newcommand{\gotq}{\mathfrak{q}}

\newcommand{\gots}{\mathfrak{s}}

\newcommand{\familiar}[1][r]{familia$#1$}
\newcommand{\Familiar}[1][r]{Familia$#1$}
\newcommand{\freemodule}[2][A]{{}_{#1}#2}
\newcommand{\affinespaceof}[2][\kk]{\AAA_{#1}\left(#2\right)}

\DeclareMathOperator{\ev}{ev}

\DeclareMathOperator{\length}{length}

\DeclareMathOperator{\Proj}{Proj}
\DeclareMathOperator{\rk}{rk}

\DeclareMathOperator{\Spec}{Spec}
\DeclareMathOperator{\sat}{sat}
\DeclareMathOperator{\Minors}{Minors} 

\newcommand{\Gor}{\operatorname{Gor}}
\newcommand{\HilbGor}[2]{\Hilb_{#1}^{\Gor}{(#2)}}

\newcommand{\SheafySpec}[1]{\ccS{}\mathit{pec}_{#1} \ }
\newcommand{\SheafyProj}[1]{\ccP{}\mathit{roj}_{#1} \ }

\def\hook{\lrcorner}

\newcommand{\set}[1]{\left\{#1\right\}}
\newcommand{\fromto}[2]{#1, \dotsc, #2}
\newcommand{\setfromto}[2]{\set{\fromto{#1}{#2}}}
\newcommand{\linspan}[1]{\left\langle#1\right\rangle}
\newcommand{\rellinspan}[1]{\linspan{#1}_{\operatorname{rel}}}

\newcommand{\reduced}[1]{{#1}_{\operatorname{red}}}

\newenvironment{prf}[1][]
  {\medskip\par\noindent{\bf Proof#1. }}
  {\nopagebreak\qed\par\smallskip}

\newcommand{\noprf}{\nopagebreak\qed\par}

\newcounter{betweenenumi}
\newenvironment{interruptedenumerate}{\setcounter{betweenenumi}{0}}{}
\newenvironment{insideenumerate}{%
\begin{enumerate}
\setcounter{enumi}{\arabic{betweenenumi}}
}
{%
\setcounter{betweenenumi}{\arabic{enumi}}
\end{enumerate}}

\renewcommand{\theenumi}{(\roman{enumi})}

\newcommand{\eemail}[1]{{\href{mailto:#1}{\nolinkurl{#1}}}}

\title{Cactus scheme, catalecticant minors, and scheme theoretic equations}

\author{Jaros{\l}aw Buczy{\'n}ski
\and Hanieh Keneshlou}

\usepackage[textsize=tiny]{todonotes}

\newcommand{\ann}[1][F]{\operatorname{Ann}\!\left(#1\right)}
\newcommand{\apolar}[1][F]{\operatorname{Apolar}\!\left(#1\right)}

\begin{document}

\maketitle

\begin{abstract}
The $r$-th cactus variety of a subvariety $X$ in a projective space generalizes the $r$-th secant variety of $X$ and it is defined using linear spans of finite subschemes of $X$ of degree $r$.
One of its original purposes was to study the vanishing sets of catalecticant minors.
In this article, we equip the cactus variety with a scheme structure, via ``relative linear spans'' of families of finite schemes over a potentially non-reduced base.
In this way, we are able to study the vanishing scheme of the catalecticant minors.
For a sufficiently high degree Veronese variety, we show that $r$-th cactus scheme and the zero scheme of appropriate catalecticant minors agree on a dense open subset which is the complement of the $(r-1)$-th cactus variety (or scheme).
This article is the first part of a series.
In the follow up,
as an application,
we can describe the singular locus of (in particular) secant varieties to high degree Veronese varieties in terms of singularities of the Hilbert scheme.
We will also generalize the result to high degree Veronese reembeddings of other varieties and schemes.
\end{abstract}

\medskip
{\footnotesize
\noindent\textbf{addresses:} \\
J.~Buczy\'nski, \eemail{jabu@mimuw.edu.pl},
   Institute of Mathematics of the Polish Academy of Sciences, ul. \'Sniadeckich 8, 00-656 Warsaw, Poland.\\
H.~Keneshlou, \eemail{hanieh.keneshlou@uni-konstanz.de},
   Department of Mathematics and Statistics, Universität Konstanz, 
   Universit{\"a}tsstra{\ss}e 10, 78464 Konstanz, Germany.

\noindent\textbf{keywords:}\\
cactus variety, cactus scheme, catalecticant minors, Veronese variety, secant variety, scheme theoretical equations, families of schemes, linear span.

}

\tableofcontents

\section{Introduction}\label{sec_intro}
Throughout this paper, we work over an algebraically closed field $\kk$ of any characteristic.

Let $V$ be a $\kk$-vector space of dimension $(n+1)$ and $\PP_\kk(V)$
  be the projective space with a homogeneous coordinate ring
  $S:=\kk[\alpha_0,\ldots,\alpha_n]= \bigoplus_{d=0}^{\infty} S^{d}V^*$.
Here $S^{d}V^*$ is the vector space of homogeneous degree $d$ polynomials in $S$.
Let $S^{(d)}V$ denote the $d$-th symmetric tensor power of $V$,
  which we can consider as degree $d$ homogeneous part
  of the divided powers ring dual to $S$.
The main results of this article concern certain natural subschemes of the projective space $\PP_{\kk}(S^{(d)}V)$,
whose homogeneous coodrinate ring is $\bigoplus_{e=0}^{\infty} S^e(S^{d}V^*)$.
In particular, we will consider homogeneous ideals in 
$\bigoplus_{e=0}^{\infty} S^e(S^{d}V^*)$ and 
the vector space $S^{d}V^* = S^1(S^{d}V^*)$ plays a double role in our story: on one hand it is a degree $d$ homogeneous part of $S$, and on the other hand, it is the degree $1$ part of $\bigoplus_{e=0}^{\infty} S^e(S^{d}V^*)$.
  
For any $0\leqslant i\leqslant d$ the $(i,d-i)$-th catalecticant matrix is a matrix whose entries are elements of $S^d V^*$ interpreted as linear forms on $\PP_{\kk}(S^{(d)}V)$. 
More precisely, considering the linear embedding
\[
  S^{(d)}V  \subset S^{(i)}V\otimes_{\kk} S^{(d-i)} V
  \simeq \Hom_{\kk}(S^{i}V^*,S^{(d-i)} V)
\]
of symmetric tensors into partially symmetric tensors, the catalecticant matrix of $F\in S^{(d)}V$ is the matrix of the linear map $F\colon S^{i}V^* \to S^{(d-i)} V$.
Thus, such matrix determines the following map of vector bundles on $\PP_{\kk}(S^d V)$:
\[
  \Upsilon^{i,d-i}\colon \ccO_{\PP_\kk(S^{(d)}V)} \otimes_{\kk}
    S^iV^*\longrightarrow  \ccO_{\PP_\kk(S^{(d)}V)}(1)\otimes_{\kk}
    S^{(d-i)}V.
\]

For each $i$ and each positive integer $r$ the $(r+1)\times (r+1)$ minors of the $(i,d-i)$-th catalecticant matrix generate a homogeneous ideal $\Minors_r(\Upsilon^{i,d-i})$ in the homogeneous coordinate ring of $\PP_{\kk}(S^{(d)}V)$.
The purpose of this article is to study the subscheme of $\PP_{\kk}(S^{(d)}V)$ defined by this ideal, which we denote by $\Upsilon^{i,d-i}_r(\PP_{\kk} V)$.

It is known that:
\begin{itemize}
 \item If $r=1$, then for any $1\leqslant i\leqslant d-1$,
   the ideal $\Minors_1(\Upsilon^{i,d-i})$ is equal to the ideal of Veronese variety
   $\nu_d(\PP_{\kk} V)$ \cite{pucci_Veronese_variety_and_catalecticant_matrices}.
  \item If $r=2$ and the base field $\kk$ is $\CC$, then for any $2\leqslant i\leqslant d-2$,
   the ideal $\Minors_2(\Upsilon^{i,d-i})$ is the ideal of the secant variety to the Veronese variety $\sigma_2(\nu_d(\PP_{\CC} V))$ \cite{raicu_3_times_3_minors}.
   
   In particular, in the two cases above this ideal does not depend on $i$.
   \item For any $r\geqslant 1$, the $r$-th secant variety
   $\sigma_r(\nu_d(\PP_{\kk} V))$ is contained in $\Upsilon^{i,d-i}_r(\PP_{\kk} V)$.
   In fact, even more, the $r$-th cactus variety
   $\cactus{r}{\nu_d(\PP_{\kk} V)}$ is contained in $\Upsilon^{i,d-i}_r(\PP_{\kk} V)$.
   \item If $r \leqslant i \leqslant d-r$ and $d\geqslant 2r$,
   then $\reduced{\Upsilon^{i,d-i}_r(\PP_{\kk} V)}$,
 the reduced subscheme
     of $\Upsilon^{i,d-i}_r(\PP_{\kk} V)$, is equal to the cactus variety
   $\cactus{r}{\nu_d(\PP_{\kk} V)}$
   (if $\kk=\CC$, then see
   \cite[Thm~1.5]{nisiabu_jabu_cactus};
   for arbitrary $\kk$, see
   \cite[Thm~7.8]{jabu_jelisiejew_finite_schemes_and_secants}).
   In this case, the ideal
   $\Minors_r(\Upsilon^{i,d-i})$
   might potentially depend on $i$,
   but the reduced vanishing set is again
   the same for any $i$ in this range.
\end{itemize}

The \emph{cactus variety} mentioned in the last two items above is defined as the following closure:
  \[
  \cactus{r}{\nu_d(\PP_{\kk} V)}= \overline{\bigcup_{\substack{R
  \subset \PP_{\kk} V \\ \length R \leqslant r}} \set{\text{scheme theoretic linear span of }\nu_d(R)}}.
  \]
Our goal in this article is to introduce a geometrically meaningful scheme structure on
  the cactus variety and compare it to the scheme structure of
  $\Upsilon^{i,d-i}_r(\PP_{\kk} V)$ for $r \leqslant i \leqslant d-r$.

\subsection{Cactus scheme and catalecticant minors}

We define the cactus scheme (Definition~\ref{def_cactus_schemes}) using the relative linear spans (Definition~\ref{def_relative_linear_span}) of not necessarily flat families of subschemes of $\nu_d(\PP_{\kk} V)$,
and we denote the resulting scheme as
$\cactussch{r}{\nu_d(\PP_{\kk} V)}$, see Section~\ref{sec_cactus_scheme} for details and their properties. One can also define another scheme $\cactusflat{r}{\nu_d(\PP_{\kk} V)}$  using spans of only flat families
(also see Definition~\ref{def_cactus_schemes}).
This latter scheme has a nice description as an image under a morphism from a projectivised vector bundle over the Hilbert scheme of $r$ points on $\PP_{\kk} V$,
see Lemma~\ref{lem_cactus_flat_is_determined_by_Hilb}.

In the theorem below and throughout the paper, for any locally closed subschemes $X$ and $Y$ in the same ambient scheme (for instance, $X, Y \subset \PP_{\kk}(S^{(d)}V)$) by $ X \setminus Y$ we mean the difference of the schemes, that is the locally closed subscheme supported on the set $\reduced{X}\setminus \reduced{Y}$ with scheme structure around $x\in \reduced{X}\setminus \reduced{Y}$ identical as the scheme structure of $X$ around $x$.
In particular, $X\setminus Y$ does not depend on the scheme structure of $Y$.

\begin{theorem}\label{thm_cactus_scheme_and_minors_for_PV}
   We have the following relations between the locally closed subschemes
   of $\PP_{\kk}(S^{(d)}V)$:
   \begin{enumerate}
    \item
    \label{item_inclusion_of_cactus_in_catalecticant}
    $\cactusflat{r}{\nu_d(\PP_{\kk} V)} \subset \cactussch{r}{\nu_d(\PP_{\kk} V)} \subset \Upsilon^{i,d-i}_r(\PP_{\kk} V)$ for any $r\geqslant 1$, $d\geqslant r$, $r-1\leqslant i\leqslant d-1$.
    \item
    \label{item_equality_of_cactus_in_catalecticant}
If $r\geqslant 1$, $d\geqslant 2r$, $r\leqslant i\leqslant d-r$,
    then
    \begin{multline*}
      \cactusflat{r}{\nu_d(\PP_{\kk} V)} \setminus \cactusflat{r-1}{\nu_d(\PP_{\kk} V)} = \cactussch{r}{\nu_d(\PP_{\kk} V)} \setminus \cactussch{r-1}{\nu_d(\PP_{\kk} V)}\\
      = \Upsilon^{i,d-i}_r(\PP_{\kk} V) \setminus \Upsilon^{i,d-i}_{r-1}(\PP_{\kk} V).
    \end{multline*}
   \end{enumerate}
\end{theorem}
The theorem is proven in Section~\ref{sec_cactus_of_Veronese}, with its major ingredients stated as Theorems~\ref{thm_cactus_scheme_contained_in_vanishing_of_catalecticant_minors}, \ref{thm_relative_apolarity_lemma}, and \ref{thm_Upsilon_contained_in_cactus}.
It is perhaps worth to explicitly remark about Theorem~\ref{thm_cactus_scheme_and_minors_for_PV}
that:
\begin{itemize}
 \item The nontrivial inclusions to prove are $\cactussch{r}{\nu_d(\PP_{\kk} V)} \subset \Upsilon^{i,d-i}_r(\PP_{\kk} V)$ in~\ref{item_inclusion_of_cactus_in_catalecticant}
 and $\Upsilon^{i,d-i}_r(\PP_{\kk} V) \setminus \Upsilon^{i,d-i}_{r-1}(\PP_{\kk} V) \subset   \cactusflat{r}{\nu_d(\PP_{\kk} V)} \setminus \cactusflat{r-1}{\nu_d(\PP_{\kk} V)}$ in~\ref{item_equality_of_cactus_in_catalecticant}.
 \item
 The symmetry $\Upsilon^{i,d-i} = (\Upsilon^{d-i,i})^T$ implies
 $\Upsilon^{i,d-i}_r(\PP_{\kk} V)=\Upsilon^{d-i,i}_r(\PP_{\kk} V)$
 and thus we can swap the roles of $i$ and $d-i$ in the theorem.
 In particular, the conclusions of \ref{item_inclusion_of_cactus_in_catalecticant} hold also for $1\leqslant i \leqslant d-r+1$.
 In fact, we expect that the same conclusion holds for any $d$ and $i$.
 \item
 As a consequence of~\ref{item_equality_of_cactus_in_catalecticant},
 we see that the scheme structure of $\Upsilon^{i,d-i}_r(\PP_{\kk} V) \setminus \Upsilon^{i,d-i}_{r-1}(\PP_{\kk} V)$ does not depend on the choice of $i$, within the specified range $r\leqslant i \leqslant d-r$, similarly to $r=1$, and $r=2$.
\end{itemize}

\subsection{Structure of
Hilbert scheme of finite Gorenstein schemes and catalecticant minors}

Our second result compares the scheme structure of $\Upsilon^{i,d-i}_r(\PP_{\kk} V) \setminus \Upsilon^{i,d-i}_{r-1}(\PP_{\kk} V)$
to the scheme structure of $\HilbGor{r}{\PP_{\kk} V}$, the Hilbert scheme of finite Gorenstein subschemes of $\PP_{\kk} V$ of degree $r$.
See \cite{jelisiejew_PhD}
for a comprehensive survey on
$\Hilb_r(\PP_{\kk} V)$ and
$\HilbGor{r}{\PP_{\kk} V}$,
particularly \cite[Eq.~(4.6), Thm~4.7, \S4.3, Def.~4.33]{jelisiejew_PhD}
for the definitions and initial discussions.

\begin{theorem}\label{thm_universal_linear_span_isomorphic_to_cactus_intro}
   Suppose $r\geqslant 1$, $d\geqslant 2r$,
     and $r\leqslant i\leqslant d-r$.
  Then there exists a vector bundle
  $\ccE$ over $\HilbGor{r}{\PP_{\kk} V}$
  and a closed subset $Z^{\operatorname{Gor}} \subset \PP_{\kk}(\ccE)$
    such that:
    \begin{itemize}
       \item the projection $\PP_{\kk}(\ccE) \setminus Z^{\operatorname{Gor}} \to \HilbGor{r}{\PP_{\kk} V}$ is smooth surjective, and
       \item $\PP_{\kk}(\ccE) \setminus Z^{\operatorname{Gor}}$ is isomorphic as a scheme to $\Upsilon^{i,d-i}_r(\PP_{\kk} V) \setminus \Upsilon^{i,d-i}_{r-1}(\PP_{\kk} V)$.
    \end{itemize}
\end{theorem}
We prove the theorem in Section~\ref{sec_cactus_of_Veronese}, where we also construct explicitly the vector bundle $\ccE$ and the closed subscheme  $Z^{\operatorname{Gor}}$.

We remark that the combinations of Theorems~\ref{thm_cactus_scheme_and_minors_for_PV} and~\ref{thm_universal_linear_span_isomorphic_to_cactus_intro} give also the isomorphism:
\[
 \PP_{\kk}(\ccE) \setminus Z^{\operatorname{Gor}}
 = \cactusflat{r}{v_d(\PP_{\kk}V)}\setminus \cactusflat{r-1}{v_d(\PP_{\kk}V)}
\]
for $d\geqslant 2r$.
As a consequence, this equality gives us a foundation to study singularities of cactus and secant varieties to high degree Veronese varieties and of catalecticant loci, by comparing them to the singularities of the Hilbert scheme.
This will be explicitly executed in a follow up article, where we aim to establish the following:
\begin{itemize}
 \item If  $d\geqslant 2r$, then
 $\sigma_r(\nu_d(\PP_{\kk}^n))\setminus \sigma_{r-1}(\nu_d(\PP_{\kk}^n))$ is non-singular if and only if  either $n\leqslant 3$ or $r\leqslant 5$.
 \item Moreover, if $d\geqslant 2r$ and
 $\sigma_r(\nu_d(\PP_{\kk}^n))\setminus \sigma_{r-1}(\nu_d(\PP_{\kk}^n))$ is singular, then some of the singularities occur outside of any of the $\sigma_r(\nu_d(\PP_{\kk}^{n-1}))$ loci, compare to \cite[Sect.~6]{furukawa_han_sing_of_secant_vars_of_Veronese}.
 \item For $r\leqslant i \leqslant d-r$,
       and $n$ and $r$ sufficiently large,  $\Upsilon^{i,d-i}_r(\PP_{\kk} V)$ is non-reduced.
\end{itemize}

Results in the same direction as the first two items, but using different methods, have been independently obtained by Doyoung Choi, Justin Lacini, Jinhyung Park, John Sheridan in a paper in preparation
\cite{choi_lacini_park_sheridan_sings_and_syz_of_secant_vars}.
For the third item, we rely on results in preparation by Piotr Oszer.

\subsection{Conjecture of Deery and Geramita}

Our investigations are also partially motivated by questions about the structure of the catalecticant ideals posed by Geramita in \cite[pp.~151--154]{geramita_catalecticant_varieties}.

\begin{conjecture}[Deery, Geramita]
   \label{conj_deery_geramita}
   The ideal $\Minors_r(\Upsilon^{i,d-i})$ is saturated for any $r, d, i$.
\end{conjecture}
The conjecture is stated in \cite[p.~153]{geramita_catalecticant_varieties} as a ``Theorem of Deery''. According to \cite[p.~154]{geramita_catalecticant_varieties},  ``Deery's argument is very clever, and uses an interesting way to grade the polynomial ring. This will be part of Deery's PhD thesis.''
However, upon an inquiry about the details of the arguments,
in a personal communication from January 2012, Geramita wrote the first named author:
``[Deery] decided to not do mathematics anymore (\dots)\footnote{The omitted part of the communication involves a few more Hollywood style romantic details, with Italy involved in the background, but not relevant to mathematics.} and then he essentially disappeared off the face of the earth and I have never heard from him again. Unfortunately, any proof he might have had disappeared with him as well. I spent a great deal of time trying to remember what he had said about that very result and tried to reconstruct the proof from some notes of his that I had. I was unable to do so. Consequently, that result must remain a question mark.  I have tried, off and on, over several years to reprove it -- all without success.''

At this point we cannot resolve Conjecture~\ref{conj_deery_geramita}, or even contribute to its special cases.
However, we believe that our methods used to prove Theorem~\ref{thm_cactus_scheme_and_minors_for_PV}\ref{item_equality_of_cactus_in_catalecticant} can eventually decide if $\Minors_r(\Upsilon^{i,d-i})$ is saturated for $r\leqslant i \leqslant d-r$.
For instance, if Conjecture~\ref{conj_linear_ideals_are_saturated} is true and equality of
affine versions of
$\cactussch{r}{\nu_d(\PP_{\kk} V)}$ and $ \Upsilon^{i,d-i}_r(\PP_{\kk} V)$ holds
(without removing the lower rank loci),
then $\Minors_r(\Upsilon^{i,d-i})$ is saturated.
On the other hand, if Conjecture~\ref{conj_linear_ideals_are_saturated} fails to hold for many families of schemes, then very likely some examples of non-saturated
$\Minors_r(\Upsilon^{i,d-i})$  can be constructed.

\subsection{Overview}

In Section~\ref{sec_delibarations_schemes_modules_ideals} we propose a slightly non-standard perspective on the schemes and their subschemes, particularly, subschemes of an affine or projective space.
Moreover, we discuss in detail modules over finite local Gorenstein $\kk$-algebras $A$
 and also partially generalize some standard statements about ideals
 in a polynomial ring over $\kk$ to ideals in a polynomial ring with coefficients in $A$.
Further, in Section~\ref{sec_linear_spaces_and_linear_spans} we investigate linear algebra over $A$ as above (and in fact over more general algebras).
That is, we define ``family of linear spaces'', linear span of a family of schemes, and also we define the notion of (linearly) independent family of schemes. We introduce the notion of \familiar{}, which we heavily exploit in following sections. All these notions are accompanied by numerous examples, some of which illustrate unintuitive pathologies, and underline the major differences between this case and linear algebra over a field.

Section~\ref{sec_cactus_scheme} defines and explores the notions of relative linear span and cactus scheme. It concludes with a proof that cactus schemes of Veronese varieties are contained in the zero schemes of catalecticant minors.
In Section~\ref{sec_relative_apolarity}  we discuss the apolarity in a relative, but not necessarily flat setting. It culminates with a
version of the apolarity lemma in this setting. This section relies heavily on the algebra from Section~\ref{sec_delibarations_schemes_modules_ideals}.
Finally, Section~\ref{sec_cactus_of_Veronese} summarizes the results obtained so far
 in order to prove the main results of the article.

\subsection*{Acknowledgements}
We are grateful to Weronika Buczy{\'n}ska,
  Joachim Jelisiejew,
  and Mateusz Micha{\l}ek,
 for numerous discussions and helpful comments.
We also thank Doyoung Choi, Justin Lacini, Jinhyung Park, and John Sheridan for sharing with us their work in preparation.

JB is supported by the National Science Center, Poland, projects ``Complex contact manifolds and geometry of secants'', 2017/26/E/ST1/00231
and ``Advanced problems in contact manifolds, secant varieties, and their generalisations (\mbox{APRICOTS+})''  2023/51/B/ST1/02799.
HK is supported by the Deutsche Forschungsgemeinschaft, Projektnummer 467575307.
Moreover, part of the research towards the results of this article was done during the scientific semester Algebraic Geometry with Applications to Tensors and Secants in Warsaw.
The authors are grateful to many fruitful discussions with the participants,
 and for the partial support by the Thematic Research Programme ``Tensors: geometry, complexity and quantum entanglement'', University of Warsaw, Excellence Initiative – Research University and the Simons Foundation Award No. 663281 granted to the Institute of Mathematics of the Polish Academy of Sciences for the years 2021-2023.

\section{Delibarations about schemes, modules and ideals}
\label{sec_delibarations_schemes_modules_ideals}
In this section we collect several thoughts about scheme structures.
Our main interest is in closed or locally closed subschemes of affine and projective spaces.
In \S\ref{sec_scheme_inclusion}
  we explain that it is often sufficient to focus on finite local schemes.
In \S\ref{sec_finite_subschemes_as_tensors}
  we present our reinterpretation
  of finite subschemes
  of affine or projective spaces `as tensors in $A\otimes_{\kk} W$, where the scheme is $\Spec A$, and the affine or projective spaces corresponds to the $\kk$-vector space $W$.  This method is going to be later useful in the context of relative apolarity (Section~\ref{sec_relative_apolarity}).
  Also the discussion in this \S\ref{sec_Gorenstein_duality} serves as algebraic foundations for  Section~\ref{sec_relative_apolarity}.
  Although in this subsection we never mention the word ``scheme'' explicitly, but instead we focus on the tensor space $A\otimes_{\kk} W$ as above, seeing it as a free $A$ module and we study its submodules and related duality.
  The focus is on the case of finite local Gorenstein algebras $A$, which is justified by the observations from \S\ref{sec_scheme_inclusion}.
  Finally,
  in \S\ref{sec_Macaulay_and_Gotzmann} we investigate the bounds on growth of homogenoeus ideals in $A\otimes_{\kk}\Sym W^*$, analogous to some special cases of Macaulay bounds and Gotzmann's persistence, but in a relative setting and in not necessarily flat setting.

\subsection{Scheme inclusion criterion}\label{sec_scheme_inclusion}

By the very basic definitions in set theory,
  for a fixed set $W$ and two of its subsets $A, B \subset W$,
  we have $A\subset B$ if and only if
  for all $x\in A$ we have $x \in B$.
To check if $A=B$ is equivalent to check if for all $w\in W$, we have $x\in A$ if and only if $x\in B$.
This works well for algebraic varieties, for instance, for subvarieties of affine or projective spaces: to verify if one subvariety is contained in the the other it is enough to check the inclusion on points.
However, on the level of schemes the conditions become more interesting.

\begin{definition}\label{def_enough_k_points}
   We say that a scheme $W$ \emph{has enough $\kk$-points} if $W$ is locally Noetherian and for any point $w\in W$, the closure $\overline{\set{w}}$ contains a point $w' \simeq \Spec \kk$.
\end{definition}

For instance, any scheme $W$ of finite type over $\kk$ has enough $\kk$-points.
If $w\in W$ is a $\kk$-point, then the spectrum $\Spec \ccO_{W,w}$ of the local ring and its completion have enough $\kk$-points.
On the other hand,
  $\AAA^n_{\kk(t)}= \Spec \kk(t)[\fromto{x_1}{x_n}]$
  is Noetherian, but does not have enough $\kk$-points.
Although it is not strictly necessary for all our results to restrict to such schemes, some formulations of the statements and the arguments proving some other statements are simpler when working with schemes that have enough $\kk$-points.
At the same time, the notion is general enough for our needs.
Indeed, Noetherian or locally Noetherian is a natural assumption, many statements in  algebraic geometry
will not work without this assumption.
Projective, affine, quasiprojective, quasiaffine over an algebraically closed field $\kk$ have enough $\kk$-points.
All these are natural assumptions,
and in fact these are all we use in the statements and proofs.
The proofs are easier, as for all local schemes with enough $\kk$-points we always have the splitting of the structure ring:
$\kk \to \ccO_{W,w} \to \ccO_{W,w}/\gotm = \kk$.
thus $\ccO_{W,w} = \kk \oplus \gotm$ (as a $\kk$-vector space), which is heavily exploited in some arguments,
  see for instance statement and proof of Proposition~\ref{prop_residue_and_derivative_of_tensor_corresponding_to_scheme} or
  proof of Lemma~\ref{lem_annihilator_of_a_point_from_socle}.
The splitting fails if the base field is $\RR$ and $\ccO_{W,w} = \RR[s,t]/(s^2+1, t^2)$, or if $\kk=\CC$ and $\ccO_{W,w} = \CC(t)[s]/s^2$, or analogous examples.
Dealing with these local rings is not a disaster,
but it is a nuisance and  adds a bit to some proof.
The following proposition shows that if a scheme has enough $\kk$-points, then we do not have to worry about
  these local rings.

\begin{proposition}
\label{prop_scheme_inclusion_equivalence}
   Suppose $W$ is a $\kk$-scheme which has enough $\kk$-points,
   and let $X$ and $Y$ be two closed subschemes of $W$.
   Then the following conditions are equivalent:
   \begin{enumerate}
    \item \label{item_scheme_inclusion}
          $X\subset Y$ as schemes,
    \item \label{item_all_finite_subschemes}
          for all finite (over $\kk$) local subschemes
             $R\subset W$ if $R\subset X$, then $R \subset Y$,
    \item  \label{item_Gorenstein_finite_subschemes}
          for all finite (over $\kk$) local Gorenstein subschemes
             $R\subset W$ if $R\subset X$, then $R \subset Y$.
   \end{enumerate}
\end{proposition}

The equivalence of \ref{item_all_finite_subschemes} and \ref{item_Gorenstein_finite_subschemes} can seen as the statement, that every finite local scheme is a (scheme theoretic) union of its Gorenstein subschemes.

\begin{proof}
   The implications \ref{item_scheme_inclusion} $\Longrightarrow$ \ref{item_all_finite_subschemes} $\Longrightarrow$ \ref{item_Gorenstein_finite_subschemes} are immediate.
   Thus, assume that \ref{item_Gorenstein_finite_subschemes} holds.

   Let $w\in X$ be any point and let $x$ be a $\kk$-point
     in $\overline{\set{w}}$.
   Since $X$ is closed $x \in X$ and taking the Gorenstein scheme $R = \Spec \kk$ supported at $x$ we have $x\in Y$ by Condition~\ref{item_Gorenstein_finite_subschemes}.
   Replacing $W$ by an open neighborhood of $x$
     we can assume that $W$ is affine.
   Furthermore, we can assume that $W=\Spec \ccO_{W,x}$ is local (and Noetherian).

   Let $\gotm \subset \ccO_{W,x}$ be the maximal ideal,
      $I_X \subset \ccO_{W,x}$ be the ideal of $X$
      and $I_Y \subset \ccO_{W,x}$ be the ideal of $Y$.
   We have to show that $I_Y \subset I_X$.
   Suppose on the contrary, that there exists $f \in I_Y$, but $f\notin I_X$.
   Let $k$ be the maximal integer such that $f \in \gotm^{k-1} + I_X$ (so that $f \notin \gotm^{k} + I_X$).
   Note that $k$ is finite since $\ccO_{W,x}$ is Noetherian and $k>0$ since $f\in \gotm^{0} + I_X= (1)$.

   Define $I_0 = \gotm^k + I_X$ and construct a finite sequence of ideals $I_0 \subset I_1 \subset \dotsb \subset I_l$ inductively.
   Note that the quotient $\ccO_{W,x}/I_0$ is local and finite $\kk$-algebra by our assumption on (local) Noetherianity of $W$ and $\gotm^k\subset I_0$.
   Therefore the quotient is Artinian.
   Thus also $\ccO_{W,x}/I_j$ is going to be local Artinian for any $j$, since $I_0\subset I_j$.
   If the socle of $\ccO_{W,x}/I_j $ is $ \langle f + I_j\rangle$, then set $l = j$ and stop the construction.
   If not, then there is an element $g \in \ccO_{W,x}\setminus I_j$
     such that $f \notin I_j + (g)$.
   For instance,
     one can pick any $g$ that modulo $I_j$ is a non-zero element of the socle which is not proportional to the class of $f$.
   Define $I_{j+1}:=I_j + (g)$.
   The construction will eventually stop because the algebra $\ccO_{W,x}/I_0$ is Artinian.

   Then $\Spec \ccO_{W,x}/I_l$ is a local Artinian Gorenstein subscheme of $X$.
   By \ref{item_Gorenstein_finite_subschemes} it is contained in $Y$, thus $I_Y \subset I_l$, a contradiction,
      since $f\in I_Y$, but $f \notin I_l$.
\end{proof}

Let us remark that this criterion has also an interesting analogue in an invariant setting, when there is a suitable group action involved.
Since it is not directly required in the arguments presented in this article, we leave it out of this paper.

\subsection{Finite subschemes as tensors}\label{sec_finite_subschemes_as_tensors}

For every vector space $W$ over $\kk$ of dimension $N+1$, we consider the underlying affine space
$\AAA_{\kk}^{N+1}$ and projective space $\PP_{\kk}^N$, which we denote by $\affinespaceof{W}$ and $\PP_{\kk} W$ respectively.
Fix $Q=\Spec A$ an affine scheme of finite type over $\kk$, so that $A$ is a finitely generated $\kk$-algebra.
Our main interest is in the case where $Q$ is finite and thus also $A$ is a finite $\kk$-algebra.

Pick a basis $\setfromto{x_0}{x_N}$ of $W$ and let $\setfromto{\alpha_0}{\alpha_N}$ be the dual basis of $W^*$. Then
\[
 \affinespaceof{W}= \Spec \Sym W^*
   = \Spec \kk[\fromto{\alpha_0}{\alpha_N}],
\]
and we think of $W^* = (\Sym W^*)_1$ as a subset of $\Sym W^*$. Moreover, elements of $W^*$, in particular $\alpha_i$, are functionals on $W$.
By a slight abuse of notation, if
$\Theta \in \Sym W^*$ we denote
$1 \otimes \Theta\in A\otimes_{\kk} \Sym W^*$ also by $\Theta$ and if $T\in A\otimes_{\kk} W$ then by $\Theta(T)$ we mean the element of $A$ obtained by evaluating the polynomial $\Theta$ at $T$.

The set of $\kk$-points of $\affinespaceof{W}$ is in a natural bijection with $W$.
Unless potentially confusing,
we will denote the vector in $W$
and the corresponding $\kk$-point in $\affinespaceof{W}$
by the same symbol.
Now we generalize this correspondence to $A$-points.

If $F\colon Q\to \affinespaceof{W}$ is any morphism,
then $F^*\colon \Sym W^* \to A$ is the underlying morphism of algebras.
We define the \emph{corresponding tensor} $F_{\otimes} \in A\otimes_{\kk} W = \Hom_{\kk}(W^*, A)$ as $F_{\otimes} := F^*|_{W^*}$.
Explicitly in coordinates,
\[
   F_{\otimes} = F^*(\alpha_0) \otimes x_0 + \dotsb + F^*(\alpha_N) \otimes x_N.
\]

Conversely, given any tensor $T\in A\otimes_{\kk} W$, we can define the \emph{corresponding morphism} of schemes $T_{\operatorname{sch}} \colon Q \to \affinespaceof{W}$,
  such that $T_{\operatorname{sch}}^*\colon \Sym W^* \to A$ is defined as $T\in \Hom_{\kk}(W^*, A)$ on the generators $W^*$ of $\Sym W^*$ and then we extend the linear map to an algebra homomorphism.
In coordinates,
\[
 \text{if } T=a_0\otimes x_0 + \dotsb + a_N\otimes x_N \text{ for some } a_i \in A, \text{ then } T_{\operatorname{sch}}^*(\alpha_i)=a_i.
\]

Clearly, these two operations are inverses of each other:
$(F_{\otimes})_{\operatorname{sch}} =F$
and $(T_{\operatorname{sch}})_{\otimes} =T$.
Moreover, they satisfy the following natural properties:
\begin{proposition}\label{prop_basic_properties_of_tensor_of_a_scheme}
   In the notation of $Q, A, W$ as above,
   suppose $F\colon Q\to \affinespaceof{W}$ is a morphism of schemes.
   Denote by $\overline{F}\subset Q\times_{\kk}\affinespaceof{W}$ the graph of $F$.
   Then:
   \begin{enumerate}
    \item \label{item_embedding_in_terms_of_tensor}
       $F$ is a closed embedding if and only if $F_{\otimes}(W^*)$ generate $A$.
    \item \label{item_evaluating_polynomials_vs_pullback}
       If $\Theta\in \Sym W^*$,
       then $F^*\Theta = \Theta(F_{\otimes})\in A$.
    \item  \label{item_ideal_of_the_image_in_terms_of_tensor}
   The ideal of the scheme theoretic image
   $F(Q)$ is
    \[
      I =\set{\Theta \in \Sym W^*\mid \Theta (F_{\otimes})=0} \subset \Sym W^*.
    \]
    \item  \label{item_ideal_of_graph_in_terms_of_tensor}
   The ideal in $A\otimes_{\kk} \Sym W^*$ of $\overline{F}$ is equal to
   \[
      \set{\Theta \in A\otimes \Sym W^*\mid \Theta (F_{\otimes})=0} \subset \Sym W^*.
   \]
    \item \label{item_composition_in_terms_of_tensors}
    If $Q'=\Spec A'$ is another affine scheme of finite type and $G\colon Q'\to Q$ is a morphism, then
    \[
       (F\circ G)_{\otimes} = (G^*\otimes \id_{W}) (F_{\otimes}).
    \]
   \end{enumerate}
\end{proposition}
\begin{prf}
   All these items are straightforward consequences of definitions. We provide the details for reader's convenience.

   To prove \ref{item_embedding_in_terms_of_tensor}
   note that $F\colon Q\to \affinespaceof{W}$ is a closed embedding if and only if the ring map is surjective, that is $F^*(\Sym W^*) =A$.
   The latter equality is equivalent to the claim that for any set of generators of $\kk$-algebra $\Sym W^*$, the set of images of these generators generates $A$. In particular, the embedding assumption is equivalent to $F_{\otimes}(W^*)$ generating $A$, as claimed.

   For
   Item~\ref{item_evaluating_polynomials_vs_pullback}
   we express $\Theta = \sum_{J} k_{J} \alpha^J$ in coordinates
   (here $J$ is a multiindex, $k_J\in \kk$, and only finitely many coefficients $k_J$
    are non-zero):
   \begin{multline*}
     \Theta(F_{\otimes}) =
     \Theta\left(\textstyle{\sum_{i=0}^{N}} F^*\alpha_i\otimes x_i\right)
      = \textstyle{\sum_{J}} k_{J} \left(\alpha^J(\textstyle{\sum_{i=0}^{N}} F^*\alpha_i\otimes x_i)\right)\\
      = \textstyle{\sum_{J} k_{J}} \left(\textstyle{\prod_{i\in J}}  F^*\alpha_i\right)
      = F^*\Theta.
   \end{multline*}
   Now  \ref{item_ideal_of_the_image_in_terms_of_tensor} becomes a direct consequence of \ref{item_evaluating_polynomials_vs_pullback}, since $I = \ker F^*$.

   To prove
   \ref{item_ideal_of_graph_in_terms_of_tensor},
   let $\overline{I}$
   denote the ideal of the graph.
   By the definition of a graph, $\overline{I}$~is
     generated by $\alpha_i- F^*(\alpha_i)$.
   Since $F^*(\alpha_i) = \alpha_i(F_{\otimes})$,
   the generators can be written as
    $\set{\alpha_i - \alpha_i(F_{\otimes}) \mid i\in \setfromto{0}{N}}$.
   Now, if $\Theta \in A\otimes_{\kk}\Sym W^*$, then using the generators we replace any variable $\alpha_i$
   by $\alpha_i(F_{\otimes})$,
     to see that $\Theta \equiv \Theta(F_{\otimes}) \mod \overline{I}$.
   Since $\Theta(F_{\otimes})\in A$, and $A\otimes_{\kk}\Sym W^* /\overline{I} \simeq A$,
   it follows that $\Theta \in \overline{I}$ if and only if
   $\Theta \equiv 0 \mod \overline{I}$ if and only if $\Theta(F_{\otimes})= 0$.

   Finally, to see that \ref{item_composition_in_terms_of_tensors} holds
   note that
   \begin{multline*}
     (F\circ G)_{\otimes} =
       \textstyle{\sum_{i=0}^N} (G^*\circ F^* \alpha_i) \otimes x_i
       =
       \textstyle{\sum_{i=0}^N} (G^*\otimes \id_W) (F^* \alpha_i \otimes x_i)\\
       =
       (G^*\otimes \id_W) \left(\textstyle{\sum_{i=0}^N}  F^* \alpha_i \otimes x_i\right)
       =
       (G^*\otimes \id_W) (F_{\otimes}).
   \end{multline*}
\end{prf}

For a while we restrict our attention to finite schemes $Q$.
We commence by presenting three simple examples.

\begin{example}
   Suppose $Q=\kk[t]/t^2$.
   Any map $F\colon Q\to \affinespaceof{W}$ is determined uniquely by two vectors $v, v'\in W$:
   $F_{\otimes}= 1 \otimes v  + t \otimes v'$, which we will write as $v+t v'$ for brevity.
   The unique point of support of the image of
   $F$ is $v\in \affinespaceof{W}$.
   $F$ is an embedding if and only if $v'\ne 0$.
   These two observations are generalized in
 Proposition~\ref{prop_residue_and_derivative_of_tensor_corresponding_to_scheme}\ref{item_support_using_residue} and \ref{item_embedding_using_derivative}.
   For two morphisms $F, G \colon Q\to \affinespaceof{W}$ with $F_{\otimes}= v  + t v'$
   and  $G_{\otimes}= w  + t w'$
   their images $F(Q)$ and $G(Q)$ are equal if and only if $v=w$ and $v' = \lambda w'$ for some $\lambda \in \kk\setminus \set{0}$.
\end{example}

\begin{example}
   Suppose $Q_k=\kk[t]/t^k$.
   Any map $F\colon Q_k\to \affinespaceof{W}$ is determined uniquely by $k$ vectors $v, v', \dotsc, v^{(k-1)}\in W$:
   $F_{\otimes}=  v  + t v' + \dotsb + t^{k-1}v^{(k-1)}$.
   The unique point of support of
   $F$ is $v\in \affinespaceof{W}$.
   $F$ is an embedding if and only if $v'\ne 0$
    (again, compare to
    Proposition~\ref{prop_residue_and_derivative_of_tensor_corresponding_to_scheme}).
   On the other hand, if $v'=0$,
     then the map $F$
     factorizes through $Q_{k-1}$, that is $Q_k\to Q_{k-1}\stackrel{G}{\to} \affinespaceof{W}$
     with $G_{\otimes} = v  + t v^{(2)} + \dotsb + t^{k-2}v^{(k-1)}$.
   If $v'\neq 0$ and $F_{\otimes}= v + t v'$ (without higher order terms), then the image of $Q_k$ is contained in the affine line $\AAA^1$ passing through $v$ and tangent to $v'$ and $F(Q_k)$ is the unique subscheme of $\AAA^1$ of length $k$ supported at $v$.
   More generally, the linear span
   of $F(Q_k)$
   (the smallest linear subspace $V \subset W$ such that $F(Q_k) \subset \affinespaceof{V}$)
   is equal to the linear span of $v, v', \dotsc, v^{(k-1)}$.
\end{example}

\begin{example}
   Let $Q=\Spec\kk[s, t]/(s^2,t^2)$ and suppose $W\simeq \kk^3$ with a basis $\set{x_0, x_1, x_2}$.
   Consider $F\colon Q\to \affinespaceof{W}$ such that
   $F_{\otimes} = s x_1 + t x_2 + s t x_0$.
   This $F$ is an embedding of $Q$, $F(Q)$ is supported at $0\in \affinespaceof{W}$,
   and the ideal of $F(Q)$ is:
   $
     (\alpha_0-\alpha_1 \alpha_2, \alpha_1^2, \alpha_2^2)
   $.
\end{example}

If $Q=\Spec A$ is a finite scheme over $\kk$,
  then it is a finite disjoint union
    of its connected components
  $Q=Q_1\sqcup \dotsb \sqcup Q_l$,
  with each $Q_i= \Spec A_i$ finite local scheme,
  and $A= A_1\times \dotsb \times A_l$.

\begin{proposition}
   Suppose $Q$ is a finite scheme as above,
   and $F\colon Q\to \affinespaceof{W}$ is a morphism.
   Set $F_i\colon Q_i \to \affinespaceof{W}$ to be the restriction
     of $F$ to the irreducible component $Q_i$.
   Then
   $F_{\otimes} =
   (F_{1})_{\otimes} + \dotsb + (F_{l})_{\otimes}$.
\end{proposition}
\begin{proof}
   We have $F_{\otimes} \in A\otimes_{\kk}W$ and
   as $\kk$-vector spaces
    $A\otimes_{\kk}W= A_1\otimes_{\kk}W \oplus \dotsb \oplus A_l\otimes_{\kk} W$.
    Decomposing $F_{\otimes}$ into direct components as above
    we obtain the desired expression of $F_{\otimes}$ as the sum of $(F_{i})_{\otimes}$'s.
\end{proof}

Suppose $Q=\Spec A$ is a finite local scheme over $\kk$ and
$\gotm \subset A$ is the maximal ideal.
Then we have a natural splitting $A= \kk\cdot 1 \oplus \gotm$ as $\kk$-vector spaces.
Thus, each tensor $T\in \freemodule{W}$
  can be uniquely written as
  $T=T_{\kk} + T_{\gotm}$ with
  $T_{\kk} = (T \modulo \gotm\otimes_{\kk} W) \in (A/\gotm) \otimes_{\kk} W = W$ and $T_{\gotm}\in \gotm\otimes_{\kk} W$.
As we see in Proposition~\ref{prop_residue_and_derivative_of_tensor_corresponding_to_scheme} below, $T_{\kk}$ corresponds to the point of support of $T_{\operatorname{sch}}(Q)$ and $T_{\gotm}$ corresponds to how $Q$ is mapped into $\affinespaceof{W}$ up to translation by $T_{\kk}$.
In particular, $T_{\gotm} \modulo \gotm^2$ is responsible for the image of the Zariski tangent space of $Q$.

\begin{proposition}
   \label{prop_residue_and_derivative_of_tensor_corresponding_to_scheme}
   With $Q$ as above,
   let $F \colon Q\to \affinespaceof{W}$ be a morphism.
   Then:
   \begin{enumerate}
    \item
        \label{item_support_using_residue}
        The support of $F(Q)$ is $(F_{\otimes})_{\kk}$ seen as a $\kk$ point in $\affinespaceof{W}$.
    \item
        \label{item_embedding_using_derivative}
        $F$ is an embedding if and only if
    \[
      ((F_{\otimes})_{\gotm} \modulo \gotm^2) \in (\gotm/\gotm^2)\otimes_{\kk} W \simeq \Hom_{\kk} \left((\gotm/\gotm^2)^*, W\right)
    \]
is injective.   \end{enumerate}
\end{proposition}

\begin{prf}
  Let $Q'\simeq \Spec \kk \subset Q$
  be the reduced subscheme of $Q$
  and let $G\colon Q'\to Q$ be the embedding morphism.
  Then the first claim follows from
  Proposition~\ref{prop_basic_properties_of_tensor_of_a_scheme}\ref{item_composition_in_terms_of_tensors}.

  An element of $\Hom_{\kk} \left((\gotm/\gotm^2)^*, W\right)$ is injective if and only if the same element seen as $\Hom_{\kk} (W^*, \gotm/\gotm^2)$ is surjective.
  Since $\gotm/\gotm^2 = \gotm \otimes_A (A/\gotm)$,
  by Nakayama's Lemma,
  $(F_{\otimes})_{\gotm} \modulo \gotm^2$
  is a surjective element of $\Hom (W^*, \gotm/\gotm^2)$
  if and only if $(F_{\otimes})_{\gotm}(W^*)\subset \gotm$ generate $\gotm$ as an $A$-module.
  For any $\alpha \in W^*$, we have
  $(F_{\otimes})_{\gotm} (\alpha)= (F_{\otimes} - (F_{\otimes})_{\kk})(\alpha) =F_{\otimes}(\alpha)-\alpha((F_{\otimes})_{\kk})$,
  with the second term in $\kk\subset A$.
  Thus $(F_{\otimes})_{\gotm}(W^*) \subset F_{\otimes}(W^*) + \kk \subset A$, and vice versa,
  $F_{\otimes}(W^*) \subset (F_{\otimes})_{\gotm}(W^*) + \kk \subset A$.
  In particular, $(F_{\otimes})_{\gotm}(W^*)$ generates $A$ if and only if $F_{\otimes}(W^*)$ does,
  which is equivalent to $F$ being an embedding by Proposition~\ref{prop_basic_properties_of_tensor_of_a_scheme}\ref{item_embedding_in_terms_of_tensor}.
\end{prf}

We want to analyse similarly subschemes of $\PP_{\kk} W=\Proj (\Sym W^*)$.
Denote by $\pi \colon \affinespaceof{W}\setminus\set{0} \to \PP_{\kk} W$ the standard quotient map.

\begin{lemma}\label{lem_lifting_finite_schemes_from_projective_spaces}
   Suppose $Q$ is finite or local scheme over $\kk$, and $F\colon Q\to \PP_{\kk} W$ is any morphism.
   Then there exists a morphism $\hat{F}\colon Q\to \affinespaceof{W}$
     such that $F=\pi \circ \hat{F}$.
\end{lemma}
\begin{proof}
   Since $Q$ (and thus its image $F(Q)$) is finite (in particular, $0$-dimen\-sion\-al over $\kk$) or local,
   there exists a hyperplane
   $\set{\alpha=0}\subset\PP_{\kk} W$
   for some $\alpha \in W^*$
   such that $F(Q)\cap \set{\alpha=0} =\emptyset$.
   In other words,
   $F(Q)\subset \set{\alpha\ne0} \simeq \AAA^N$.
   We can now embed $\AAA^N$ onto a height $1$
   hyperplane $Z(\alpha-1)$ of $\affinespaceof{W}$,
   and $F$ factorizes through $Z(\alpha-1)$
   as claimed.
\end{proof}

Thus, by Lemma~\ref{lem_lifting_finite_schemes_from_projective_spaces}
  and the construction of the tensor corresponding to
  a morphism $Q\to \affinespaceof{W}$,
  for every finite scheme $Q$ and a morphism $F\colon Q\to\PP_{\kk} W$
  there exists a tensor $\hat{F}_{\otimes} \in A\otimes_{\kk} W$.
Note that not all tensors in $A\otimes_{\kk} W$
  determine a morphims $Q\to \PP_{\kk} W$
  and the tensor is not uniquely determined.
To see this in detail,
    write $Q=Q_1 \sqcup \dotsb \sqcup Q_l$ as a union of irreducible (and connected) components,
    where $Q_i=\Spec A_i$
    and $A= A_1 \times \dotsb \times A_l$.
Let $\gotm_i \subset A$ be the maximal ideal in
   $A$   corresponding to the local ring $A_i$.
Also let $1_{A_{i}} \in A$ denote the image of $1$ from $A_i$
  in the natural vector space splitting $A= A_1 \oplus \dotsb \oplus A_l$.
For each $\hat{T}\in A\otimes_{\kk} W$ we write uniquely
\[
\hat{T} = \sum_{i=1}^l (1_{A_i}\otimes v_i + \hat{T}_{\gotm_i})
\]
with $v_i\in W$ and $\hat{T}_{\gotm_i} \in (\gotm_i \cap A_i) \otimes_{\kk} W$.

\begin{proposition}
   In the setting as above, $\hat{T}\in A\otimes_{\kk} W$ determines a morphism of schemes $T_{\operatorname{sch}}\colon Q\to\PP_{\kk} W$ if and only if $v_i\ne 0$ for all $i$.
\end{proposition}
\noprf
   The claim of the proposition is intuitive, and straightforward to prove.
   We skip the details.

\subsection{Gorenstein duality for submodules of free modules}\label{sec_Gorenstein_duality}
In this section $A$ is a $\kk$-algebra and $W$ is a finite dimensional $\kk$-vector space.
We will consider finitely generated free modules over $A$, and to be consistent with the rest of the paper we  will view them as $\freemodule{W}:=A\otimes_\kk W$ for some $W$ as above.
We also want to study $\freemodule{W}$
in a company of another free module
$\freemodule{W^*}:= A\otimes_\kk W^*$ arising from the dual vector space $W^* = \Hom_{\kk}(W, \kk)$.
We have the natural pairing $\ev\colon \freemodule{W^*}\times \freemodule{W} \to A$: for $\phi\in \freemodule{W^*}$ and $w\in \freemodule{W}$ let $\ev(\phi, w) := \phi(w)$.
The map $\ev$ is $A$-bilinear and
we naturally have $\freemodule{(W^*)^*} = \freemodule{W}$.

Note that $\freemodule{W^*}$ is \emph{not} necessarily the dual module $D(\freemodule{W})$ to $\freemodule{W}$ in the sense of \cite[\S21.1]{Eisenbud},
as $D(\freemodule{W})$ is not necessarily free.
Instead, the geometric interpretation of $\freemodule{W}$ and  $\freemodule{W^*}$ are the modules corresponding to dual (and trivial) vector bundles on $\Spec A$.
The algebraic interpretation  of these ``weird'' duality
is that $\freemodule{W^*} = \Hom_{A} (W, A)$,
in contrast to  the dual module $D(\freemodule{W})$, which is $\Hom_{\kk} (W, \kk)$ (for finite local $A$) or, more generally,
$\Hom_{A} (W, \omega_A)$.

\begin{definition}
   $M\subset \freemodule{W}$ is an $A$-submodule. We define the \emph{perpendicular submodule}
     of $M$ to be
     $
       M^{\perp}:=\set{\phi\in \freemodule{W^*} \mid \forall_{m\in M} \  \phi(m) =0}.
     $
\end{definition}

This is a very natural generalization of the definition for vector spaces over a field, 
yet the behavior of $\perp$ might be slightly unintuitive, 
as the following examples illustrate.
Note that for $A$-submodules $M$ and $M'$ in $\freemodule{W}$ we always have:
\begin{equation}
\label{equ_double_perp_and_reversing_inclussion}
  M \subset M^{\perp\perp}, \quad \text{ and if } M\subset M' \subset \freemodule{W} \text{ then } (M')^{\perp}\subset M^{\perp} \subset \freemodule{W^*}.
\end{equation}

\begin{example}\label{ex_nonGorenstein_fails_double_perpendicularity}
   Suppose $A= \kk[s,t]/(s^2, st, t^2)$ and $W=\kk^2$.
   Let $M = \linspan{(s,0),(0,t)}$.
   Then
   \[
      M^{\perp} = \linspan{(s,0),(t,0),(0,s),(0,t)} \subset \freemodule{W^*}
   \]
   and also
   \[
     M^{\perp\perp} = \linspan{(s,0),(t,0),(0,s),(0,t)} \subset \freemodule{W}.
   \]
   In particular, $M^{\perp\perp} \ne M$.
\end{example}

\begin{example}\label{ex_infinite_fails_double_perpendicularity}
   Suppose $A= \kk[[t]]$ or $A=\kk[t]$ and $W=\kk$.
   Let $M = (t^k) \subset \freemodule{W} \simeq A$
   for some $k>0$.
   Then
   $
      M^{\perp} = 0
   $
   and
   $
     M^{\perp\perp} = \freemodule{W}.
   $
   In particular, $M^{\perp\perp} \ne M$.
\end{example}

In our proofs in Section~\ref{sec_relative_apolarity} we will however need a number of cases with $M = M^{\perp\perp}$.
Thus below we present two such desirable situations in Propositions~\ref{prop_perp_of_perp_for_section} and \ref{prop_perp_of_perp_for_Gorenstein}.
The first of these cases in a sense corresponds to a section of the trivial vector bundle over $\Spec Q$.

\begin{proposition}\label{prop_perp_of_perp_for_section}
   Suppose $A$ is a local $\kk$-algebra with maximal ideal
     $\gotm$,
     and that $T \in \freemodule{W}$ is such that
     its residue $(T \modulo \gotm) \in  A/\gotm \otimes_{\kk} W \simeq W$ is non-zero.
   Finally, let $M=T\cdot A \subset \freemodule{W}$ be the submodule generated by $T$.
   Then $M^{\perp\perp} =M$.
\end{proposition}

\begin{prf}
   Let $T_{\kk} = (T \modulo \gotm) \in W$, and pick a  complement of the set $\set{T_{\kk}}$ to
   a basis $\set{T_{\kk}, \fromto{x_{1, \kk}}{x_{N, \kk}}}$ of $W$.
   Pick also lifts of $x_{i, \kk}$ to $x_i \in \freemodule{W}$ such that $(x_i \modulo \gotm) = x_{i, \kk}$.
   By Nakayama's Lemma
   $\set{T, \fromto{x_1}{x_N}}$  minimally generates $\freemodule{W}$.
   Since $\freemodule{W}$ is a free module, any minimal set of generators
   (in particular, $\set{T, \fromto{x_1}{x_N}}$)
   is a basis of
   $\freemodule{W}$.
   Let $\set{\tau, \fromto{\alpha_1}{\alpha_N}}$ be the dual basis of $\freemodule{W^*}= \Hom(\freemodule{W}, A)$.
   Then $M^{\perp} = A\cdot(\fromto{\alpha_1}{\alpha_N})$ and immediately,
   we get that $M^{\perp\perp} = A\cdot T =M$ as claimed.
\end{prf}

For the rest of  this subsection we focus on finite local Gorenstein $\kk$-algebras.
Before the next proposition we phrase a technical lemma,
  which is a perhaps standard observation
  for which we have not found any reference.

\begin{lemma}\label{lem_Gor_multiplying_collection_to_socle}
   Suppose $A$ is a finite local Gorenstein $\kk$-algebra with a socle $\gots$,
   and $\setfromto{a_1}{a_r}$ is a finite collection of elements of $A$, at least one of which is non-zero.
   Then there exists an element $b \in A$
   such that
   $\setfromto{ba_1}{ba_r} \subset \gots$ and at least one of $ba_i$ is non-zero.
\end{lemma}

\begin{prf}
   Denote by $\gotm$ the maximal ideal of $A$.
   Let $d$ be the socle degree of $A$,
   so that $\gots = \gotm^{d}$.
   Suppose for some $k$ we have $a_i \in \gotm^k$
   for all $i$, but $a_{i_0} \notin \gotm^{k+1}$
   for at least one of the indices $i_0$.
   If $k=d$, then take $b_d=b=1 \in A$, and we are done.

   So suppose $k<d$. Note that $\gotm \cdot a_{i_0}$ is not identically $0$, as otherwise, $a_{i_0}$ is annihilated by the maximal ideal, that $a_{i_0} \in \gots = \gotm^d$, a contradiction with $a_{i_0} \notin \gotm^{k+1}$.
   Pick $b_k \in \gotm$
   such that $b_k a_{i_0} \ne 0$.
   Then, $\setfromto{b_k a_1}{b_k a_r} \subset \gotm^{k+1}$ and it contains a nonzero element, thus we may start over with a higher value of $k$,
   and eventually take $b$ equal to the product of all $b_k$ that appeared in the construction.
\end{prf}

Now we are ready to prove the second important case when $M=M^{\perp\perp}$.

\begin{proposition}\label{prop_perp_of_perp_for_Gorenstein}
   Suppose $A$ is a finite local Gorenstein
   $\kk$-algebra.
   Then for any finite-dimensional vector space $W$ and  an $A$-submodule $M\subset \freemodule{W}$
   we have:
   \begin{enumerate}
    \item \label{item_dual_of_submodule_versus_perp}
       $\Hom_A (M, A) = \freemodule{W^*}/M^{\perp}$, and
    \item \label{item_perp_of_perp_for_Gorenstein}
    $M^{\perp\perp} = M$.
   \end{enumerate}
\end{proposition}
Example~\ref{ex_infinite_fails_double_perpendicularity} illustrates that if we skip the finiteness assumption on the algebra $A$, the conclusions of the proposition fail to hold, even if $A$ is Gorenstein (or even regular) and local. Moreover, for finite algebras we cannot skip the Gorenstein assuption, see Example~\ref{ex_nonGorenstein_fails_double_perpendicularity}. On the other hand, the conclusions hold for finite Gorenstein algebras $A$, but we do not need the non-local case, and the proof is slightly simpler for local $A$.

\begin{prf}
   By definition of Artinian Gorenstein ring, there exists a (nonunique) isomorphism of $A$-modules $A$ and $\Hom_{\kk}(A, \kk)$.
   Pick any such isomorphism $A\to  \Hom_{\kk}(A, \kk)$ and denote by $s^*\colon A \to \kk$ the image of $1\in A$ under this isomorphism.
   Then, for any $a\in A$ the image of $a$
     under the isomorphism is the map taking
     $b\mapsto s^*(a\cdot b)$.
   Note that $s^*$ takes a non-zero value on any non-zero element of the socle of $A$.

   Consider the sequence:
   \[
     0 \to M^{\perp} \to \freemodule{W^*} \to \Hom_A(M, A) \to 0,
   \]
   where $\freemodule{W^*} \simeq \Hom_A (\freemodule{W}, A)$ (using the isomorphism $A\to  \Hom_{\kk}(A, \kk)$ as above) and the second map is the restriction map.
   We claim that the sequence is exact, which proves the first claim.
   The first non-trivial map in the sequence is injective by definition. The second one is surjective by \cite[Prop.~21.5(a,b)]{Eisenbud}.
   Thus, it remains to verify
     that for any $\phi\in \freemodule{W^*}$,
     its image in $\Hom_A(M, A)$ is $0$ if and only if $\phi \in M^{\perp}$.
   Explicitly, $\phi$ is mapped to a morphism that takes $m \to  s^*(\phi(m))$.
   Thus, if $\phi\in M^{\perp}$, then $\phi(m) =0$ for all $m \in M$ and consequently $\phi$ is mapped to $0$.
   On the other hand, if $\phi \notin M^{\perp}$, then take $m \in M$ such that $\phi(m)\ne 0$.
   Pick any $a\in A$ such that $a\cdot \phi(m)$ is non-zero and contained in the socle of $A$, by Lemma \ref{lem_Gor_multiplying_collection_to_socle}.
   Then, $a\cdot m \in M$ and
   \[
    s^*(\phi(a \cdot m)) = s^*(a \cdot \phi(m)) \ne 0,
   \]
   Thus, $\phi$ is mapped to a non-zero element of $\Hom_A(M, A)$, completing the proof of the exactness of the above short sequence.

Note that $\Hom_A(M, A) \simeq \Hom_A(M, \omega_A)\simeq \Hom_{\kk}(M,\kk)$ by \cite[\S21.1]{Eisenbud}.
Thus $\dim_{\kk}\Hom_A(M, A) = \dim_{\kk} M$ and by using the exact sequence twice (for $M$ and $M^{\perp}$) it follows that $\dim_{\kk} M^{\perp\perp}=\dim_{\kk} M$.
Since $M\subset M^{\perp\perp}$ and both $W$ and $A$ are finitely dimensional over $\kk$, we must have $M^{\perp\perp}= M$ as claimed
\end{prf}

In the remainder of this subsection we prove three lemmas
about finite local Gorenstein algebras and finitely generated modules over such rings. These lemmas are important ingredients of the main proofs, but also potentially are of independent interest.
 The first lemma is a criterion for a module over such ring to be free (or, equivalently, flat).
The second lemma is also a freeness criterion, but for submodules of a free module.
The setting for the final lemma is more specific, restricted to $A=\kk[t]/(t^2)$.

\begin{lemma}\label{lem_criterion_for_free_module_over_Gor}
   Suppose $A$ is a finite local Gorenstein $\kk$-algebra.
   Denote by $\gotm$ the maximal ideal of $A$ and by
   $s$ any nonzero element of the ($1$-dimensional) socle of $A$.
   Let $M$ be a finitely generated $A$-module.
   Then the following conditions are equivalent:
   \begin{enumerate}
    \item  \label{item_free_module_criteria_M_free}
        $M$ is a free $A$-module,
    \item \label{item_free_module_criteria_M_locally_free}
    the coherent sheaf $\ccM$ on $\Spec A$ corresponding to $M$ is locally free,
    \item \label{item_free_module_criteria_dimensions_agree}
        $\dim_\kk (s\cdot M) = \dim_{\kk} (M/(\gotm \cdot M))$.
   \end{enumerate}
\end{lemma}
\begin{prf}
   The equivalence of \ref{item_free_module_criteria_M_free} and \ref{item_free_module_criteria_M_locally_free}
   is clear, as $\Spec A$ consists of only one point.
 Pick a basis of $M/(\gotm \cdot M)$ together with its lift $\setfromto{m_1}{m_r}$ to $M$.
   By Nakayama's lemma $\setfromto{m_1}{m_r}$ minimally generate $M$.
   Then $\setfromto{s\cdot m_1}{s\cdot m_r}$
    $\kk$-linearly spans $s\cdot M$, hence we always have $\dim_\kk (s\cdot M) \leqslant \dim_{\kk} (M/(\gotm \cdot M))$.

   Now assume $M$ is free as in \ref{item_free_module_criteria_M_free}.
   Then,
   $M = A\cdot m_1 \oplus \dotsb \oplus A\cdot m_r$
   and thus $s\cdot M = \kk \cdot m_1 \oplus \dotsb \oplus \kk \cdot m_r$ as claimed.
   Conversely, suppose \ref{item_free_module_criteria_M_free} fails to hold, that is there is a non-trivial $A$-relation among $\setfromto{m_1}{m_r}$,
   say:
   \[
     a_1 \cdot m_1 + \dotsb + a_r \cdot m_r =0,
   \]
   with $a_i \in A$.
   By Lemma~\ref{lem_Gor_multiplying_collection_to_socle}, for some $b\in A$ we get $ba_i = c_i s$ with at least one $c_i$ non-zero.
   Therefore,
   \[
     0 =  b a_1 \cdot m_1 + \dotsb + b a_r \cdot m_r =
         c_1 (s \cdot m_1) + \dotsb + c_r (s \cdot  m_r),
   \]
   obtaining a non-trivial relation between the generators $s \cdot m_1, \dotsc, s \cdot m_r$
   of $s\cdot M$.
   Hence $\dim_{\kk} (s\cdot M) < r = \dim_{\kk} (M /(\gotm \cdot M))$ and \ref{item_free_module_criteria_dimensions_agree}
     fails to hold, as claimed.
\end{prf}

\begin{lemma}\label{lem_criterion_for_free_submodule}
   Suppose $A$ is a finite local Gorenstein
   $\kk$-algebra.
   Denote by $\gotm$ the maximal ideal of $A$ and by
   $s$ any nonzero element of the socle of $A$.
   Let $M \subset \freemodule{W}$ be an $A$-submodule.
   Then, the following conditions are equivalent:
   \begin{enumerate}
    \item  \label{item_free_submodule_criteria_M_free}
        $M$ is a free $A$-module,
    \item
      \label{item_free_submodule_criteria_dimensions_agree}
        $\dim_\kk (M \cap s \cdot \freemodule{W})
        = \dim_{\kk} \left(M/(\gotm \cdot \freemodule{W} \cap M)\right)$,
    \item
      \label{item_free_submodule_criteria_intersection_with_socle}
        $s\cdot M = (M \cap s \cdot \freemodule{W})$.
    \end{enumerate}
\end{lemma}

\begin{prf}
       Consider $\freemodule{W}/ (\gotm \cdot\freemodule{W})$
          and the image of $M \to \freemodule{W}/ (\gotm \cdot \freemodule{W})$.
       Pick a basis of the image $\fromto{\overline{m}_1}{\overline{m}_r}$, a complement $\fromto{\overline{w}_{r+1}}{\overline{w}_N}$
of this independent set
       to a basis of
       $\freemodule{W}/ (\gotm \cdot \freemodule{W})$,
       and
       their lifts $\setfromto{m_1}{m_r} \subset M$
       and $\setfromto{w_{r+1}}{w_N} \subset \freemodule{W}$.
       Note that $M = A\cdot (\fromto{m_1}{m_r})
       + M'$, for some $A$-submodule $M' \subset \gotm\otimes_{\kk} \linspan{\fromto{w_{r+1}}{w_N}}$.
       It follows that
       $\setfromto{s\cdot m_1}{s\cdot m_r}$
       is a $\kk$-basis of $s\cdot M$,
       in particular,
       $\dim_{\kk} (s\cdot M) =r
       = \dim_{\kk} (M/(\gotm \cdot \freemodule{W} \cap M)$.
       Since $s\cdot M \subset (M \cap s \cdot \freemodule{W})$
       we always have $\dim_\kk (M \cap s \cdot \freemodule{W}) \geqslant \dim_{\kk} (M/(\gotm \cdot \freemodule{W} \cap M)$ and thus \ref{item_free_submodule_criteria_dimensions_agree} is equivalent to \ref{item_free_submodule_criteria_intersection_with_socle}.

       On the other hand, item  \ref{item_free_submodule_criteria_M_free}
        is equivalent to $M'=0$ by Lemma~\ref{lem_criterion_for_free_module_over_Gor}.
  Thus, if \ref{item_free_submodule_criteria_M_free} holds, then \ref{item_free_submodule_criteria_intersection_with_socle} holds.

  Now assume \ref{item_free_submodule_criteria_intersection_with_socle} holds.
  If $M'\ne 0$, then by Lemma~\ref{lem_Gor_multiplying_collection_to_socle} applied to the $w_i$-coefficients of a non-zero element of $M'$, we get a non-zero element of $(M' \cap s \cdot \freemodule{W})$, which is not in $s\cdot M$, a contradiction.
  Thus, $M'=0$ and \ref{item_free_submodule_criteria_M_free} holds
       as claimed.
\end{prf}

Now assume $A =\kk[t]/(t^2)$. Considering the natural embedding $\kk\subset A$,
for an element $\Phi \in \freemodule{W}= \kk[t]/(t^2)\otimes_{\kk} W$ we can write it (uniquely) as
$\Phi= \varphi + t \varphi'$ with $\varphi, \varphi' \in W$ and define $\partial \Phi:=\varphi'$.
If $M \subset \freemodule{W}$ is a submodule,
then we define the following linear subspace:
\[
  \partial M :=
  \set{\partial\Phi \mid \Phi\in M}\subset W.
\]

\begin{lemma}\label{lem_properties_of_partial}
With $A=\kk[t]/(t^2)$ and notation as above,
for any $A$-submodule $M\subset \freemodule{W}$
we have the following properties of $\partial M$:
\begin{enumerate}
 \item \label{item_fibre_contained_in_partial}
    $M/ (M\cap t W) \subset \partial M$,
 \item
 \label{item_dimensions_of_partial_and_projection}
    Considering $W\subset \freemodule{W}$ induced by $\kk\subset A$ we have
 $\dim_{\kk} M = \dim_{\kk} \partial M +\dim_{\kk} M \cap W$; equivalently,
 $\codim_{\kk} (M\subset \freemodule{W}) = \codim_{\kk} (\partial M\subset W) +\codim_{\kk} (M \cap W\subset W)$.
\end{enumerate}

\end{lemma}
\begin{prf}
  If $\varphi \in M/ (M\cap t W)$,
  then $\varphi +t\varphi'\in M$ for some $\varphi'\in W$
  thus $t\cdot(\varphi +t\varphi') = t\varphi \in M$ and $\varphi = \partial (t\varphi) \in \partial M$, proving \ref{item_fibre_contained_in_partial}.

  To prove \ref{item_dimensions_of_partial_and_projection} consider the surjective $\kk$-linear map
  $\partial\colon M\to \partial M$, taking $\Phi$ to $\partial \Phi$.
  The kernel of this map is precisely $M \cap W$,
  which proves the desired equality of dimensions.
  The equality of codimensions is then straightforward, as $\dim_{\kk}(\freemodule{W}) = 2\dim_{\kk} W$.
\end{prf}

\subsection{Relative Macaulay bound and Gotzmann's persistence}
\label{sec_Macaulay_and_Gotzmann}

Macaulay bound and Gotzmann's persistence theorems are standard tools in commutative algebra
and they govern the numerical behavior of
homogeneous ideals in a standard graded polynomial rings.
For a reference, we state in Proposition~\ref{prop_classical_Macaulay_and_Gotzmann}
the special cases of these theorems that we will need.
The goal of this section is to conclude some analogous consequences in a simplified relative case,
see Proposition~\ref{prop_relative_Macaulay_and_Gotzmann}. We restrict our attention to the generality required in this article.

Suppose $Q =\Spec A$ is a finite local Gorenstein scheme with the maximal ideal $\gotm \subset A$ and the socle $\gots \subset A$, which is $1$-dimensional. Assume further that $S$ is a $\kk$-algebra. The main cases of interest are $S=\Sym V^*$ for a vector space $V$ over $\kk$, and then we will often consider the standard grading on $S$ and $\freemodule{S}:=A\otimes_{\kk} S$: elements of $S^dV^*$ have degree $d$ and $A$ has degree $0$. However, this assumption is not required in the first two lemmas.

We will consider families of ideals in $S$ parameterized by $Q$.
That is, let $I \subset A\otimes_{\kk} S$ be an ideal, and  construct two
ideals in $S$:
\begin{itemize}
 \item $I_{\kk} = I \modulo \gotm \cdot S$ (this is the ideal of the special fibre),
 \item $I_{\gots}  = I \cap \gots \cdot S \subset \gots \cdot A\otimes_{\kk}S \simeq S$ (we call this one the \emph{ideal of socle fibre} of $I$).
\end{itemize}
Note that seemingly the latter ideal $I_{\gots}$
depends on the choice of isomorphism of $S$ modules
$\gots \cdot A\otimes_{\kk}S \simeq S$,
  but this choice is limited to a multiplication
  by an invertible scalar in $\kk\simeq \gots$.
Thus, the resulting ideal is in fact uniquely determined.

It is easy to verify that
$I_{\kk}\subset I_{\gots}$.
Moreover, as illustrated in the following lemma, these two ideals can be used to reduce a general case to the case of the base $\kk[t]/(t^2)$, which we will exploit in the proof of Proposition~\ref{prop_relative_Macaulay_and_Gotzmann}.

\begin{lemma}\label{lem_reduction_to_dual_numbers}
   Suppose $S$, $A$, $\gots \subset A$ and $I\subset \freemodule{S}$ are as above and let $A' = \kk[t]/(t^2)$.
   Fix any non-zero $s\in \gots$,
    and define $\xi\colon A' \to A$ by $t \mapsto s$.
 Denote by the same letter $\xi \colon \freemodule[A']{S} \to \freemodule{S}$  the induced map $\xi\otimes \id_S$ and let $I' = \xi^{-1} (I)$. 
 Then,
   \begin{enumerate}
    \item $I'_{\kk}=I_{\kk}$ and  $I'_{\gots'}=I_{\gots}$, where $\gots'=(t) \subset A'$ is the socle of $A'$,
 and
    \item if $\gotq \subset S$ is any ideal
    and $(I : \gotq \cdot  \freemodule{S}) =I$,
      then $(I' : \gotq \cdot  \freemodule[A']{S}) =I'$.
   \end{enumerate}
\end{lemma}
\begin{prf}
   The first item is immediate from the definitions.
   For the second item, let
   $\Phi \in (I' : \gotq \cdot  \freemodule[A']{S})$, so that
   $\Phi \cdot \gotq \subset I'$.
   Then, $ \xi(\Phi\cdot \gotq)= \xi(\Phi) \cdot \gotq\cdot \freemodule{S}  \subset I$, hence $\xi(\Phi) \in I$,
   and by definition of $I'$, we must have $\Phi \in I'$ as claimed.
\end{prf}

We also need another lemma about saturated ideals.

\begin{lemma}\label{lem_socle_fibre_ideal_of_saturated_is_saturated}
   Suppose $A$, $S$, $\freemodule{S}$, $\gots$ are as above.
   Let $\gotq\subset S$ be an ideal.
   If $I \subset \freemodule{S}$
   is an ideal such that $I = (I: \gotq \cdot \freemodule{S})$,
   then also $I_{\gots}= (I_{\gots} \colon \gotq)$.
\end{lemma}
\begin{proof}
   Suppose $\Phi \in (I_{\gots} \colon \gotq)$, 
   then $\Phi \cdot \gotq \subset I_{\gots}$. 
   Thus, $ \Phi \cdot \gots \cdot \gotq  \subset I$ and therefore $ \Phi \cdot \gots \subset I$ and $\Phi \in I_{\gots}$, showing that $I_{\gots}$ is saturated with respect to $\gotq$, as claimed.
\end{proof}

In the proposition below and throughout the article we will denote by $\gotn \subset A\otimes_{\kk}\Sym V^*$ the homogeneous ideal generated by $V^*$.
Moreover, for an ideal $I$ in $\Sym V^*$ or in  $A\otimes_{\kk}\Sym V^*$ we will write $I^{\sat}$ to be the saturation with respect to $\gotn$, that is
$I^{\sat} := (I : \gotn^{\infty})$.

\begin{proposition}
   \label{prop_classical_Macaulay_and_Gotzmann}
   Suppose $I\subset \Sym V^*$ is a homogeneous ideal
   with respect to the standard grading and
   let $S=\Sym V^*/I$  be the graded quotient algebra.
   Suppose $i$ and $r$ are two non-negative integers.
   Then:
   \begin{enumerate}
    \item
      \label{item_classical_Macaulay}
      (Macaulay's bound)
      If $i\ge r$ and $\dim S_i\leqslant r$,
       then
       \begin{itemize}
        \item $\dim S_{k+1} \le \dim S_k$, and for all $k\geqslant i$, and
        \item  $\dim S_k \geqslant \dim S_i$ for all $r-1\leqslant k \leqslant i$.
       \end{itemize}
    \item
      \label{item_classical_Gotzmann}
      If $i\geqslant r$, $\dim S_i=r$ and $I$ is generated in degrees at most $i$,
      then for all $k\geqslant i$ we have:
      \begin{itemize}
       \item (Gotzmann's persistence) $\dim S_k = r$,
       \item (saturation) $I^{\sat}_{k} = I_k$,
       \item (weak Lefschetz for saturation)
         If $I = I^{\sat}$, then multiplication by a general linear form $\alpha \in V^*$ determines an invertible $\kk$-linear map
       $S_{l}\to S_{l+1}$ for any $l\ge r-1$,
       \item (weak Lefschetz property) multiplication by a general linear form $\alpha \in V^*$ determines an invertible $\kk$-linear map
       $S_{k}\to S_{k+1}$.
      \end{itemize}
   \end{enumerate}
\end{proposition}
\noprf
For more general statements and proofs of Macaulay's bound and Gotzmann's persistence, see for instance \cite[Thms~3.3, 3.8]{green_generic_initial_ideals}.
The claim about saturation follows from applying
Macaulay's bound to $I^{\sat}$ and using $I\subset I^{\sat}$ with equality in very large degrees.
The injectivity in weak Lefschetz properties follows from the saturation property and finiteness of primary decomposition. The surjectivity follows from the dimension count and Macaulay's bound.

We have the following generalizations and consequences
 of Macaulay's bound and Gotzmann's Persitence theorems.

\begin{proposition}\label{prop_relative_Macaulay_and_Gotzmann}
   Fix a positive integer $r$ and two degrees $j > i \geqslant r$.
   Suppose $I\subset \freemodule{S}$ is a homogeneous ideal such that
     both $(\freemodule{S}/I)_i$ and $(\freemodule{S}/I)_j$ are flat (hence free) $A$-modules of rank $r$, and let $J = (I)_{\leqslant i}$.
   Then,
  \begin{enumerate}
   \item \label{item_Macaulay_for_I}
   $(\freemodule{S}/I)_k$ is flat of rank $r$ for all $i < k < j$ and $I$ has no minimal generators in degrees $\fromto{i+1, i+2}{j}$.
   \item \label{item_Macaulay_and_Gotzmann_for_J}
    $(\freemodule{S}/J)_k = (\freemodule{S}/J^{\sat})_k$ and they are both flat of rank $r$ in all degrees $k\geqslant i$.
   \item \label{item_Macaulay_and_Gotzmann_for_Jsat}
   $(\freemodule{S}/J^{\sat})_k$ is flat of rank $r$ in all degrees $k\geqslant r-1$.
   \end{enumerate}
\end{proposition}

\begin{prf}
   Both algebras $S/I_{\kk}$ and $S/I_{\gots}$ satisfy
   \[
      \dim_{\kk}(S/I_{\bullet})_k
      \begin{cases}
          \geqslant r & \text{for } r-1 \leqslant k \leqslant j \\
          \leqslant r & \text{for } i \leqslant k
      \end{cases}\quad \text{ with } \bullet \in \set{\kk, \gots}
   \]
   by the Macaulay bound in Proposition~\ref{prop_classical_Macaulay_and_Gotzmann}.
   In particular, for $i\leqslant k \leqslant j$  we have
   $\dim_{\kk}(S/I_{\kk})_k = \dim_{\kk}(S/I_{\gots})_k = r$.
   Thus, the flatness of $I$ in degrees between $i$ and $j$ 
     follows from Lemma~\ref{lem_criterion_for_free_submodule}.

   Similarly, both algebras $S/J_{\kk}$ and $S/J_{\gots}$ satisfy
   \[
      \dim_{\kk}(S/J_{\bullet})_k
      \begin{cases}
          \geqslant r & \text{for } r-1 \leqslant k \leqslant i-1 \\
          = r & \text{for } i \leqslant k 
      \end{cases}.
   \]
   by Macaulay bound and Gotzmann's persistence of Proposition~\ref{prop_classical_Macaulay_and_Gotzmann}.
   As above, this proves the flatness of $J$
   in degrees at least $i$.
   In degrees $k$ between $i$ and $j$ both $I_k$ and $J_k$ are $A$-modules of the same $\kk$-dimension 
      and with $J \subset I$.
   Thus, $I_k=J_k$ and therefore $I$ has no generators 
      in these degrees, concluding the proof of \ref{item_Macaulay_for_I}.

   To show the remaining claims, by Lemma~\ref{lem_reduction_to_dual_numbers} it is enough to assume that $A = \kk[t]/(t^2)$.
   Since $J\subset J^{\sat}$ we must have
   \[
      \dim_{\kk}(S/(J^{\sat})_{\bullet})_k \leqslant \dim_{\kk}(S/J_{\bullet})_k 
   \]
   By Lemma~\ref{lem_socle_fibre_ideal_of_saturated_is_saturated}
   the ideal $(J^{\sat})_{\gots}$ is saturated. As $
      \dim_{\kk}(S/(J^{\sat})_{\gots})_k =
      \dim_{\kk}(S/(J)_{\gots})_k =r$ for $k\gg 0$ and the Hilbert function of a saturated ideal is non-decreasing, we obtain
   \[
      \dim_{\kk}(S/(J^{\sat})_{\gots})_k =r \text{ for } k \geqslant r-1.
   \]
   We claim that $((J^{\sat})_{\gots})_{k} = ((J^{\sat})_{\kk})_{k}$ for all $k \ge r-1$.
   Indeed, for some $k_0$ (initially equal to $i$) this is true for all $k\ge k_0$, and we always have $((J^{\sat})_{\gots})_{k} \supset ((J^{\sat})_{\kk})_{k}$.
   We argue by a downward induction on $k_0$.
   If $k_0=r-1$, then we are done.
   If $k_0\geqslant r$,
   then pick a general $\alpha\in V^*$.
   Suppose $\Phi \in S_{k_0-1}$ is such that
   $\alpha \Phi \in ((J^{\sat})_{\kk})_{k_0}$.
   Thus,
   $\alpha \Phi + t\xi \in (J^{\sat})_{k_0}$
      for some $\xi \in S_{k_0}$.
   The multiplication by $\alpha$ determines
     an isomorphism
     $\left(S/(J^{\sat})_{\gots}\right)_{k_0-1} \stackrel{\cdot \alpha}{\to} \left(S/(J^{\sat})_{\gots}\right)_{k_0}$ by the weak Lefschetz property in  Proposition~\ref{prop_classical_Macaulay_and_Gotzmann}.
   Thus, let $\Psi \in S_{k_0-1}$
     be such that $\alpha \Psi = \xi \modulo ((J^{\sat})_{\gots})_{k_0}$.
   Therefore, we have
   $
    \alpha \Phi + t\alpha \Psi \in  (J^{\sat})_{k_0}
   $
   and by generality of $\alpha$  and saturatedness of $J^{\sat}$ we must have
   $\Phi + t \Psi \in  (J^{\sat})_{k_0-1}$.
   Thus, $\Phi \in ((J^{\sat})_{\kk})_{k_0-1}$,
   and we may proceed by induction, proving the claim.

\end{prf}

\section{Families of linear subspaces and linear span}\label{sec_linear_spaces_and_linear_spans}

Fix a base scheme $Q$ over $\kk$.
In this section, our main interest is in the case where $Q$ has enough $\kk$-points in the sense of Definition~\ref{def_enough_k_points}, but in \S\ref{sec_families_of_linear_spaces_examples} this assumption is not required.
Let $W$ be a fixed vector space over $\kk$ of dimension $N+1$.
As in \S\ref{sec_finite_subschemes_as_tensors}
we denote by $\affinespaceof{W} =\AAA^{N+1}_{\kk}= \Spec \Sym W^*$ and by $\PP_{\kk}W = \PP^N=\Proj \Sym W^*$.
Consider
$\freemodule[Q]{\ccS}:=\ccO_Q[\fromto{\alpha_0}{\alpha_N}] = \ccO_Q \otimes_{\kk} \Sym W^*$,  the trivial sheaf of graded $\ccO_Q$-algebras with $\deg(\alpha_i)=1$ for all $i$, and the coefficients $\ccO_Q$ contained in degree $0$.
Here $\setfromto{\alpha_0}{\alpha_N}$ is the basis of $W^*$.
Thus  $\affinespaceof[Q]{W} = Q\times \AAA^{N+1}_{\kk}
= \SheafySpec{Q} (\freemodule[Q]{\ccS})$.
In this section, we analyze in detail the properties of the following notions:

\begin{definition}[linear ideal sheaf, families of linear spaces]
   A \emph{linear ideal sheaf} in $\freemodule[Q]\ccS$
   is a homogeneous ideal sheaf $\ccI\subset \freemodule[Q]{\ccS}$
   which is generated by its degree $0$ and $1$ parts, $\ccI_0 \subset \ccO_Q$ and $\ccI_1\subset \ccO_Q \otimes_{\kk} W^*$.
   A \emph{family of linear subspaces} of $W$ parameterized by $Q$
   is a subscheme of $\affinespaceof[Q]{W}= \SheafySpec{Q} (\freemodule[Q]{\ccS})$ defined by a linear ideal sheaf.
   A \emph{family of projective linear subspaces} of $\PP_{\kk} W$ parameterized by $Q$
   is a subscheme $Z(\ccI)$ of
   $\PP_Q W= \SheafyProj{Q} (\freemodule[Q]{\ccS})$ obtained as  the vanishing scheme of a linear ideal sheaf $\ccI\subset \freemodule[Q]{\ccS}$.
\end{definition}

The distinction between affine and projective cases above is a little subtle. Any family of linear spaces in $W$ gives rise to a family of projective linear spaces in $\PP_Q W$, by using the same linear ideal sheaf.
However, conversely, at this point we do not know if the ``affine cone'' in $\affinespaceof[Q]{W}$
over any family of projective linear subspaces of $\PP_{\kk} W$ is a family of linear spaces in $W$. See Conjecture~\ref{conj_linear_ideals_are_saturated} and the following comments for more details.

\subsection{Examples, properties, and saturatedness conjecture}
\label{sec_families_of_linear_spaces_examples}
If $Q=\Spec \kk $, then a family of linear spaces over $\Spec \kk$ is either an empty subscheme of $\affinespaceof{W}$ or a linear subspace of $\affinespaceof{W}$ (seen as a subscheme), which uniquely corresponds to a $\kk$-linear subspace of $W$.
(The empty scheme is allowed for consistency with further definitions.)
Note that already here both the linear ideals $(1)$ and $(W^*)$ give rise to the same family of projective linear subspaces in $\PP_{\kk} W$, namely to the empty scheme, while their affine counterparts are different.
More generally, any vector subbundle $\ccE$ of the trivial vector bundle
  $\ccO_Q\otimes_{\kk} W$
  can be geometrically interpreted as a family of linear subspaces of $W$ over $Q$
  with $\ccI_0=0$ and 
\[  
  \ccI_1 = \ccE^{\perp} \subset \ccO_Q \otimes_{\kk} \linspan{\fromto{\alpha_0}{\alpha_N}}= \ccO_Q\otimes_{\kk}W^*.
\]
The (naive) projectivization $\PP(\ccE)$ of this bundle is the corresponding family
 of projective linear spaces in $\PP_{\kk}W$.
These two cases are \emph{flat} cases: the sheaves $\ccI$ and $\ccS/\ccI$ are flat $\ccO_{Q}$-modules.
As we will see, in this article, the primary technical issue will be to deal with non flat cases too.

\begin{example}[Blow up]\label{ex_blow_up_as_family_of_linear_spaces}
   Let $Q=\AAA_{\kk}^2=\Spec \kk[s,t]$, and let $\ccI=(t\alpha_0-s\alpha_1)\subset \kk[s,t, \alpha_0,\alpha_1]$. Then $\ccI$ is a linear ideal sheaf and the corresponding family
   of projective linear subspaces in $\PP_{\kk}^1$ is the the blow of $Q$ in the origin.
   The general fibre of the family is a point, while the special fibre (over the origin) is $1$-dimensional.
   The affine analogue would have $1$ or $2$ dimensional fibres instead.
\end{example}

In  Example \ref{ex_blow_up_as_family_of_linear_spaces}, we have a nice stratification of $Q$ by the dimension of fibres.
Such stratification however makes little sense if the base scheme $Q$ is not reduced.

\begin{example}
   An interesting variant of Example~\ref{ex_blow_up_as_family_of_linear_spaces} is when we replace the base $Q$ with $\Spec \kk[s,t]/(s^2,t^2)$, and consider the linear ideal generated by the same polynomial, $\ccI=(t\alpha_0-s\alpha_1)\subset \kk[s,t]/(s^2,t^2) \otimes_{\kk} \kk[\alpha_0,\alpha_1]$.
   Then the primary decomposition shows that 
   \[
      \ccI = \left(t\alpha_0-s\alpha_1, st \right)\cap \left( \alpha_0, \alpha_1\right).
   \]
   Both ideals in the intersection are linear ideals. 
   The first ideal has degree $0$ generator $st$, 
     and its properties are very similar to properties of the blow up ideal, 
     but over the base $\Spec \kk[s,t]/(s^2,st,t^2)$.
   The second ideal is a flat ideal sheaf, describing the zero vector subbundle 
     of $Q\times \kk^2$.
   Thus, one can claim that the whole stratum of $Q$ 
      with positive dimensional ``fibres'' is contained in the proper subscheme $\Spec \kk[s,t]/(s^2,st,t^2)$.
   However, this family arises as a pullback of the (affine version of) family in the previous example,
      where there was no dimension $0$ part at all.
\end{example}

\begin{proposition}\label{prop_linear_sheaf_ideals_locally}
   Suppose $Q$ is a base scheme and $\ccI \subset \freemodule[Q]\ccS$
      be a homogeneous ideal sheaf.
   Then the following conditions are equivalent:
   \begin{enumerate}
     \item \label{item_sheaf_is_linear}
     $\ccI$ is a linear ideal sheaf,
     \item \label{item_affines}
     for any open affine subset $U\subset Q$,
     the $\ccI(U) \subset 
        \ccO_Q(U)[\fromto{\alpha_0}{\alpha_N}]$ 
        is generated in degrees $0$ and $1$.
     \item \label{item_locallyholds}
     for any $q\in Q$ the localisation
       $\ccI_q \subset \ccO_{Q,q}[\fromto{\alpha_0}{\alpha_N}]$ 
       is generated in degrees $0$ and $1$.
   \end{enumerate}
\end{proposition}
\begin{proof}
These equivalences are consequences of the triviality of the sheaf of algebras $\freemodule[Q]{\ccS}$ over $Q$, and also of its grading.

In detail,
 \ref{item_sheaf_is_linear}
 means that there exists an open cover of $Q$ such that for each open set $U$ in this
  cover the ideal $\ccI(U)$ is generated in degrees $0$ and $1$.
 In particular, the implication
 \ref{item_affines}$\Longrightarrow$ \ref{item_sheaf_is_linear}
 is clear.
Also \ref{item_sheaf_is_linear}$\iff$\ref{item_locallyholds}
holds since each grading of $\ccI$ is a coherent sheaf and surjectivity of sheaf map $\ccI_1\otimes_{\ccO_{Q}} \ccS_{d-1} =  \ccI_1\otimes_{\kk} S^{d-1}W^* \to \ccI_d$  is equivalent to surjectivity of its localisations at each point  \cite[\href{https://stacks.math.columbia.edu/tag/01AG}{Lem.~01AG}]{stacks_project}.

The final implication is a ``\cite{hartshorne}-style exercise''. Suppose \ref{item_sheaf_is_linear}
holds, and that $U= \Spec A \subset Q$ is an open affine subset.
It is covered by basic open affine subsets $\set{U_f}$, with $f\in A$ and each $U_f = \Spec A_f = \set{f\ne 0} \subset U$,  such that $\ccI(U_f) $ is generated in degrees $0$ and $1$.
Fix $\Phi\in \ccI(U)$.
For any $f$ (after clearing the denominators) we can express
   $f^{k_f} \cdot \Phi \in \ccI(U)_{\le 1} \cdot \freemodule[Q]\ccS(U)$
   for some non-negative integer $k_f$.
Since $U$ is quasicompact by
 \cite[\href{https://stacks.math.columbia.edu/tag/00E8}{Lem.~00E8}]{stacks_project},
there is a finite combination of the $f^{k_f}$'s 
  that gives a unit in $A$, and thus 
  $ \Phi \in \ccI(U)_{\le 1} \cdot \freemodule[Q]\ccS(U)$,
  proving \ref{item_affines}.
\end{proof}

\begin{proposition}
Suppose $Q$ and $Q'$ are two base schemes, $G\colon Q' \to Q$ is a morphism of $\kk$-schemes and $\ccI$ and $\ccJ$ are two linear ideal sheaves in $\freemodule[Q]{\ccS}$. Then:
  \begin{enumerate}
    \item $\ccI +\ccJ$ is a linear ideal sheaf in $\freemodule[Q]{\ccS}$,
    \item $G^*\ccI$ is a linear ideal sheaf in $\freemodule[Q']{\ccS} = \ccO_{Q'} \otimes_{\kk} \Sym W^*$.
  \end{enumerate}
\end{proposition}

\begin{prf}
  Both items follow from the equivalences of Proposition~\ref{prop_linear_sheaf_ideals_locally}
\end{prf}

\begin{conjecture}[Linear ideals are saturated]
\label{conj_linear_ideals_are_saturated}
   Suppose $\ccI \subset \freemodule[Q]{\ccS}$ is a linear ideal sheaf.
   Let $\gotn:=(\fromto{\alpha_0}{\alpha_N}) \subset \freemodule[Q]{\ccS}$ be the ideal sheaf $\freemodule[Q]{\ccS}_{\geqslant 1}$.
   Then, $(\ccI \colon \gotn^{\infty})$ is a linear ideal sheaf too.
\end{conjecture}

The conjecture is a standard statement if the base $Q = \Spec \kk$: a homogeneous ideal generated in degree $1$ is always either saturated or its saturation is $(1)$.
It seems that the natural generalization to any base should also be true and is a surprisingly easy-to-state, but non-trivial open problem. For the purpose of applications we have in mind, it would be sufficient to prove the claim for a finite local base $Q= \Spec A$ with $A = \kk[\fromto{t_1}{t_m}]/J$ with
$(\fromto{t_1}{t_m})^{s} \subset J$ for some $s$. The following comments might be relevant to someone trying to resolve the conjecture (over finite local $Q$ as above):
\begin{itemize}
 \item The conjecture is true in the flat case (if $\ccI$ is a flat sheaf over $Q$), with a similar proof as in the case $Q=\Spec \kk$.
 \item By Nakayama's Lemma, 
        it is sufficient to prove the conjecture for ideal sheaf $\ccI \subset (\fromto{t_1}{t_m}) \cdot \freemodule[Q]{\ccS}$.
 \item The conjecture is true if $A = \kk[t]/(t^s)$, by induction on $s$.
 \item The conjecture is true if $N=0$.
\end{itemize}

\subsection{\texorpdfstring{\Familiar's}{Familiars} and linear span of a family}

Before dealing with a relative setting 
we fix some notation about linear span and embedded finite schemes.

\begin{definition}[linear span and linearly independent subchemes]
\label{def_linear_span}
Suppose $R \subset \PP_{\kk} W$ is a closed subscheme and $S=\freemodule[\kk]{S}=\Sym W^*$
  is the homogeneous coordinate ring of $\PP_{\kk} W$ with $I(R)\subset S$ the homogeneous ideal of $R$.
Let $\affinespaceof{W} = \Spec S$ be the scheme (affine space) corresponding to $W$.
Then, we define
\begin{itemize}
 \item \emph{the affine linear span} of $R$ to be the affine subspace $\affinespaceof{W'} \subset \affinespaceof{W}$, where $W'\subset W$  is perpendicular to $I(R)_{1} \subset W^*$, and
  \item
  \emph{the (projective) linear span} of $R$ to be the closed subscheme (isomorphic to a projective space)
  $\linspan{R} = \PP_{\kk}^{k-1} \subset \PP_{\kk} W$ defined by $I(R)_{\leqslant 1}$.
\end{itemize}
We also say that $R$ is \emph{(linearly) independent} if $R$ is a finite subscheme  and its degree is equal to the $\kk$-dimension of the affine linear span.
Such $R$ is also sometimes called linearly non-degenerate,
 or $R$ imposes independent conditions on forms of degree $1$.
For brevity we will simply call such schemes \emph{independent}.
\end{definition}

The purpose of this subsection is to define and discuss how analogous notions behave in families.
The following setting will be relevant to  the rest of
  this section.
\begin{setting}[family of subschemes of $X$]
\label{set_family_of_subschemes}
 Consider a projective scheme $X$
    with a fixed closed embedding
    $\iota \colon X \hookrightarrow \PP W$.
    Unless potentially confusing,
      $\iota$ will be skipped from the notation and
      we will just consider
      $X\subset \PP_{\kk} W$ as a subscheme.
Moreover, $Q$ is a base scheme which has enough $\kk$-points.
Suppose $\ccR \subset  Q\times X \subset \PP_{Q} W = Q\times_{\kk} \PP_{\kk} W $ is a subscheme.
Denote by $\phi \colon \ccR \to Q$ the natural projection and we think of $\phi$ as a \emph{family of subschemes of $X$} parametrised by $Q$.
\end{setting}
Note that at this point, 
we do not assume that $\phi$ in Setting~\ref{set_family_of_subschemes} is flat or surjective.
Instead, in most of the situations we will assume that $\phi$ is finite. Moreover, the following notion of linearly independent family is  going to be frequently used.

\begin{definition}[linearly independent family]
   In Setting~\ref{set_family_of_subschemes},
   we say that $\phi$ is \emph{(linearly) independent}
    if the morphism $\phi$ is finite, and
   for all $\kk$-points $q\in Q$,
     the scheme theoretic fibre $\phi^{-1} (q)$ is an independent subscheme of $ \set{q} \times \PP^N_{\kk}  \simeq \PP^N_{\kk}$ (in the sense of Definition~\ref{def_linear_span}).
\end{definition}

\begin{definition}[\familiar's]
In Setting~\ref{set_family_of_subschemes},
suppose in addition that $\phi\colon \ccR \to Q$
  is finite.
We say that $\phi$ is
  a \emph{family of finite subschemes of $X$ of degree at most $r$}
  if for every closed point $q\in Q$
  the fiber $\phi^{-1}(q)$
is a scheme of length at most $r$.
Furthermore, we say that $\phi$ is a  \emph{\familiar{}} (more precisely, a \familiar{} in $X$ over $Q$)
if in addition each fiber $\phi^{-1}(q)$ is a scheme of length equal to $r$.
A \familiar{} is \emph{flat} (respectively,
\emph{independent}) if in addition $\phi$ is flat (respectively, independent).
\end{definition}

\begin{SCfigure}
 \centering
 \includegraphics[width=0.47\textwidth]{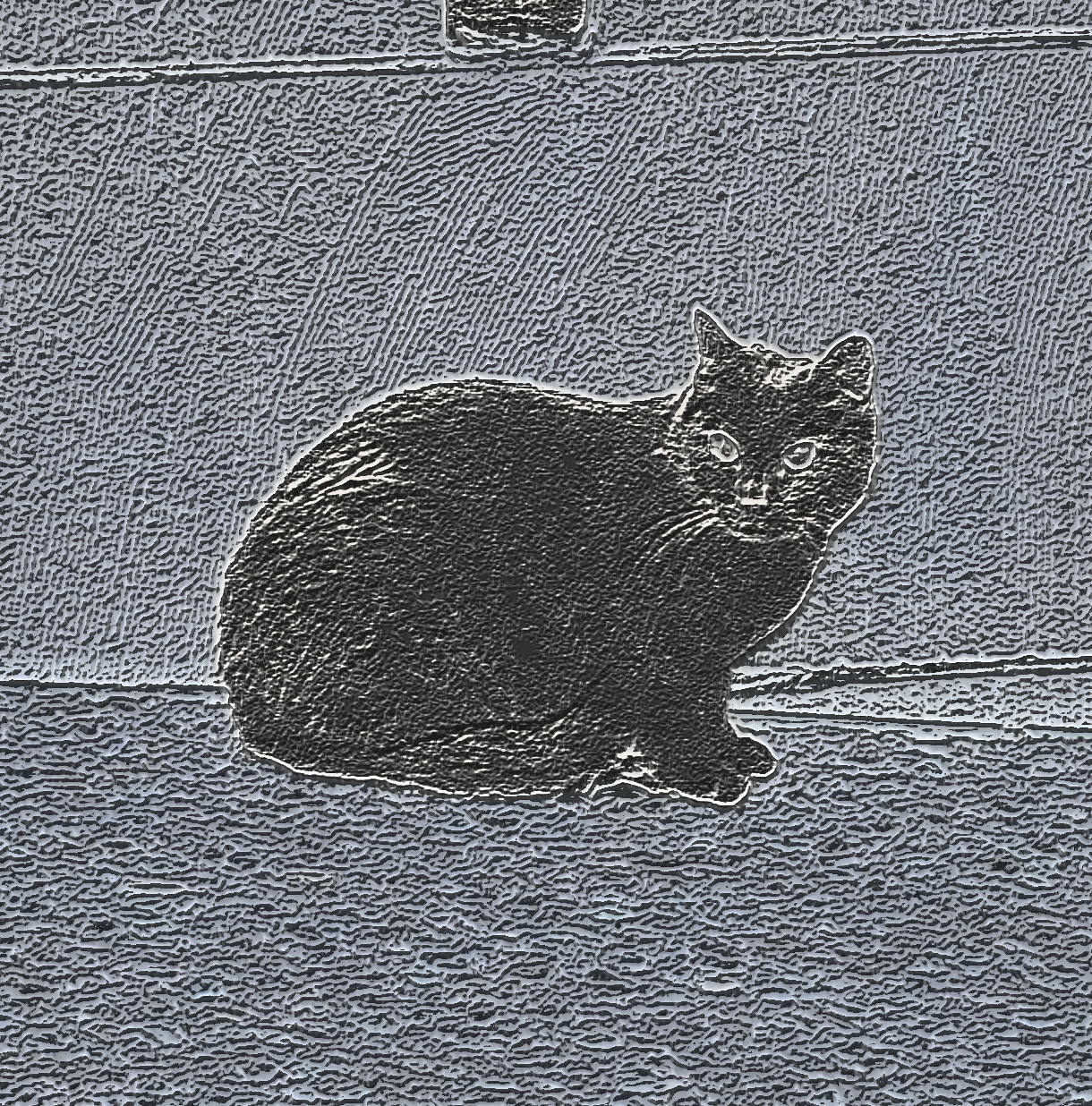}
 \caption{
  In European folklore, a familiar was a demon that assumed an animal shape, such as a black cat, dog, or toad, and that would assist witches and cunning folk in their practice of magic.
 In this article, \familiar{} is an abbreviation for ``family of finite subschemes of $X$ tightly of degree at most $r$'', where ``tightly'' refers to the assumption that the family is not of lower degree than $r$ at any point. \Familiar's will assist us in the investigation of the scheme structure of catalecticant minors. This might also sound like witchcraft to tribes not accustomed to the magical beauty of algebraic geometry.
}
\end{SCfigure}

%


Couple of examples of \familiar{}'s for low values of $r$ are: for $r=0$, a \familiar[0]{} is an empty subscheme $Q\times X \supset \emptyset \to Q$ and it is always flat and linearly independent;
for $r=1$, a \familiar[1]{} is a subscheme $\ccR\subset Q$ with the same support as $Q$, and an embedding into $Q\times X$ and it is always linearly independent.
Also for $r=2$, any \familiar[2]{} is necessarily linearly independent.
We remark that a \familiar{} over a reduced base scheme $Q$ is automatically flat.

\begin{example}
   Let $Q= \Spec\kk[t]/(t^2)$ and set $\ccR = \Spec\kk[t, s]/(s^2, st,t^2)$ with a natural morphism $\phi\colon \ccR\to Q$ induced by $\kk[t]/(t^2) \to \kk[t, s]/(s^2, st,t^2)$.
   Then $\phi$ is of degree at most $3$ (or at most $r$ for any $r\geqslant 2$).
   Moreover, $\phi$ is a \familiar[2], but $\phi$ is not a flat \familiar[2].
   Instead,
   $\Spec\kk[t, s]/(s^2, t^2) \to  \Spec\kk[t]/(t^2)$
   is a flat \familiar[2].
\end{example}

\begin{example}\label{ex_finite_family_not_tight_for_any_r}
   Let $Q = \AAA_{\kk}^1$ and $\ccR=\Spec \kk$ (single reduced point).
   Then any embedding $\ccR\to Q$ is of degree at most $1$,
     but it is not a \familiar{} for any~$r$.
\end{example}

\begin{proposition}\label{prop_base_change_of_familiars}
   Suppose  $\xi\colon Q'\to Q$
      is a morphism of schemes which have enough $\kk$-points and
      $\phi:\ccR \to Q$ is a \familiar{} over $Q$.
      Then, the induced morphism
       $\phi'\colon \ccR \times_{Q} Q' \to Q'$
       is a \familiar{} over $Q'$.
\end{proposition}

\begin{prf}
     The map $\phi'$ is finite, as it is a base change of a finite map $\phi$.
     Each fibre of a $\kk$-point $q'\in Q'$ is equal to a fibre of
       $\xi(q')\in Q$ and hence the degree of $(\phi')^{-1}(q')$ is equal to
     the degree of $\phi^{-1}(\xi(q'))$, that is $r$.
     These show that $\phi'$ is indeed a \familiar{} over $Q'$.
\end{prf}

\begin{definition}
   In Setting~\ref{set_family_of_subschemes}
   let
$\ccI\subset \ccO_Q\otimes_{\kk}\Sym W^*$
   be the homogeneous ideal sheaf of
    $\ccR \subset \PP_Q W$.
   The \emph{affine (equation) linear span} of the family
      $\phi\colon\ccR \to Q$ is
      the family of linear subspaces
      of $\affinespaceof{W}$
      defined by the linear ideal sheaf $(\ccI_{1})$.
   We also define \emph{(projective equation) linear span} to be the family of projective linear subspaces
   $\langle\phi\rangle \colon \langle\ccR\rangle \to Q$ of $\PP_{\kk}W$
      defined by the linear ideal sheaf $(\ccI_{\leqslant 1})$.
\end{definition}

The equation linear span behaves well
 and in a controlled way under the assumption that
 $\phi$ is linearly independent.
Without this assumption, wild things may happen,
  compare Lemma~\ref{lem_dimension_of_linear_span} 
to Example~\ref{ex_bad_linear_span}.

\begin{example}\label{ex_bad_linear_span}
   Let $A=\kk[s,t]/(s^2, st, t^2)$, $Q=\Spec A$ and let $\ccR\subset \PP_{Q}^3$ be the subscheme defined by the following ideal in $A[\alpha, \beta, \gamma, \delta]$:
\begin{align*}
\ccI &:= (\alpha-\beta, \gamma,\delta) \cap (\alpha, \gamma, \delta) \cap (\beta, \gamma- s\alpha, \delta - t \alpha) \\
    &\phantom{:}= (t\delta, s\delta, t\gamma, s\gamma, \\
    &\phantom{:=}\delta^2, \gamma\delta, \beta\delta, \gamma^2, \beta\gamma, t\alpha^2-t\alpha\beta-\alpha\delta, s\alpha^2-s\alpha\beta-\alpha\gamma, \alpha^2\beta-\alpha\beta^2)
 \end{align*}   
   One can think of $\ccR$ as of a family of three points in $\PP^3$, over $Q$: two constant points $(1:1:0:0)$ and $(0:1:0:0)$ and one ``moving'' infinitesimally, $(1:0:s:t)$.
   Then $\phi\colon \ccR \to Q$ is a flat \familiar[3], but it is not linearly independent (for $t=s=0$, the third point is in the span of the first two points).
   Then, counter-intuitively, $\dim_\kk \langle\ccR\rangle = 3$ (or the dimension of the affine linear span is $4$)
   and this illustrates that equation linear span 
   does not behave well under the base change $\Spec \kk \to Q$, since
   $\dim_{\kk}\linspan{\ccR_{\kk}} =1$, while   $\dim_{\kk}\linspan{\ccR}_{\kk} = 3$ (respectively, $2$ and $4$ for the affine linear span).
\end{example}
Due to pathologies such as the one described in Example~\ref{ex_bad_linear_span} we expect that different notions of the linear span of families might need to be considered.
This is why we add the adjective ``equation'' to ``equation linear span'' to leave space for some other notions, depending on the context. However, in this article, we will primarily consider the spans of linearly independent families and the only linear span of a family we consider is the equation linear span.

   The same conclusion as in Example~\ref{ex_bad_linear_span} is true (although the explicit defining equations may be different) if instead of the base scheme $Q$ as above we use some other scheme, such as spectrum of $\kk[t,s]$, $\kk[[t,s]]$, $\kk[t,s]/(t^3, s^3)$ etc. In particular, this pathology is not a consequence of the singularity type of $Q$.
   In Example~ \ref{ex_double_bad_linear_span}, 
     we will use a variant of the above example where $A=\kk[t,s]/(t^3, s^3)$.

We conclude this subsection with two observarions about comparing the linear span of a base change to base change of linear span: one inclusion always holds and if the base change is flat (for instance, an open immersion or an embedding $\Spec \ccO_{Q, q} \to Q$ of the local scheme near a closed point $q\in Q$), then the equality holds.

\begin{lemma}\label{lem_inclusion_on_linear_spans}
   Suppose as in Proposition~\ref{prop_base_change_of_familiars}
   that $\xi\colon Q'\to Q$ is a
   morphism of schemes which have enough $\kk$-points and
      $\phi$ is a \familiar{} over $Q$,
   and $\phi'\colon \ccR' \to Q'$
       with $\ccR \times_{Q} Q'$ is the induced \familiar{} over $Q'$.
   Then, $\xi\times \id_{\PP_\kk W} (\linspan{\ccR'}) \subset \linspan{\ccR}$.
\end{lemma}

\begin{prf}
   It is straightforward as the pullback of any degree $0$ or $1$ equation of $\ccR$ is a degree $0$ or $1$  equation of $\ccR'$.
\end{prf}

\begin{lemma}[flat base change]
\label{lem_flat_base_change_of_linear_spans}
   Suppose $\xi$, $Q$, $Q'$, $\phi$, $\phi'$ are as in Lem\-ma~\ref{lem_inclusion_on_linear_spans}.
   If $\xi\colon Q'\to Q$ is in addition
   a flat morphism, then
   $(\xi\times \id_{\PP_\kk W})^* (\linspan{\phi}) = \linspan{\phi'}$.
\end{lemma}


\begin{prf}
   Let $\ccI\subset \ccO_{\PP_Q W}$ be the ideal sheaf of $\ccR$
   and set $\ccI(i):=\ccI \otimes_{\ccO_{\PP_Q W}}\ccO_{\PP_Q W}(i)$
   for each integer $i$.
   Then, the linear span  $\linspan{\ccR}$ 
     is defined by the coherent sheaf
   $\phi_*\ccI\oplus\phi_*\ccI(1) \subset
   \ccO_Q\otimes_{\kk}(\kk\oplus W^*)$, while its pullback by $\xi\times \id_{\PP_{\kk} W}$ is defined by
   $\xi^*\phi_*\ccI\oplus\xi^*\phi_*\ccI(1) \subset
   \ccO_{Q'}\otimes_{\kk}(\kk\oplus W^*)$.
   Similarly,
   $\linspan{\ccR'}$ is defined by the coherent sheaf
   $\phi'_* (\xi\times \id_{\PP_{\kk} W})^*\ccI\oplus\phi'_*(\xi\times \id_{\PP_{\kk} W})^*\ccI(1) \subset
   \ccO_{Q'}\otimes_{\kk}(\kk\oplus W^*)$.
   Thus, the claim follows from the flat base change 
    \cite[\href{https://stacks.math.columbia.edu/tag/02KH}{Lem.~02KH}]{stacks_project}
  applied to $\ccI(i)$ for $i=0$ and $i=1$.
\end{prf}

In Corollary~\ref{cor_base_change_of_linear_span_of_independent_familiar}, 
we further observe that such natural base change property also holds for linearly independent \familiar's, without the flatness assumption on the morphism $\xi$.

\subsection{Linear spans of independent \texorpdfstring{\familiar's}{familiars} and minors}

Fix a base scheme $Q$
  which has enough of $\kk$-points.
In this section, we will often consider
  the special case $Q=\Spec A$,
  where $(A,\mathfrak{m}, \kk)$ is a local ring with residue field
  $\kk \subset A$.
The statements here are either valid only for such local case or proven using restrictions to such cases.

\begin{remark}
   In the local  setting as above, if $\phi\colon \ccR\to Q$
     with $\ccR\subset \PP_A^N$,
     then $\phi$ is a \familiar{}
     if and only if
     $\phi$ is finite of degree at most $r$,
     but not of degree at most $r-1$.
     In other words,
     the situation as in  Example~\ref{ex_finite_family_not_tight_for_any_r}
     is not possible here.
\end{remark}

\begin{remark}
   All the claims in this subsection also hold 
   if instead of projective linear span we consistently 
    consider the affine linear span 
    (and shift the dimension accordingly). 
    The proofs of the affine case are identical. Since our main concern in this article is about the projective linear span, we skip the affine case.
\end{remark}

\begin{theorem}\label{thm_Nakayama_for_linear_spans}
Suppose $(A,\mathfrak{m}, \kk)$ is a local ring with residue field
  $\kk \subset A$ and
let $\ccR\subset \PP_A^{N}$ be a linearly independent familiar over $Q=\Spec A$.
Then
\begin{enumerate}
  \item
  \label{item_linear_forms_surjective_for_independent_familiar}
  the restriction map
  $\HHH^0(\PP^N_A,\ccO_{\PP^N_A}(1)) \to \HHH^0(\ccR,\ccO_{\ccR}(1)) \simeq\HHH^0(\ccR,\ccO_{\ccR})$ is surjective, and
  \item
  \label{item_extending_linear_span_to_flat_family}
  there exists an $A$-linear embedding (a flat family of linear subspaces) $\PP_A^{r-1}\subset \PP_A^{N}$ such that $\ccR\subset \PP_A^{r-1}$.
\end{enumerate}

\end{theorem}

\begin{proof}
We focus on proving Item~\ref{item_extending_linear_span_to_flat_family}, while Item~\ref{item_linear_forms_surjective_for_independent_familiar} will pop out along the way.

Denote by $\ccR_{\kk} \subset \ccR$ the fibre over the unique closed point $\Spec \kk\in Q$.
Let $\ccI_{\ccR}\subset \ccO_{\PP_A^N}$ and $\ccI_{\ccR_{\kk}} \subset \ccO_{\PP_{\kk}^N}$ be the sheaves of  ideals of $\ccR$ and of $\ccR_{\kk}$.
We will first show that the natural map
$$g:\HHH^0(\PP_A^N, \mathcal{I}_{\ccR}(1))\longrightarrow  \HHH^0(\PP_{\kk}^N, \mathcal{I}_{\ccR_{\kk}}(1))$$
is surjective (note that Example~\ref{ex_bad_linear_span} illustrates the failure of this conclusion if $\phi$ is not linearly independent).
Let
\[
 0\longrightarrow \mathcal{I}_{\ccR}(1)\longrightarrow \ccO_{\PP_A^N}(1)\longrightarrow \ccO_{\ccR}\longrightarrow 0
\]
be the twisted ideal--ring short exact sequence of $\ccR$, taking into account that $\ccO_{\ccR}(1)\simeq \ccO_{\ccR}$ since $Q$ is local and $\ccR\to Q$ is finite, thus a general hyperplane in $\PP_A^N$ avoids $\ccR$.
Now we consider the following commutative diagram, in which we exploit that
$W^* = \HHH^0(\PP_{\kk}^N,\ccO_{\PP_{\kk}^N}(1))$ and
$\freemodule{W^*} =\HHH^0(\PP_A^N,\ccO_{\PP_A^N}(1))$:
\begin{equation}
\label{equ_ideal_ring_sequences}
\xymatrix{
&  & 0\ar[d]&& 0\ar[d]\\
&  & \mathfrak{m} \cdot \freemodule{W^*}\ar[d]\ar[rr]^{f_{\mathfrak{m}}}&& \mathfrak{m}\cdot \HHH^0(\ccR,\ccO_{\ccR})\ar[d]\\
0\ar[r]&  \HHH^0(\PP_A^N,\mathcal{I}_{\ccR}(1))\ar[d]^{g} \ar[r]& \freemodule{W^*}\ar[d]^{h}\ar[rr]^{f_A}&& \HHH^0(\ccR, \ccO_{\ccR})\ar[d]\\
0\ar[r]&  \HHH^0(\PP_{\kk}^N,\mathcal{I}_{\ccR_{\kk}}(1)) \ar[r]& W^*\ar[d]\ar[rr]^{f_{\kk}}&& \HHH^0(\ccR_{\kk},\ccO_{\ccR_{\kk}})\ar[r]\ar[d]&0\\
& & 0 && 0.
}
\end{equation}
The vertical and horizontal sequences of the diagram are exact.
In particular, as
$\linspan{\ccR_{\kk}}\simeq \PP_{\kk}^{r-1}$
and $\dim_{\kk} \HHH^0(\ccO_{\ccR_{\kk}}) =  r$,
the map $f_{\kk}$ is surjective.  Since $f_{\kk}$ is surjective, there is a linear subspace $V^* \subset
W^*$ mapping isomorphically onto $\HHH^0\left(\ccR_{\kk}, \ccO_{\ccR_{\kk}}(1)\right)$. We can lift this subspace up to a free $A$-submodule $\freemodule{V^*}\subset \freemodule{W^*}$, such that the restriction  $h\vert_{\freemodule{V^*}}:\freemodule{V^*} \longrightarrow V^*$ is surjective with kernel $\mathfrak{m}\otimes_A (\freemodule{V^*}) = \mathfrak{m}\cdot \freemodule{V^*}$.

Furthermore, Nakayama's Lemma implies that
$
  (f_A)|_{\freemodule{V^*}}\colon \freemodule{V^*} \to \HHH^0(\ccR, \ccO_{\ccR})
$
is surjective, which proves \ref{item_linear_forms_surjective_for_independent_familiar}.
Moreover, the surjectivity constructs another commutative diagram with exact rows and columns:
\begin{equation}\label{equ_ideal_ring_sequences_restricted}
\xymatrix{
  & 0\ar[d]& && 0\ar[d]
\\
  & \mathfrak{m} \cdot \freemodule{V^*}\ar[d]\ar[rrr]^{(f_{\mathfrak{m}})|_{\mathfrak{m} \cdot \freemodule{V^*}}}&&& \mathfrak{m}\cdot \HHH^0(\ccR,\ccO_{\ccR})\ar[d]
\\
  & \freemodule{V^*}\ar[d]^{h|_{\freemodule{V^*}}}\ar[rrr]^{(f_A)|_{\freemodule{V^*}}}&&& \HHH^0(\ccR, \ccO_{\ccR})\ar[d]\ar[r]& 0
\\
  0 \ar[r]& V^*\ar[d]\ar[rrr]^{\simeq}&&& \HHH^0(\ccR_{\kk},\ccO_{\ccR_{\kk}})\ar[r]\ar[d]&0
\\
 & 0 && &0
}
\end{equation}
The Snake Lemma applied to \eqref{equ_ideal_ring_sequences_restricted} implies that $(f_{\mathfrak{m}})|_{\mathfrak{m} \cdot \freemodule{V^*}}$ is surjective, hence also
$f_{\mathfrak{m}}$ is
surjective.
Using this, we can apply Snake Lemma again, this time to \eqref{equ_ideal_ring_sequences},
and conclude that $g$ is surjective as claimed.

Since $\linspan{\ccR_{\kk}}=\PP_{\kk}^{r-1}$, take a basis
$\setfromto{\beta_0}{\beta_{N-r}}$
of $\HHH^0\left(\PP_{\kk}^N,\mathcal{I}_{\ccR_{\kk}}(1)\right)$
and lift it to $\setfromto{\overline{\beta}_0}{\overline{\beta}_{N-r}} \subset \HHH^0\left(\PP_A^N,\mathcal{I}_{\ccR}(1)\right)$
so that $g(\overline{\beta}_i) = \beta_i$.
Mapping these two sets to $W^*$ and $\freemodule{W^*}$ respectively, we see that $\setfromto{\overline{\beta}_0}{\overline{\beta}_{N-r}}$ is a collection of $A$-linearly independent forms, as it is a lift of a linearly independent subset from $W^*$ and $\freemodule{W^*}$ is a free module.
Therefore, 
$\ccR\subset\PP_A^{r-1}$ where $\PP_A^{r-1}$ is the zero scheme of
$\setfromto{\overline{\beta}_0}{\overline{\beta}_{N-r}}$ and this concludes the proof of \ref{item_extending_linear_span_to_flat_family}.
\end{proof}

\begin{lemma}\label{lem_dimension_of_linear_span}
   Let $\phi \colon \ccR \to Q$ be linearly independent \familiar{}.
   Then, the dimension of each fibre of
            $\langle\phi\rangle$ over any closed point of $Q$ is $r-1$.
\end{lemma}
\begin{proof}
    By the flat base change Lemma~\ref{lem_flat_base_change_of_linear_spans}, the statement is local, so without loss of generality we can assume $Q=\Spec(A)$ where $(A,\gotm, \kk)$ is a local ring.
    The lower bound on the dimension follows from Lemma~\ref{lem_inclusion_on_linear_spans}, the upper bound follows
    from Theorem~\ref{thm_Nakayama_for_linear_spans}.
\end{proof}

Note that Example~\ref{ex_bad_linear_span} ilustrates that the conclusion of Lemma~\ref{lem_dimension_of_linear_span}
fails to hold if $\ccR$ is not linearly independent --- in fact, in such case the dimension of the fibre of a linear span of the \familiar{} can be larger than $r-1$.
We can also show that, unlike in Example~\ref{ex_bad_linear_span},
the linear span of an independent \familiar{} behaves well under base change.

\begin{corollary}
  \label{cor_base_change_of_linear_span_of_independent_familiar}
  Suppose $\xi\colon Q'\to Q$ is a morphism of schemes with enough of $\kk$-points,  and $\phi\colon\ccR \to Q$ is an independent \familiar.
  Define $\phi'\colon\ccR' \to Q'$ with $\ccR' := \ccR \times_Q Q'$
  to be the pullback \familiar.
  Then,
  $(\xi\times {\PP_{\kk} W})^* (\linspan{\phi}) = \linspan{\phi'}$.
\end{corollary}
\begin{prf}
   By Lemma~\ref{lem_flat_base_change_of_linear_spans}
   it is sufficient to prove the claim if both $Q$ and $Q'$ are local, $Q=\Spec A$ and $Q'=\Spec A'$ for local rings $A$ and $A'$.

   Let $\ccI\subset \ccO_{\PP_A W}$ be the ideal sheaf of $\ccR$
   and set $\ccI(i):=\ccI\otimes_{\ccO_{\PP_A W}} \ccO_{\PP_A W}(i)$.
   The linear span  $\linspan{\ccR}$ is defined by
   $\HHH^0(\ccI\oplus\ccI(1)) \subset
   A\oplus \freemodule{W^*}$, while the pullback by $\xi\times \id_{\PP_kk W}$ is defined by
   $
    \xi^*\HHH^0(\ccI\oplus\ccI(1))  \subset
   A'\oplus \freemodule[A']{W}^*
   $,
   where $\xi^*\HHH^0(\ccI\oplus\ccI(1))$
   denotes the image of the natural map:
   \[
    A'\otimes_{A}\HHH^0(\ccI\oplus\ccI(1)) \to     A'\otimes_{A}\HHH^0(\ccO_{\PP_A W}\oplus\ccQ_{\PP_A W}(1)) = A'\oplus \freemodule[A']{W}^*.
   \]
   Similarly,
   $\linspan{\ccR'}$ is defined by
   \[
   \HHH^0((\xi\times \id_{\PP_{\kk} W})^*(\ccI\oplus \ccI(1))) = \HHH^0(A'\otimes_{A}(\ccI\oplus \ccI(1))) \subset
   A'\oplus \freemodule[A']{W}^*.
   \]
   By Lemma~\ref{lem_inclusion_on_linear_spans}
   we always have
   \[
     A'\otimes_{A}\HHH^0(\ccI\oplus\ccI(1)) \subset \HHH^0(A'\otimes_{A}(\ccI\oplus \ccI(1))),
   \]
   and
   $
     \xi^*\HHH^0(\ccI) = \HHH^0(A'\otimes_{A}\ccI)
   $ always holds since both are equal to the ideal in $A'$ defining the image scheme $\phi'(\ccR')$.
   Whence it remains to prove
    $
      \HHH^0(A'\otimes_{A}\ccI(1))\subset \xi^*\HHH^0(\ccI(1)).
    $
   Since both $\phi$ and $\phi'$ are linearly independent, by Theorem~\ref{thm_Nakayama_for_linear_spans}\ref{item_linear_forms_surjective_for_independent_familiar} there are short exact sequences:
   \begin{alignat*}{1}
      0 \to
      \HHH^0(\ccI(1))
      & \to
      \freemodule{W} \to
      \HHH^0(\ccO_{\ccR}(1)) \to 0,\\
      0 \to
      \HHH^0(A'\otimes_{A}\ccI(1))
      & \to
      \freemodule[A']{W} \to
      \HHH^0(\ccO_{\ccR'}(1)) \to 0.
   \end{alignat*}
   We tensor the top sequence with
   $A'\otimes_{A}(\dotsc)$ and observe that
   \[A'\otimes_{A}\HHH^0(\ccO_{\ccR}(1))\simeq A'\otimes_{A}\HHH^0(\ccO_{\ccR}) = \HHH^0(\ccO_{\ccR'})\simeq \HHH^0(\ccO_{\ccR'}(1)).\]
   As a result
   we obtain a commutative diagram with exact rows:
   \[
   \begin{tikzcd}
      &
      A'\otimes_{A} \HHH^0(\ccI(1))
       \arrow[r]
      &
      \freemodule[A']{W} \arrow[r] \arrow[d, equal]
      &
      \HHH^0(\ccO_{\ccR'}) \arrow[r]
      \arrow[d, equal]
      &  0\\
      0 \arrow[r]
            &
      \HHH^0(A'\otimes_{A}\ccI(1))
       \arrow[r]
      &
      \freemodule[A']{W} \arrow[r]
      &
      \HHH^0(\ccO_{\ccR'}) \arrow[r] &  0.
\end{tikzcd}
   \]
Thus, the image of the top left map, that is
$ \xi^* \HHH^0(\ccI(1))$, concides with
$\HHH^0(A'\otimes_{A}\ccI(1))$ as claimed.
\end{prf}

\begin{corollary}
   \label{cor_linear_span_is_union_of_subspans_over_finite_local_Gor}
   Suppose $Q$ is a scheme with enough $\kk$ points and
   $\phi\colon \ccR\to Q$ is a \familiar[r].
   Then
   \begin{enumerate}
    \item \label{item_local_presentation}
    (Local presentation)
      $\linspan{\ccR} =
     \bigcup_{q\in Q(\kk)}
      \linspan{\ccR|_{\Spec \ccO_{Q,q}}}
   $.
    \item
    \label{item_finite_subscheme_presentation}
    (Finite subscheme presentation)
    If $\phi$ is independent,
    then
         \[
           \linspan{\ccR} =
     \textstyle{\bigcup_{Q' \subset Q}}
     \linspan{\ccR|_{Q'}},
         \]
   where the union is over
      all finite local subschemes $Q' \subset Q$.
    \item \label{item_finite_Gor_morphism_presentation}
    (Finite Gorenstein  morphism presentation)
     If $\phi$ is independent,
    then
  \[
     \linspan{\ccR} =
     \bigcup_{g\colon Q' \to Q}
     (g\times \id_{\PP_{\kk}W} )(\linspan{\ccR'}),
   \]
   where the union is over
     all morphism $g$ from any finite over $\kk$ local Gorenstein scheme $Q'$, $\ccR' =Q' \times_Q \ccR$ and $\phi':\ccR' \to Q'$ is the pulled back \familiar.
  \item \label{item_finite_Gor_subscheme_presentation}
  (Finite Gorenstein subscheme presentation)
     If $\phi$ is flat and independent,
     then
   \[
     \linspan{\ccR} =
     \bigcup_{Q' \subset Q}
     \linspan{\ccR|_{Q'}},
   \]
   where the union is over
      all finite local Gorenstein subschemes $Q' \subset Q$.
   \end{enumerate}
\end{corollary}

\begin{prf}
   In all cases, the inclusion
   $ \supset
     $
   follows from Lemma~\ref{lem_inclusion_on_linear_spans}.
   To argue for the opposite inclusions we use Proposition~\ref{prop_scheme_inclusion_equivalence}, and we pick a finite local Gorenstein $R'\subset \linspan{\ccR}$.
   Our goal is to prove that $R'$ is contained in the appropriate union, or more  precisely in one of the summands.

   In case of \ref{item_local_presentation}
   we set $q$ to be the image of the point of support of $R'$
   and the containment $R'\subset \linspan{\ccR|_{\Spec \ccO_{Q,q}}}$ follows from the flat base change, Lemma~\ref{lem_flat_base_change_of_linear_spans}.

   In case of
   \ref{item_finite_subscheme_presentation},
   we set $Q'$ to be the image of $\phi(R')$ and the containment $R'\subset \linspan{\ccR|_{Q'}}$ follows from  Corollary~\ref{cor_base_change_of_linear_span_of_independent_familiar}.

   In case of
   \ref{item_finite_Gor_morphism_presentation},
   we set $Q':=R'$ and $g\colon Q'\to Q$ 
    to be the projection $\phi$ and the containment
     $R'\subset (g\times \id_{\PP_{\kk} W} )(\linspan{\ccR'})$ 
      follows again from  Corollary~\ref{cor_base_change_of_linear_span_of_independent_familiar}.

Finally, we prove \ref{item_finite_Gor_subscheme_presentation}, where $\phi$ is flat and independent.
Using \ref{item_finite_subscheme_presentation}, we may suppose $Q$ is finite and local.
Thus, $\linspan{\ccR} \simeq \PP_Q^{r-1}$ 
and for each subscheme $Q' \subset Q$ the linear span of the restriction $\linspan{\ccR|_{Q'}}$ is the restriction of linear span $\linspan{\ccR}|_{Q'}$
   by Corollary~\ref{cor_base_change_of_linear_span_of_independent_familiar}.
That is, $\linspan{\ccR|_{Q'}} = \PP_{Q'}^{r-1}$, and we can apply
Proposition~\ref{prop_scheme_inclusion_equivalence} to $Q$ and multiply it by $\PP_{\kk}^{r-1}$:
   \[
     \linspan{\ccR} =
     Q\times_{\kk}\PP_{\kk}^{r-1} =
     \left(\bigcup_{Q' \subset Q}
      Q'\right)\times_{\kk}\PP_{\kk}^{r-1} =
     \bigcup_{Q' \subset Q} (Q'\times_{\kk}\PP_{\kk}^{r-1}) =
     \bigcup_{Q' \subset Q}
     \linspan{\ccR|_{Q'}}.
   \]
\end{prf}

It is tempting to claim that  it should be enough to consider only embeddings $Q'\subset Q$  of Gorenstein schemes as in Corollary~\ref{cor_linear_span_is_union_of_subspans_over_finite_local_Gor}\ref{item_finite_Gor_subscheme_presentation} also in the non-flat case.
This however is not the case, as illustrated by the following example.

\begin{example}
  Consider
  $A=\kk[s,t]/(s^2, st,t^2)$
  and let $\ccR \subset \PP_A^2$
  be the finite scheme of length $6$, with the following homogeneous ideal $\ccI_{\ccR} \subset A[\alpha,\beta, \gamma]$:
  \[
     \ccI_{\ccR} = (t\alpha, s\beta, s\alpha-t\beta, \alpha^2, \alpha\beta, \beta^2)
  \]
  Let $\phi\colon \ccR \to Q=\Spec A$ be the projection morphism.
  Then $Q$ is not Gorenstein, $\phi$ is an independent but not flat \familiar[3], and its linear span is defined by the ideal $(t\alpha, s\beta, s\alpha-t\beta)$.
  For any proper subscheme $Q'\subsetneqq Q$
  the ideal of $\ccR' := \ccR|_{Q'}$ contains $s\alpha$ (and hence also $t\beta$), which is a linear function in the ideal of each $\linspan{\ccR'}$,
  and thus also in the intersection of all those ideals.
  Therefore, the union of linear spans of restrictions to Gorenstein subschemes does not cover
  $\linspan{\ccR}$.
\end{example}

\subsection{Minors and linear spans}

Suppose $Q$ is a scheme with enough $\kk$-points,
and let $M$ be a $u_1\times u_2$ matrix 
 whose entries are linear forms on $\affinespaceof[\ccO_Q(Q)]{W}$.
In a coordinate free manner, this is equivalent to existence of vector spaces $U_1,U_2$ over $\kk$ of dimensions respectively $u_1$ and $u_2$,
and the $\ccO_Q$-linear map $M\colon \freemodule[Q]{W} \to \freemodule[Q]{U_1}\otimes_{Q} \freemodule[Q]{U_2}$.
We will also consider the corresponding map of projective spaces, which however might be defined only on an open subset of $\PP_Q W$.
Explicitly, let $Z= \set{M=0}$ be the closed subscheme obtained by vanishing of all entries of $M$.
Then we get a map:
\[
 M_{\PP} \colon \PP_Q W \setminus Z \to \PP_Q (U_1\otimes_{\kk} U_2).
\]
Note that $\PP_Q W \setminus Z$ is not necessarily dense, even if $M$ in nonzero in the neighbourhood of each point.
\begin{example}
   Suppose $Q=\Spec A$ with $A$  a finite local ring with the maximal ideal $\gotm\subset A$.
   Then if all entries of $M$ are in $\gotm$, then $\PP_A W \setminus Z =\emptyset$ and $M_{\PP}$ is undefined.
\end{example}

\begin{example}
   Suppose $Q=\Spec \kk[s,t]/(st)$ is a union of two affine lines intersecting in one point
   and
   $M = \begin{pmatrix}
         s\alpha & t^2s\beta \\
         0 & ts \gamma
        \end{pmatrix}
$. Then $Z$ contains the whole projective space over one of the affine lines $\set{s=0}$ and $M_{\PP}$ is defined on an open subset of $\PP_B W$, where $B=\kk[t, t^{-1}]$.
\end{example}

For an integer $r$, we define the scheme 
$\set{\rk M \leqslant r}:=Z(I_{r})\subset \PP_Q W$,
 where $I_{r}\subset \ccO_Q(Q)\otimes_{\kk} \Sym W^* $
 is the homogeneous ideal of $(r+1)\times (r+1)$ minors of $M$.
 Hence, the scheme $\set{\rk M \leqslant r}$ is a closed subscheme of $\PP_Q W$,
 and in particular, it is projective over $Q$.
Note that that if $r\geqslant u_1$ or $r\geqslant u_2$, then $I_r=\set{0}$ and $\set{\rk M \leqslant r} = \PP_Q W$.
We also define the quasiprojective (over $Q$) scheme
\[
 \set{\rk M = i}:= \set{\rk M \leqslant i}\setminus \set{\rk M \leqslant i-1}.
\]

An important special case is $X=\set{\rk M = 1}$.
Then the image of the map
$M_{\PP}|_{X} \colon X \to \PP_Q(U_{1}\otimes_{\kk} U_{2})$ is contained in $\PP_{Q} U_1\times_{Q}\PP_{Q} U_2$,
and therefore $X$ admits two projections $\pi_{i,M} \colon X\to \PP_{Q} U_i$ for $i=1,2$.

 \begin{corollary}\label{cor_relative_linear_span_of_constant_rank_subschemes}
With the above notation,
suppose $\ccR\subset \set{\rk M = 1}\subset \PP_{Q} W$ is a subscheme such that the image $\pi_{1,M}(\ccR)$ is a linearly independent \familiar{} in $\PP_{\kk} U_1$ over $Q$.
Then  $\langle \ccR\rangle \subset \PP_{Q} W$ is contained in $\set{\rk M \leqslant r}$.
\end{corollary}

\begin{proof}
By Corollary~\ref{cor_linear_span_is_union_of_subspans_over_finite_local_Gor}\ref{item_local_presentation}
it is enough to prove the claim locally, for $Q=\Spec A$ with $(A, \gotm, \kk)$ a local ring.
Then by Theorem~\ref{thm_Nakayama_for_linear_spans} there is an inclusion
$\pi_{1, M}(\ccR) \subset \PP_A U_1'\simeq \PP_A^{r-1}$ 
for some linear subspace $\PP_A U_1' \subset \PP_A U_1$
 with $\dim_{\kk} U_1' = r$ 
(note that the embedding 
$A\otimes_{\kk} U_1' \subset A \otimes_{\kk} U_1$
 is $A$-linear, 
but not necessarily comes from any $\kk$-linear embedding $U_1' \subset U_1$).
 Thus,
\begin{align*}
\ccR & \subset
M_{\PP}^{-1}(\PP_A U_1' \times_{A} \PP_A U_2)
\subset  M_{\PP}^{-1}(\PP_A(U_1'\otimes_{\kk} U_2)) =: \PP_A W'  \text{, and also}\\
\linspan{\ccR} &\subset M_{\PP}^{-1}(\linspan{\PP_A U_1' \times_{A} \PP_A U_2}) = \PP_A W' \subset \PP_A W.
\end{align*}
Thus we have reduced the problem to the $'$-ed situation, where we replace $U_1$ with $U_1'$,  $W$ with $W'$, etc.
But here, since $\dim U_1'= r$,
the claim is trivial, as all $r\times u_2$ matrices have all $(r+1) \times (r+1)$ minors trivial.
\end{proof}

The following example shows that we cannot replace the constant rank assumption
$\ccR\subset \set{\rk M=1}$ with only
bounded rank assumption
$\ccR\subset \set{\rk M\leqslant 1}$.
\begin{example}\label{ex_injectivity_assump}
Let $A=\kk$  and over $\PP_{\kk}^2$ let
\[
M=\begin{pmatrix}
\alpha& 0 & 0\\
0 & \alpha& 0\\
0 & 0 &\alpha
\end{pmatrix}
\]
We take the flat linearly independent \familiar[2]{}
$\ccR\subset \set{\rk M\leqslant 1} = Z(\alpha^2)$, given by the ideal
\[
I_{\ccR}=(\alpha^2, \beta)\subset \kk[\alpha, \beta, \gamma]
\]
Then $\linspan{\ccR} = Z(\beta)$,
but $\set{\rk M\le 2} = Z(\alpha^3)$, thus $\linspan{\ccR} \not\subset \set{\rk M\le 2}$.
\end{example}

At the moment it is not clear to us if the assumption
that $\pi_{1,M}(\ccR)$ is linearly independent in Corollary~\ref{cor_relative_linear_span_of_constant_rank_subschemes}
is strictly necessary, see Example~\ref{ex_double_bad_linear_span}.
Moreover, we do not know yet if any generalization to higher constant rank is true, 
by analogy to \cite[Thm~1.18] {galazka_phd}.
However, none of those is  necessary to prove the claims in this article.

\begin{example}\label{ex_double_bad_linear_span}
   We consider the second Veronese reemebedding of a variant Example~\ref{ex_bad_linear_span}.
   That is, suppose $Q=\Spec A= \Spec \kk[s,t]/(s^3, t^3)$ and (ignoring two identically zero coordinates) consider $\PP_{A}^7 = \Proj A[\alpha_{00},\alpha_{01},\alpha_{02},\alpha_{03},\alpha_{11},\alpha_{22},\alpha_{23},\alpha_{33}]$. Take the following three $A$-points in $\PP_A^7$:
   \[
    [1,1,0,0,1,0,0,0], [0,0,0,0,1,0,0,0],
    [1,0,s,t,0,s^2,st, t^2].
   \]
   Their homogeneous ideals are, respectively:
   \begin{align*}
     I'&=(\alpha_{00}-\alpha_{01},\alpha_{01}-\alpha_{11}, \alpha_{02},\alpha_{03},\alpha_{22},\alpha_{23},\alpha_{33}),\\
     I''&=(\alpha_{00},\alpha_{01}, \alpha_{02},\alpha_{03},\alpha_{22},\alpha_{23},\alpha_{33})\\
     I'''&=(\alpha_{01},\alpha_{11}, \alpha_{02}- s\alpha_{00},\alpha_{03}- t\alpha_{00},\alpha_{22}- s^2\alpha_{00},\alpha_{23}- st\alpha_{00},\alpha_{33}- t^2\alpha_{00}),
   \end{align*}
and the scheme $\ccR\subset \PP_A^7$ we consider is given by the intersection ideal
$I= I'\cap I''\cap I'''$.
It is straightforward to verify that $\ccR$ is a flat linearly independent \familiar[3]{} over $Q$.
Its linear span is defined by the following $5$ linear equations:
\begin{align*}
 I_1 = \langle&
 \alpha_{02}- s(\alpha_{00}-\alpha_{01}),
 \alpha_{03}- t(\alpha_{00}-\alpha_{01}),
 \alpha_{22}- s^2(\alpha_{00}-\alpha_{01}),\\
&\alpha_{23}- st(\alpha_{00}-\alpha_{01}),
 \alpha_{33}- t^2(\alpha_{00}-\alpha_{01})\rangle.
\end{align*}
Moreover,
$\ccR\subset \set{\rk M =1}$ where $M$ is the following symmetric matrix:
\[
  M=\begin{pmatrix}
       \alpha_{00} & \alpha_{01} & \alpha_{02} &\alpha_{03}\\
       \alpha_{01} & \alpha_{11} & 0 &0\\
       \alpha_{02} & 0 & \alpha_{22}& \alpha_{23} \\
       \alpha_{03} & 0 & \alpha_{23} &  \alpha_{33}
    \end{pmatrix}
\]
In the notation of Corollary~\ref{cor_relative_linear_span_of_constant_rank_subschemes},
$\pi_{1,M}(\ccR)$ and $\pi_{2,M}(\ccR)$ 
are not linearly independent.
Yet, despite the assumption of 
Corollary~\ref{cor_relative_linear_span_of_constant_rank_subschemes} is not satisfied (even after swapping the roles of $\pi_1$ and $\pi_2$),
we still have a containment
$\PP\linspan{\ccR} \subset \set{\rk M\le 3}$.
\end{example}

\section{Relative linear span and cactus scheme}\label{sec_cactus_scheme}

In this section, we introduce and discuss the cactus scheme and its ``flat'' version (Definition~\ref{def_cactus_schemes}).
These are defined as unions of relative linear spans (Definition~\ref{def_relative_linear_span}) of independent \familiar's. We show in Proposition~\ref{prop_cactus_scheme_supported_at_cactus_variety} that (under suitable assumptions) their reduced subschemes coincide with the cactus varieties known from earlier works. We also show that appropriate minors tend to vanish on the cactus scheme (Theorem~\ref{thm_X_in_constant_rank_one_then_cactus_in_bounded_rank}).

\begin{definition}
\label{def_relative_linear_span}
Let $\ccR\subset \PP_{Q}W$ be a \familiar{} over a scheme $Q$ with enough $\kk$-points, given by the homogeneous ideal sheaf $\mathcal{I}_\ccR\subset \ccS$.
Suppose that $\ccR \to Q$ is linearly independent.
We define \textit{the relative linear span}
$\rellinspan{\ccR}$ of $\ccR$ to be the scheme-theoretic image of $\linspan{\ccR}$ under the projection map
$\pi\colon \PP_Q W \to \PP_{\kk} W$, that is
\[
  \rellinspan{\ccR}:=\overline{\pi(\linspan{\ccR})}\subset \PP_{\kk} W.
\]
\end{definition}
Similarly, one can define the affine version of relative linear span inside $\affinespaceof{W} \simeq\AAA_{\kk}^{N+1}$, but we will not use it here.

Relative linear spans of flat \familiar's without the assumption of linear independence, but only over a reduced base $Q$ are discussed in \cite[\S5.6--5.7]{jabu_jelisiejew_finite_schemes_and_secants}.
In particular, the definition there is ``generic'', component by component.
If the base is not reduced, 
 then the generic definition would not take into account the non-reduced structure. 
In the case of reduced base and linearly independent flat \familiar, the two definitions agree.
On the other hand, if the \familiar{} is not independent,
then one has to deal with pathologies similar to Example~\ref{ex_bad_linear_span}, which is beyond the scope of this article.

\begin{proposition}
\label{prop_basic_properties_of_relative_span}
   Let $Q$ be a scheme with enough $\kk$-points,
    and let $\ccR \subset \PP_Q W$ be a linearly independent \familiar{} over $Q$ with $\pi\colon \PP_Q W \to \PP_{\kk}W$ the natural projection.
  Let $Q'$ be another scheme with enough $\kk$-points.
   The following are basic properties of relative linear span.
   \begin{enumerate}
    \item \label{item_scheme_is_contained_in_its_relative_linear_span}
    $\pi(\ccR) \subset \rellinspan{\ccR}$
    \item
    \label{item_relative_linear_span_via_pullback}
    If $Q'\to Q$ is a morphism and $\ccR'= Q'\times_{Q}\ccR \subset \PP_{Q'} W$ is the pullback,
    then $\rellinspan{\ccR'}\subset \rellinspan{\ccR}$ and if in addition the scheme theoretic image of $Q'$ under $Q' \to Q$ is $Q$, then $\rellinspan{\ccR'} = \rellinspan{\ccR}$.
   \end{enumerate}
\end{proposition}

\begin{prf}
Item~\ref{item_scheme_is_contained_in_its_relative_linear_span} follows from the definition,
  since $\ccR\subset \linspan{\ccR}$,
  and thus $\overline{\pi(\ccR)}\subset \overline{\pi\left(\linspan{\ccR}\right)} = \rellinspan{\ccR}$.
Item~\ref{item_relative_linear_span_via_pullback} follows immediately from Corollary~\ref{cor_base_change_of_linear_span_of_independent_familiar}.
\end{prf}

\begin{definition}
\label{def_cactus_schemes}
Let $X\subset \PP_{\kk} W$
  be a locally closed subscheme.
\textit{The $r$-th cactus scheme} of $X$ is defined as
\[
  \cactussch{r}{X}:=
   \overline{\bigcup_{\ccR \to Q} \rellinspan {\ccR}},
\]
where the scheme theoretic union
  is over all independent \familiar[(r')]'s $\ccR\to Q$ in $X$ over schemes $Q$ with enough $\kk$-points for all $r'\le r$.
Furthermore, we define \textit{the flat} version of $r$-th cactus scheme:
\[
  \cactusflat{r}{X}:=\overline{\bigcup_{\ccR\to Q \text{ flat}} \rellinspan {\ccR}},
\]
but here the union is over all independent and flat  \familiar[(r')]'s
$\ccR\to Q$ in $X$ over schemes $Q$ with enough $\kk$-points for all $r'\le r$.
\end{definition}

\begin{proposition}\label{prop_basic_properties_of_cactus_schemes}
  For any locally closed subschemes $X, Y \subset \PP_{\kk}(W)$ and any integer $r\geqslant 1$,
  the following are basic properties
   of their cactus schemes.
  \begin{enumerate}
   \item  \label{item_cactus_flat_in_cacactus_scheme_in_linspan}
    $\cactusflat{r}{X} \subset \cactussch{r}{X} \subset \linspan{X}$.
   \item If $r=1$, then $\cactusflat{1}{X}= \cactussch{1}{X} = \overline X$.
   \item If $Y\subset X$, then
     $\cactusflat{r}{Y} \subset \cactusflat{r}{X}$
     and
     $\cactussch{r}{Y} \subset \cactussch{r}{X}$.
  \end{enumerate}
\end{proposition}

\begin{lemma}\label{lem_cactus_is_determined_by_finite_local_Gor}
  In the definitions of $\cactussch{r}{X}$ and $\cactusflat{r}{X}$
  it is enough to use $Q$ which are finite local and Gorenstein.
\end{lemma}

\begin{proof}
  The lemma is an immediate consequence of the definitions and 
Corollary~\ref{cor_linear_span_is_union_of_subspans_over_finite_local_Gor}\ref{item_finite_Gor_morphism_presentation}.
\end{proof}

Suppose $X$ contains at least $r$ of $\kk$-points which are linearly independent
(for example $\dim X>0$ and $\dim_{\kk}\linspan{X_{\red}} \ge r-1$).
Then in the definitions of
$\cactusflat{r}{X}$ and
$\cactussch{r}{X}$
it is enough to consider \familiar's,
without the necessity of considering \familiar[(r')]'s for $r'<r$.
Indeed, locally, we can always add some $r-r'$ of the independent $\kk$-points multiplied by the local base $Q$ to obtain a \familiar{} out of \familiar[(r')].
Thus, in the proofs we will often for a simplicity of notation only consider the \familiar's.

Moreover, in the rest of the article, we will
frequently assume that the scheme $X\subset \PP_{\kk} W $
 under consideration imposes independent linear conditions 
 to its subschemes of length at most $r$, that is:
\begin{itemize}
 \item any finite subscheme $R\subset X$ of length at most $r$ is linearly independent,
 \item in particular,
    any \familiar{} in $X$ is independent.
\end{itemize}
The main example of such $X$ is the Veronese variety $\nu_d(\PP_{\kk} V) \subset \PP_{\kk}(S^{(d)} V)$ for $d\geqslant r-1$.

\begin{proposition}\label{prop_cactus_scheme_supported_at_cactus_variety}
   Suppose $X\subset \PP_{\kk} W$
   imposes independent conditions to its subschemes of length at most $r$.
   Then the reduced subschemes of
   $\cactussch{r}{X}$ and of
   $\cactusflat{r}{X}$ are equal to the ``cactus variety'' $\cactus{r}{X}
   := \overline{\bigcup_{R\subset X} \linspan{R}}$, where the union is over all  subschemes $R$ of length at most $r$.
\end{proposition}

\begin{prf}
   Any finite subscheme $R\subset X$ of length $r'$ with $r'\le r$ produces a ``trivial'' \familiar[(r')]{}
   $R\to \Spec \kk$ over $\Spec \kk$ in $X$.
   By our assuptions, this \familiar[(r')]{} is independent and also (trivially) flat, hence its relative linear span (which is equal to $\linspan{R}$) is contained in both  $\cactussch{r}{X}$ and
   $\cactusflat{r}{X}$, proving
  $\cactus{r}{X} \subset \reduced{\cactusflat{r}{X}}\subset \reduced{\cactussch{r}{X}}$.

  On the other hand, by Lemma~\ref{lem_cactus_is_determined_by_finite_local_Gor}
  the cactus scheme is the closure of the union of $\rellinspan{\ccR}$ for a \familiar{} over
  a finite local base $Q$.
  If $p\in \rellinspan{\ccR}$ is any point, then $p$ is an image of some point in $p'\in \linspan{\ccR}$,
  and the image of $p'$ in $Q$ must be the unique point $q \in Q$.
  Thus, $p\in \linspan{\ccR|_q}$ 
(where $\ccR|_q$ is seen as a subscheme of $X\subset \PP_{\kk} W$) by Corollary~\ref{cor_base_change_of_linear_span_of_independent_familiar}, proving that $p\in  \cactus{r}{X}$, and concluding the arguments.
\end{prf}

By the above proposition, we think of $\cactusflat{r}{X}$ and $\cactussch{r}{X}$
 as of two fairly natural (geometrically defined) choices of scheme structure on the cactus variety that was  studied previously \cite{nisiabu_jabu_cactus}, \cite{bernardi_jelisiejew_marques_ranestad_cactus_rank_of_a_general_form}, \cite{galazka_mandziuk_rupniewski_distinguishing}.
Lemma~\ref{lem_cactus_flat_is_determined_by_Hilb} 
shows that $\cactusflat{r}{X}$ is simpler to study,
 while the determinantal criterion 
 (Theorem~\ref{thm_X_in_constant_rank_one_then_cactus_in_bounded_rank})
 indicates that in fact $\cactussch{r}{X}$ 
 is more likely to be the interesting one,
 related to the structure of appropriate vanishing schemes of minors.
At this moment, we cannot give an interesting example 
 that would prove the two scheme structures are different.

Suppose $X$ is a locally closed subset of $\PP_{\kk} W$ which has at least $r$ of $\kk$-points and imposes independent linear conditions to its subschemes of length at most $r$.
We consider the Hilbert scheme $\Hilb_r(X)$ parametrising flat \familiar's.
Denote by $\ccR^{\univ}\subset \Hilb_r(X) \times X$ the universal flat \familiar,
so that the projection
$\phi^{\univ}\colon\ccR^{\univ}\to \Hilb_r(X)$ is a \familiar{}
and any flat \familiar{} $\ccR\to Q$ is a pullback of $\phi^{\univ}$ under some morphism $Q\to \Hilb_r(X)$.

\begin{lemma}\label{lem_cactus_flat_is_determined_by_Hilb}
   Suppose $X\subset \PP_{\kk} W$ has at least $r$ of $\kk$-points
   and imposes independent linear conditions to its subschemes of length at most $r$.
   In the definition of $\cactusflat{r}{X}$
  it is enough to use only one \familiar{} $\phi^{\univ} \colon \ccR^{\univ}\to \Hilb_r(X)$:
  \[
    \cactusflat{r}{X} = \rellinspan{\ccR^{\univ}}.
  \]
\end{lemma}

\begin{proof}
  By the assumptions $\phi^{\univ}$ is a flat and independent \familiar{} and the definition of cactus scheme implies that
  $\cactusflat{r}{X} \supset \rellinspan{\ccR^{\univ}}$.
  On the other hand, suppose $\ccR\to Q$ is an independent and flat \familiar{} in $X$.
  Then, using the universal map $Q\to \Hilb_{r}(X)$, from Proposition~\ref{prop_basic_properties_of_relative_span}\ref{item_relative_linear_span_via_pullback}
  we see that $\rellinspan{\ccR} \subset \rellinspan{\ccR^{\univ}}$ proving that also
  $\cactusflat{r}{X} \subset \rellinspan{\ccR^{\univ}}$.
\end{proof}

The following theorem is the main reason we propose the two cactus scheme structures.
Similarly to the situation in Corollary~\ref{cor_relative_linear_span_of_constant_rank_subschemes} let $M$ be a matrix whose entries are elements of $W^*$ (this time the coefficients of linear forms are in $\kk$, not in in $\ccO_Q(Q)$).
Suppose $\set{\rk M=0}$ is supported at $0$ so that
$M_{\PP}\colon \PP_{\kk} W \to \PP_{\kk}(U_1\otimes_{\kk} U_2)$ is an embedding.
For a scheme $X\subset \set{\rk M =1}\subset \PP_{\kk} W$
 and $i=1,2$ we denote by $\pi_{i, M}\colon X\to \PP_{\kk} U_i$ the composition of the embeddding
$X\to \PP_{\kk} U_1\times \PP_{\kk}U_2$  and the projection onto $\PP_{\kk} U_i$.

\begin{theorem}
\label{thm_X_in_constant_rank_one_then_cactus_in_bounded_rank}
 Suppose $M$, $X$ and $\pi_{i, M}$ are as above with additional assumptions:
 \begin{itemize}
  \item $\pi_{1,M}$ is an embedding of $X$, and
  \item $\pi_{1,M}(X)$ imposes  independent linear conditions to its subschemes of length at most $r$.
 \end{itemize}
 Then $\cactussch{r}{X}\subset\set{\rk M\le r}$.
\end{theorem}
\begin{prf}
Let $Q$ be any scheme that has enough $\kk$-points and
$\ccR\to Q$ be any \familiar{} in $X$.
Then the assumptions imply that
$ (\id_{Q}\times\pi_{1,M})(\ccR) \to Q$ is a linearly independent \familiar{} in $\PP_{\kk} U_1$, thus
by
Corollary~\ref{cor_relative_linear_span_of_constant_rank_subschemes},
we have
$\rellinspan {\ccR} \subset \set{\rk M \le r}$.
From the definition of the cactus scheme as a union of such relative linear spans we conclude that
$\cactussch{r}{X} \subset \set{\rk M \le r}$ as claimed.
\end{prf}

Recall the catalecticant map
$\Upsilon^{i, d-i}\colon
\ccO_{\PP_{\kk}(S^{(d)}V)} \otimes_{\kk} S^i V^* \to
\ccO_{\PP_{\kk}(S^{(d)}V)}(1) \otimes_{\kk} S^{(d-i)}V$ from Section~\ref{sec_intro}, which we can view as a matrix, called a catalecticant matrix.
As a conclusion we obtain the vanishing of catalecticant miniors on the cactus scheme of Veronese variety.

\begin{theorem}\label{thm_cactus_scheme_contained_in_vanishing_of_catalecticant_minors}
Consider the Veronese variety
$\nu_d(\PP_{\kk} V) \subset \PP_{\kk}(S^{(d)} V)$.
For all positive integers $r$, and  $i$ such that and $r-1 \leqslant i\leqslant d-1$,
we have $\cactussch{r}{\nu_d(\PP_{\kk}V)} \subset \Upsilon^{i,d-i}_{r} =  \set{\rk \Upsilon^{i, d-i} \leqslant r}$.
\end{theorem}
\begin{proof}
For $M$ being the catalecticant matrix $\Upsilon^{i, d-i}$ and $X=\nu_d(\PP_{\kk} V)$
we verify that all the assumptions of Theorem~\ref{thm_X_in_constant_rank_one_then_cactus_in_bounded_rank} are satisfied.
Indeed, the catalecticant matrix is $0$ only at $0\in S^{(d)}V$ and
by the result of Pucci \cite{pucci_Veronese_variety_and_catalecticant_matrices}, we have
\[
   \nu_d(\PP^n)=\Upsilon^{i,d-i}_1 = \set{\rk \Upsilon^{i,d-i} = 1}.
\]
Moreover, $\pi_{1, \Upsilon^{i,d-i}}(\nu_{d}(\PP_{\kk} V)) = \nu_i(\PP_{\kk}V) \subset \PP_{\kk}(S^{(i)} V)$.
Hence,
$\pi_{1, \Upsilon^{i,d-i}}(\nu_{d}(\PP_{\kk} V))$ imposes independent linear conditions on its subschemes of length at most $r$ for $i\geqslant r-1$.
Therefore, by Theorem~\ref{thm_X_in_constant_rank_one_then_cactus_in_bounded_rank}
we have
 $\cactussch{r}{\nu_d(\PP_{\kk} V)} \subset \Upsilon^{i,d-i}_r = \set{\rk \Upsilon^{i,d-i} \le r}$
as claimed.
\end{proof}

\begin{corollary}
   Suppose $d$ and $r$ are positive integers such that 
   $d\geqslant 2r-3$. 
   Then, for any integer $i$, 
   we have $\cactussch{r}{\nu_d(\PP_{\kk}^n)} \subset \Upsilon^{i,d-i}_{r} =  \set{\rk \Upsilon^{i, d-i} \leqslant r}$.
\end{corollary}
\begin{prf}
   If $i\leqslant 0$ or $i\geqslant d$, then the catalecticant minors are trivial and there is nothing to prove.
   If $r-1\leqslant i \leqslant d-1$, then we apply Theorem~\ref{thm_cactus_scheme_contained_in_vanishing_of_catalecticant_minors}.
   If $1\leqslant i \leq r-2$, then
   the symmetry of catalecticant minors provides
   $\set{\rk \Upsilon^{i, d-i} \leqslant r}=\set{\rk \Upsilon^{d-i, i} \leqslant r}$,
    and we have $r-1\leqslant d-i \leqslant d-1$.
    Thus, the same theorem with the roles of $i$ and $d-i$
 swapped proves also this case.
\end{prf}

\section{Relative apolarity}
\label{sec_relative_apolarity}
In this section, we provide a very general setting of relative apolarity developing 
the classical apolarity theory \cite{emsalem}, 
and some of its extended versions in \cite{joachimapo} and \cite{KlKlep}.
Suppose $V$ is a vector space over $\kk$ and $Q$ is a scheme with a morphism $F \colon Q \to \PP S^{(d)}V$.
The main aim of this section is to define and investigate the annihilator of $F$,
which is a homogeneous ideal sheaf in the graded 
$\ccO_Q$-algebra:
$\ccO_Q \otimes_{\kk} \Sym V^* =
\bigoplus_{i=0}^{\infty}
\ccO_Q \otimes_{\kk} S^i V^*$.
For a purpose of further reference, 
we will consider a slightly more general setting.

Throughout this section, for consistency, we assume that $S^{(i)} V =0$ for every $i<0$.

\subsection{Apolarity action}\label{sec_apolarity_action}

We define the apolarity pairing:
\[
 \hook \colon \ccO_Q \otimes_{\kk} \Sym V^*
 \times \prod_{i=0}^{\infty}
\ccO_Q \otimes_{\kk} S^{(i)} V \to
\prod_{i=0}^{\infty}
\ccO_Q \otimes_{\kk} S^{(i)} V.
\]
Here
\[
\prod_{i=0}^{\infty} \ccO_Q \otimes_{\kk} S^{(i)} V = \Hom_{\kk} (\ccO_Q \otimes_{\kk} \Sym V^*, \kk),
\]
and the apolarity action is defined by the following formula: for any open subset $U\subset Q$
choose $\Theta \in \ccO_Q(U) \otimes_{\kk} \Sym V^*$ and
$F \in \prod_{i=0}^{\infty} \ccO_Q(U) \otimes_{\kk} S^{(i)} V$, define
\[
  (\Theta \hook F) (\Psi) = F(\Theta\cdot \Psi)
\]
for all $\Psi \in \ccO_Q(U) \otimes_{\kk} \Sym V^*$.
If we choose a basis $\setfromto{x_0}{x_n}$ of $V$ and its dual basis $\setfromto{\alpha_0}{\alpha_n}$ of $V^*$, then
we can identify
$\prod_{i=0}^{\infty} \ccO_Q \otimes_{\kk} S^{(i)} V$ with the divided power series ring $\kk_{dp}[[\fromto{x_0}{x_n}]]$.
The apolarity action satisfies the following properties :
\begin{enumerate}
 \item For any index $i$
 \[
   \alpha_i \hook x_0^{(a_0)} \dotsm x_n^{(a_n)} =
   \begin{cases}
      x_0^{(a_0)} \dotsm x_i^{(a_i-1)}\dotsm x_n^{(a_n)} &  \text{if } a_i >0\\
      0 & \text{if } a_i=0.
   \end{cases}
 \]
 \item The pairing is $\ccO_Q$ bilinear, that is, for any open subscheme $U\subset Q$,
 if $s_1,s_2,t_1,t_2\in \ccO_Q(U)$, $\Theta_1, \Theta_2\in \ccO_Q(U) \otimes_{\kk} \Sym V^*$,
 and $G_1, G_2 \in
\bigoplus_{i=0}^{\infty}
\ccO_Q(U) \otimes_{\kk} S^{(i)} V$ we have:
  \begin{multline*}
      (s_1\Theta_1 + s_2 \Theta_2)
      \hook (t_1 G_1 + t_2 G_2) =
        s_1t_1 \Theta_1 \hook G_1
      + s_1t_2 \Theta_1 \hook G_2
      + s_2t_1 \Theta_2 \hook G_1
      + s_2t_2 \Theta_2 \hook G_2.
  \end{multline*}
\item
   The apolarity action
   is consistent with the algebra structure
   of $\ccO_Q \otimes_{\kk} \Sym V^*$,
   that is for any open subscheme $U\subset Q$,
  $\Theta_1, \Theta_2\in \ccO_Q(U) \otimes_{\kk} \Sym V^*$, and $G \in \bigoplus_{i=0}^{\infty}
\ccO_Q(U) \otimes_{\kk} S^{(i)} V$:
   \[
      (\Theta_1 \Theta_2) \hook G = \Theta_1 \hook (\Theta_2 \hook G)
      =\Theta_2 \hook (\Theta_1 \hook G).
   \]
\end{enumerate}
These properties determine uniquely the pairing $\hook$ on
the polynomial part
\begin{align*}
\bigoplus_{i=0}^{\infty}
\ccO_Q \otimes_{\kk} S^{(i)} V &\subset
\prod_{i=0}^{\infty}
\ccO_Q \otimes_{\kk} S^{(i)} V, \text{ or equivalently,}\\
\kk_{dp}[\fromto{x_0}{x_n}] &\subset
\kk_{dp}[[\fromto{x_0}{x_n}]].
\end{align*}
For further properties of $\hook$ see for instance
\cite[\S3.1, \S4.4]{jelisiejew_PhD}.
In particular, $\hook$ behaves well with respect to morphisms of the base scheme:
\begin{equation}\label{equ_hook_preserved_by_pullback}
\text{if } f\colon Q'\to Q, \text{ then  } f^*\otimes \id_{\kk_{dp}[[\fromto{x_0}{x_n}]]} (\Theta \hook G) = f^* \otimes \id_{\Sym V^*}( \Theta) \hook f^*\otimes \id_{\kk_{dp}[[\fromto{x_0}{x_n}]]}( G).
\end{equation}
\begin{notation}\label{not_skip_id}
  For brevity and clarity of presentation, 
we will frequently skip tensorising maps with identity $\otimes \id_{\dots}$, whenever no confusion may arise.
  For instance, a shorter version of \eqref{equ_hook_preserved_by_pullback} is $f^* (\Theta \hook G) = f^* \Theta \hook f^*  G$.
  Slightly more formally,
  whenever we have a morphisms of rings,
  modules, or sheaves $\varphi\colon \ccF\to \ccG$,
  and $W$ is another ring, module, or sheaf,
  such that it makes sense to consider 
   $\ccF \otimes_{\bullet} W$ and $\ccG \otimes_{\bullet} W$,
   then the induced morphism $\ccF \otimes_{\bullet} W\longrightarrow \ccG \otimes_{\bullet} W$
 is denoted by the same symbol $\varphi$.
\end{notation}

Another important property is that the apolarity respects grading in the following sense:
\[
\hook \colon \ccO_Q \otimes_{\kk} S^{i} V^*
 \times \ccO_Q \otimes_{\kk} S^{(j)} V \to
\ccO_Q \otimes_{\kk} S^{(j-i)} V.
\]
In fact, it similarly preserves also non-standard gradings, see Lemma~\ref{lem_homogeneity_of_annihilator}.

\begin{example}
   If $Q=\Spec \kk[s,t]/(s^2, t^2)$, $d\ge 2$,
   \begin{align*}
     \Theta& = \alpha_1^2 + s\alpha_2\alpha_0, \text{ and} \\
     G &= x_0^{(d)}+ s x_0^{(d-1)}x_1 + t x_0^{(d-1)}x_2,\\
     \text{then } \Theta \hook G&= s t x_0^{(d-2)}.
   \end{align*}
   Note that if instead $Q=\Spec \kk[s,t]/(s^2, st, t^2)$,
   then the same $\Theta$ and $G$
   give $\Theta \hook G =0$.
\end{example}

Let
$\Upsilon^{i, d-i}\colon
\ccO_{\PP_{\kk}(S^{(d)}V)} \otimes_{\kk} S^i V^* \to
\ccO_{\PP_{\kk}(S^{(d)}V)}(1) \otimes_{\kk} S^{(d-i`)}V$ be  the catalecticant map as in Sections~\ref{sec_intro} and \ref{sec_cactus_scheme}.
If $Q=\Spec A$ is a finite scheme and $F\colon Q\to \PP_{\kk}(S^{(d)}V)$ is a morphism with a lifting $\hat{F}\colon Q\to S^{(d)}V$ and  the corresponding tensor $F_{\otimes}$ as in \S\ref{sec_finite_subschemes_as_tensors},
then $\Upsilon^{i, d-i}|_{Q}
\colon A \otimes_{\kk} S^i V^* \to
A \otimes_{\kk} S^{(d-i)}V$ is equal to:
\begin{equation}
  \Upsilon^{i, d-i}|_{Q} (\Theta) =
    u \cdot \Theta \hook F_{\otimes},
\end{equation}
where $u\in A$ is an invertible element responsible for the choices involved:
\begin{itemize}
 \item the choice of trivialization
    $\ccO_{\PP_{\kk}(S^{(d)}V)}(1)|_{Q} \simeq \ccO_Q$, and
 \item the choice of lifting
         $\hat{F}$.
\end{itemize}

\subsection{Definition of relative annihilator ideal sheaf}
Informally, similarly to the classically studied
non-relative case,
the annihilator of $F$ in the relative setting
is defined as all $\Theta$ such that $\Theta \hook F=0$.
We commence formalizing the definition, by restricting to an affine case, 
then proving it can be glued together to a quasi-coherent sheaf.

Whenever there is no risk of harm for the clarity, we state our definitions in a more general setting than needed just for the purpose of this paper.
We hope this will serve as a reference in further research.

\begin{definition}[Relative annihilator for affine $Q$]
\label{def_rel_ann_for_affine}
  Suppose $Q= \Spec A$ is an affine scheme and
    $F\colon Q\to \affinespaceof{S^{(\le d)}V} = \affinespaceof{\bigoplus_{j=0}^{d} S^{(j)}V}$ is a morphism.
  Then define
  \[
   \ann:=
      \set{\Theta\in A \otimes_{\kk} \Sym V^*\mid
      \Theta \hook F_{\otimes} =0}.
  \]
\end{definition}

We commence with the following semicontinuity observation:
\begin{lemma}[Semicontinuity of $\ann$ with respect to $Q$]\label{lem_semicontinuity of annihilator}
   Suppose $Q$ and $Q'$ are two affine schemes, and $G\colon Q'\to Q$,
    $F\colon Q\to \affinespaceof{S^{(\le d)}V}$ are two morphisms.
   Denote $F' = F \circ G\colon Q'\to \affinespaceof{S^{(\le d)}V}$.
   Then,
   \[
     G^*(\ann) \subset \ann[F'].
   \]
\end{lemma}
Note that by writting $G^*(\ann)$ in the statement of the lemma we exploit 
Notation~\ref{not_skip_id}.
Similarly, we use this notation in the proof.
\begin{prf}
   Let $Q= \Spec A$ and $Q'=\Spec A'$,
     and suppose
     $\Theta \in \ann \subset A\otimes_{\kk}\Sym V^*$.
   Then
   \[
     G^*(\Theta)\hook F'_{\otimes}
      = G^*(\Theta)\hook (F \circ G)_{\otimes} 
\stackrel{\text{by Prop.~\ref{prop_basic_properties_of_tensor_of_a_scheme}\ref{item_composition_in_terms_of_tensors}}}{=} G^*(\Theta)\hook
       (G^*F_{\otimes})
\stackrel{\text{by \eqref{equ_hook_preserved_by_pullback}}}{=}
  G^*(\Theta\hook F_{\otimes}) = G^*(0) =0.
  \]
  Therefore, $G^*(\ann) \subset \ann[F']$ as claimed.
\end{prf}

The following examples illustrate failures of ``continuity'' of $\ann$.
\begin{example}
   Let $Q_{k}= \Spec \kk[t]/(t^k)$,
   and $F_{k, \otimes} = x^d + ty^d + t^2 z^d$.
   Then
   \begin{align*}
      \ann[F_{k}] &
          = (t^{k-1}\beta, t^{k-2}\gamma,
                 \alpha\beta, \alpha\gamma, \beta\gamma,
                 \beta^d - t\alpha^d, \gamma^d - t^2 \alpha ^d,
                 \alpha^{d+1}
                ) &\text{ for } k\geqslant 2,\\
\text{and } \ann[F_1] &= (\beta,\gamma, \alpha^{d+1})
&\text{ for } k= 1.
\end{align*}
   It is straightforward to see that
      $\ann[F_{k}]|_{Q_{i}} \subsetneqq  \ann[F_{i}]$ for $0\leqslant i<k$.
   As one example,
      $\ann[F_{k}]|_{Q_1} =
          (\alpha\beta, \alpha\gamma, \beta\gamma,
                 \beta^d, \gamma^d, \alpha^{d+1}) \neq
                 \ann[F_1]$.
\end{example}

\begin{example}
   Consider $Q=\Spec \kk[t]/(t^2)$ and $Q'=\Spec \kk$ with the unique embedding $G\colon Q'\to Q$.
   Suppose $F_{gen}\in S^{(d)} V$ is a general element.
   Let
   $F_{\otimes}:=x_0^{(d)} + t F_{gen} \in \kk[t]/(t^2)\otimes_{\kk} S^{(d)} V$.
   Pick a homogeneous $\Theta + t \Theta'\in \kk[t]/(t^2)\otimes_{\kk} S^iV^*$.
   Then:
   \[
    (\Theta + t \Theta')\hook (x_0^{(d)} + t F_{gen}) =
    \Theta\hook x_0^{(d)} + t (\Theta \hook F_{gen} + \Theta'\hook x_0^{(d)}).
   \]
   Hence, $\Theta + t \Theta'\in \ann$ if and only if
   $\Theta \in \ann[x_0^{(d)}]$ and $\Theta \hook F_{gen}
   = -\Theta'\hook x_0^{(d)}$.
   We always have $-\Theta'\hook x_0^{(d)} =c \cdot x_0^{(d-i)}$ for some $c\in \kk$,
   but if $1\leqslant i \leqslant \frac{d}{2}$ then we never have
   $\Theta \hook F_{gen} \in \linspan{x_0^{(d-i)}}$ (unless $\Theta=0$).
   In particular, in low degrees,
    we must have strict inclusion
   $0=\ann_{i}|_{Q'} \subsetneqq  \ann[x_0^{(d)}]_i \neq 0$.
\end{example}

There are however two types of situations, where the continuity of $\ann$ holds, at least partially.
The first one is the case of open subsets (Proposition~\ref{prop_continuity_of_annihilator_for_open_subsets}), and this property will allow us to extend the definition of annihilator to a sheaf of ideals on any (not necessarily affine) scheme $Q$.
The second is the flat case,
when the rank of a catalecticant matrix is constant on $Q$, see Proposition~\ref{prop_continuity_of_Ann_for_flat_case}.

\begin{proposition}
\label{prop_continuity_of_annihilator_for_open_subsets}
    Suppose $Q=\Spec A$ is an affine scheme, $s\in A$,  and $U\simeq \Spec A_s \subset Q$ is an open subscheme with $U= Q \setminus \set{s=0}$.
    Denote by $G^* \colon A\to A_s$ the localization map.
    Then for any
       $F\colon Q\to \affinespaceof{S^{(\leqslant d)}V}$
     we have
    \[
       \ann[F|_U] = A_s\cdot G^*(\ann).
    \]
\end{proposition}
Again, we exploit Notation~\ref{not_skip_id} both in the statement and the proof of the proposition.
\begin{prf}
   The inclusion
   $A_s \cdot G^* (\ann) \subset \ann[F|_U]$
   follows from the general semicontinuity,
    see Lemma~\ref{lem_semicontinuity of annihilator}.
    Thus, to prove the opposite inclusion,
      pick any $\Theta'\in \ann[F|_U]
      \subset A_s \otimes_{\kk} \Sym V^*$.
    Therefore, there exists an integer $k\ge 0$
      such that $s^k \cdot \Theta'$ is in the image of
      $G^* \colon A \otimes_{\kk} \Sym V^*\to A_s \otimes_{\kk} \Sym V^*$,
      that is there exists $\Theta \in A \otimes_{\kk} \Sym V^*$ such that $G^*(\Theta) = s^k \cdot \Theta'$.
      Then:
    \begin{align*}
      G^* (\Theta\hook F_{\otimes})& \stackrel{\text{By \eqref{equ_hook_preserved_by_pullback}}}{=} G^*(\Theta)\hook (F|_{U})_{\otimes} = s^k \cdot \Theta' \hook (F|_{U})_{\otimes} =0, \text{that is}
      \\
      \Theta \hook F_{\otimes}&\in \ker (G^*\colon A \otimes_{\kk} S^{(\leqslant d)}V \to A_s \otimes_{\kk} S^{(\leqslant d)}V)
      \\
      &= \set{t \in A \mid t s^l =0 \text{ for some } l>0}\otimes_{\kk} S^{(\leqslant d)}V.
    \end{align*}
    Therefore, choosing a basis
      $\setfromto{f_0}{f_N}$ of $S^{(\leqslant d)}V$,
      we can write
      $
        \Theta \hook F_{\otimes} =
          \sum_{i=0}^{N} t_i f_i
      $
      with each $t_i\cdot s^{l_i} =0$ for some $l_i>0$.
      Let $l= \max(l_i \mid i \in \setfromto{0}{N})$.
      We have $s^l\cdot \Theta \hook F_{\otimes}=0$,
        that is $s^l\cdot \Theta \in \ann$.
      This proves the claim, as
      \[
         \Theta' = s^{-k-l}\cdot G^*(s^l \cdot \Theta) \in s^{-k-l}\cdot G^*(\ann).
      \]
\end{prf}

\begin{definition}
   Suppose $Q$ is a scheme and
   $F\colon Q \to \affinespaceof{S^{(\leqslant d)}V}$ is a morphism of schemes.
   We define
   $\ann \subset \ccO_Q \otimes_{\kk} \Sym V^*$
   to be the unique quasi-coherent sheaf of ideals
   such that for each open affine subset $U \subset Q$ the value of the sheaf
   $\ann(U)$ is the ideal $\ann[F|_U]$ as in Definition~\ref{def_rel_ann_for_affine}.
   Such ideal sheaf exists by Proposition~\ref{prop_continuity_of_annihilator_for_open_subsets}.
\end{definition}

Before proceeding we prove one more technical lemma about the annihilator in the case of finite local schemes. This lemma says that despite we do not have the continuity of annihilators in general, we can sometimes read the annihilator of a point from the annihilator of a family.

\begin{lemma}\label{lem_annihilator_of_a_point_from_socle}
   Suppose $Q = \Spec A$ is a finite local scheme over $\kk$ and $F\colon Q\to S^{(\le d)} V$ is a morphism.
   Let $F_{\kk}\colon \Spec \kk \to S^{(\le d)} V$ be the restriction of $F$ to the closed reduced point $\Spec \kk \in Q$ and let $s\in A$ be a non-zero element of the socle of $A$.
   Then the map of $\Sym V^*$-modules
   $\Sym V^* \to A\otimes_{\kk} \Sym V^*$
   taking $\Phi$ to $s\cdot \Phi$
   maps $\ann[F_{\kk}]$ isomorphically onto $\ann \cap s\cdot \Sym V^*$.
\end{lemma}

\begin{prf}
   Let $\gotm \subset A$ be the maximal ideal.
   We
   write $F_{\otimes} = F_{\kk} + F_{\gotm}$,
     where $F_{\gotm} \in \gotm \otimes_{\kk} \Sym V^*$.
   Take $\Phi \in \Sym V^*$.
   We have
   \[
    s\Phi \hook F_{\otimes} = s(\Phi\hook F_{\kk}) + \Phi \hook (s F_{\gotm}) = s(\Phi\hook F_{\kk}).
   \]
   Therefore, $\Phi  \in \ann[F_{\kk}]$
     if and only if $s\Phi \in \ann$ as claimed.
\end{prf}

\subsection{Descending the annihilator to projective space}

The goal of this section is to prove that we can almost equally well define the annihilator for morphisms $F\colon Q\to \PP_{\kk}(S^{(\leqslant d)}V)$.
To commence, we state the following lemma,
 whose proof and its consequence are straightforward,
 and we skip the details.

\begin{lemma}[Invariance of annihilator under rescalings]
\label{lem_invariance_of_ann_under_rescalings}
   Suppose $Q=\Spec A$ is an affine scheme,
   $F\colon Q\to \affinespaceof{S^{(\leqslant d)}V}$
    is a morphism, and $s\in A$ is any element.
   Let $sF \colon Q\to \affinespaceof{S^{(\leqslant d)}V}$ be the morphism, whose corresponding tensor is
   $s\cdot F_{\otimes}$.
   Then $\ann\subset \ann[sF]$.
   In particular, if $s$ is invertible, then $\ann[sF]= \ann$.
\end{lemma}

\begin{corollary}\label{cor_apolar_of_multiplied_F_is_the_same}
   Suppose $Q$ is a scheme
   and  $F, F'\colon Q\to \affinespaceof{S^{(\leqslant d)}V}$
   are two morphisms such that:
   \begin{itemize}
      \item $0\in\affinespaceof{S^{(\leqslant d)}V}$ is not in the image of $F$ or of $F'$, and
      \item the compositions of $F$ and $F'$ with
      $\pi \colon \affinespaceof{S^{(\leqslant d)}V} \setminus \set{0} \to \PP_{\kk}(S^{(\leqslant d)}V)$ are equal.
   \end{itemize}
   Then $\ann = \ann[F']$.
\end{corollary}

\begin{lemma}\label{lem_good_pullback_for_G_m_action}
   Suppose
   $F\colon Q\to \affinespaceof{S^{(\leqslant d)}V}$
   and set
   $F'\colon \mathbf{G}_{m, \kk} \times_{\kk} Q\to \affinespaceof{S^{(\leqslant d)}V}$
   to be the natural morphism coming from the action
   of $\mathbf{G}_{m, \kk}$ on $\affinespaceof{S^{(\leqslant d)}V}$ by rescaling.
   Let $G_1$ be any section $Q\to \mathbf{G}_{m, \kk} \times_{\kk} Q$, and $G_2$ be the projection $\mathbf{G}_{m, \kk} \times_{\kk} Q \to Q$.
   Then $G_1^*(\ann[F'])$ generates $\ann[F]$ and
    $G_2^*(\ann[F])$ generates $\ann[F']$.
\end{lemma}
\begin{prf}
   It is enough to check the statement
      for affine $Q=\Spec A$.
   Then $\mathbf{G}_{m, \kk} \times_{\kk} Q = \Spec A[t, t^{-1}]$ and $G_2^* \colon A \to A[t, t^{-1}]$ is the natural inclusion, while $G_1^*\colon A[t, t^{-1}] \to A$ is identity on $A$ and maps $t$ to an invertible element of $A$. We have $G_1^* \circ G_2^* = \id_A$.
   Moreover, using Notation~\ref{not_skip_id}, it is straightforward to check that
   \begin{align*}
     F'_{\otimes} & = t \cdot G_2^*(F_{\otimes}) \in A[t, t^{-1}]\otimes_{\kk} S^{(\leqslant d)}V, \text{ and}\\
     G_1^*(F'_{\otimes}) & = G_1^*(t) \cdot G_1^*G_2^*(F_{\otimes}) = G_1^*(t) \cdot F_{\otimes} \in A \otimes_{\kk} S^{(\leqslant d)}V.
   \end{align*}
   In particular, since both $t\in A[t, t^{-1}]$ and
   $G_1^*( t)\in A$ are invertible, we have:
   \[
     \ann[F\circ G_2] = \ann[F'] \text{ and }
     \ann[F'\circ G_1] = \ann
   \]
   by
   Lemma~\ref{lem_invariance_of_ann_under_rescalings}.
   Therefore,
   by Lemma~\ref{lem_semicontinuity of annihilator}
   we must have $G_2^*(\ann) \subset \ann[F']$ and
    $G_1^*(\ann[F']) \subset \ann$ and to prove the statement it is enough to check that:
   \begin{align*}
     \ann[F \circ G_2] & \subset (A[t, t^{-1}] \otimes_{\kk} \Sym V^*)\cdot G_2^*(\ann), \text{ and}\\
    \ann &\subset G_1^*(\ann[F']).
   \end{align*}

   Any $\Theta' \in A[t, t^{-1}] \otimes_{\kk} \Sym V^*$ can be expressed as
   $\Theta' = \sum_{i \in \ZZ} G_2^*(\Theta'_i) t^{i}$ with
   $\Theta'_i \in A\otimes_{\kk} \Sym V^*$ with only finitely many $\Theta'_i\ne 0$.
   Then $\Theta' \in \ann[F\circ G_2]$ if and only if
     $G_2^* \Theta'_i \hook G_2^* F_{\otimes} = G_2^* (\Theta'_i \hook F_{\otimes})=0$ for all $i$,
     which is only possible if $\Theta'_i \in \ann$, proving the first claim above.

   Now pick $\Theta \in \ann[F'\circ G_1] = \ann$ and consider $G_2^* \Theta$. By Lemma~\ref{lem_semicontinuity of annihilator} and
   Corollary~\ref{cor_apolar_of_multiplied_F_is_the_same} we have
   $G_2^* \Theta \in \ann[F\circ G_2] = \ann[F']$.
   Thus,
   \[
    \Theta \hook G_1^* F'_{\otimes}
      = G_1^* G_2^* \Theta \hook G_1^* F'_{\otimes}
      = G_1^* (G_2^* \Theta \hook F'_{\otimes}) = 0
   \]
   showing the second claim above, and concluding the proof of the lemma.
\end{prf}

We denote by $\pi\colon \left(\affinespaceof{S^{(\leqslant d)}V} \setminus\set{0}\right)\to \PP_{\kk}(S^{(\leqslant d)}V)$ the natural quotient map.

\begin{proposition}
   \label{prop_continuity_of_Ann_for_flat_case}
   Suppose $Q$ is a scheme and $F\colon Q\to \PP_{\kk}(S^{(\leqslant d)}V)$ is a morphism.
   Then there exists a unique quasicoherent sheaf of ideals $\ann\subset \ccO_Q\otimes_{\kk}\Sym V^*$
   such that for any open subset $U\subset Q$
   and any lift $\hat{F}_U\colon U \to \affinespaceof{S^{(\leqslant d)}V}$ such that $\pi\circ\hat{F}_U =F|_U$, the equality of ideals holds
   $\ann|_U = \ann[\hat{F}_U]$.
\end{proposition}
\begin{prf}
 Define $\hat{Q} := Q\times_{\PP_{\kk}\left(S^{(\leqslant d)}V\right)} \left(\affinespaceof{S^{(\leqslant d)}V} \setminus\set{0}\right)$, which is a $\mathbf{G}_{m, \kk}$-bundle over $Q$.
 Then, for $\hat{F}\colon \hat{Q} \to \affinespaceof{S^{(\leqslant d)}V}$
   the ideal sheaf $\ann[\hat{F}]$ is $\mathbf{G}_{m, \kk}$-invariant
   by Corollary~\ref{cor_apolar_of_multiplied_F_is_the_same}.
 Therefore $\ann[\hat{F}]$ descents to an ideal sheaf over the quotient $Q$. This descent is $\ann$ which we are looking for.

 If $U\subset Q$ lifts to the affine space then it also lifts to $\hat{Q}$ and the restriction of the $\mathbf{G}_{m, \kk}$-bundle $\hat{Q}\to Q$ to $U$ is trivial.
 $\ann|_U$ agrees with $\ann[\hat{F}_U]$ by Lemma~\ref{lem_good_pullback_for_G_m_action}.
 Uniqueness of $\ann$ follows from
 Corollary~\ref {cor_apolar_of_multiplied_F_is_the_same}
 as locally any two liftings $U \to \affinespaceof{S^{(\leqslant d)}V}$ will lead to the same ideal.
\end{prf}

\begin{lemma}
    \label{lem_no_constants_in_the_annihilator}
    Suppose $Q=\Spec A$ is a local scheme and $F\colon Q\to \PP_{\kk}(S^{(\leqslant d)}V)$ is a morphism.
    Then $\ann_0=0$.
\end{lemma}
\begin{prf}
    Suppose $a\in A$ and $a\in \ann$, so that
    $a \hook F_{\otimes}= aF_{\otimes} =0$.
    But $F_\otimes \modulo \gotm \neq 0$ by Proposition~\ref{prop_residue_and_derivative_of_tensor_corresponding_to_scheme}
    and Lemma~\ref{lem_lifting_finite_schemes_from_projective_spaces}.
    This is only possible if $a=0$.
\end{prf}

\subsection{Equivariant behaviour of apolarity and a graded case}

Suppose $\psi\colon \Sym V^* \to \Sym V^*$ is a $\kk$-algebra
 automorphism,
 and denote by $\overline{\psi}$ the $\kk$-linear dual automorphism:
\[
\overline{\psi}\colon \prod_{i=0}^{\infty}S^{(i)} V \to \prod_{i=0}^{\infty}S^{(i)} V,
\text{ with }
  \overline{\psi}(F)(\Theta) =  F( \psi^{-1}(\Theta)),
\]
taking into account that
$\prod_{i=0}^{\infty}S^{(i)} V = \Hom_{\kk}(\Sym V^*, \kk)$.
Since $\psi$ preserves the algebra structure of $\Sym V^*$, the automorphism $\overline{\psi}$ preserves the $\Sym V^*$-module structure on $\prod_{i=0}^{\infty}S^{(i)} V$ given by $\hook$:
\[
   \overline{\psi}(\Theta \hook F) = (\psi(\Theta) \hook \overline{\psi}(F)).
\]
Assume in addition that $\psi$ preserves the maximal ideal $\gotn=\bigoplus_{i=1}^{\infty} S^i V^*$, so that $\psi(\gotn^i) = \gotn^i$ for all $i$.
Since $(\gotn^i)^{\perp} = S^{(\le i-1)} V$, for such $\psi$ the dual map $\overline{\psi}$ preserves the space of polynomials $\bigoplus_{i=0}^{\infty}S^{(i)} V$. By a slight abuse of notation, we will denote the restricted automorphism of $\bigoplus_{i=0}^{\infty}S^{(i)} V$ also by $\overline{\psi}$.

We also want to keep exploiting Notation~\ref{not_skip_id}. Note that so far we were using it with respect to tensorising on the second factor, skipping for instance $\otimes \id_{\Sym V^*}$ or $\otimes \id_{S^{(\leqslant d)} V}$. Below we will use it also to skip the identity on the first factor, $\id_{\ccO_Q}
\otimes $ or $\id_{A}
\otimes $, by considering $\psi(\ann)$ or $\overline{\psi}(F_{\otimes})$ and similar.

\begin{lemma}[Equivariance of annihilator]
   Suppose $\psi$ is a $\kk$-algebra automorphism of $\Sym V^*$ preserving $\gotn$.
   Let $Q$ be a scheme and $F\colon Q\to \affinespaceof{S^{(\leqslant d)}V}$ a morphism of schemes.
   Then, $\ann[\overline{\psi}\circ F] = \psi(\ann)$.
\end{lemma}
\begin{proof}
  Without loss of generality, suppose $Q = \Spec A$ is affine.
  Assume $\Theta \in A\otimes_{\kk} \Sym V^*$. Then, 
  \begin{align*}
     \Theta \in\ann[\overline{\psi}\circ F]
     \iff & \Theta \hook (\overline{\psi}\circ F)_{\otimes} =0
     \iff \overline{\psi}(\psi^{-1}(\Theta) \hook F_{\otimes})  =0 \iff \\
     \psi^{-1}(\Theta) \hook F_{\otimes}  =0
     \iff &\psi^{-1}(\Theta) \in\ann
     \iff \Theta \in\psi(\ann).
  \end{align*}
\end{proof}

  Suppose now that $\Sym V^*$ is a positively graded (in the sense of \cite[p.~726]{haiman_sturmfels_multigraded_Hilb}) by a finitely generated abelian group $H$, and let $H^* = \Hom(H, \mathbf{G}_{m, \kk})$.
   Then the action of $H^*$ on $\Sym V^*$ preserves the maximal ideal $\gotn$ and $H^*$ acts on $\affinespaceof{S^{(\leqslant d)}V}$ for each $d$.
\begin{lemma}[Homogeneity of annihilator of a homogeneous $F$]
\label{lem_homogeneity_of_annihilator}
   If $W\subset S^{(\le d)} V$ is a weight space of the action of $H^*$, and $F\colon Q \to \affinespaceof{W}$ is a morphism of schemes,
   then $\Ann(F)$ is an $H$-homogeneous ideal sheaf in $\ccO_Q\otimes_{\kk} \Sym V^*$, where the grading of $\ccO_Q$ is trivial.
\end{lemma}
\noprf

From now on we will frequently consider the \emph{apolar algebra} of a morphism $F\colon Q\to \affinespaceof{S^{(\le d)} V}$:
let $\apolar:= \ccO_Q \otimes_{\kk} \Sym V^*/\ann$. 
Thus, $\apolar$ is a sheaf of $\ccO_Q$-algebras,
 which are all finite modules over $\ccO_Q$.
In case $Q=\Spec A$ is affine, we will simply consider $\apolar$ to be the algerba of global sections.

In particular, the special case $Q=\Spec \kk$ is vastly and classically explored in the literature, as $\apolar$ is a finite (over $\kk$) local Gorenstein algebra and all such algebras arise in this way, see for instance \cite[\S2.1, 2.3, 3.3--3.5]{jelisiejew_PhD} (and references therein) for a review of properties of these algebras.
Perhaps the most widely used and known property is that for homogeneous $F$ the apolar algebra is self dual, with a shift in gradings.
We show that an analogous duality holds more generally, whenever the base $Q$ is itself a finite local Gorenstein scheme.

\begin{lemma}[Duality]\label{lem_duality_of_apolar_algebra}
  Suppose $W\subset S^{(\le d)} V$ is a weight space of the action of $H^*$ of degree $D \in H$.
  Assume $Q=\Spec A$ is a finite local Gorenstein scheme over $\kk$,
  and $F\colon Q \to \affinespaceof{W}$ is a morphism of schemes.
  Then for all degrees $I \in H$ the following $A$-modules are isomorphic:
  \[
     \apolar_{I} = \Hom_A\left(\apolar_{D-I}, A\right).
  \]
\end{lemma}

\begin{proof}
      Consider the map $A\otimes_{\kk} \Sym V^* \to A\otimes_{\kk} S^{(\le d)} V$ given by $\Theta \mapsto \Theta \hook F_{\otimes}$.
      Its kernel is $\ann$, so that
      $A\otimes_{\kk} \Sym V^* \hook F_{\otimes}= A\otimes_{\kk} \Sym V^* / \ann
      = \apolar$ (as $A$-modules).
      Moreover, the map reverses the grading:
      \begin{equation}
         \label{equ_grading_reversed}
       \apolar_I
       = \left(A\otimes_{\kk} \Sym V^* \hook F_{\otimes} \right)_{D-I}
      \end{equation}
      Next, we verify that the following containments of $A$-submodules  (respectively, of $A\otimes_{\kk}(S^{\leqslant d} V)_{D-I}$ and of $A\otimes_{\kk}(\Sym V^*)_{D-I}$) hold:
      \begin{equation}\label{equ_apolar_inclusions_and_perps}
      \begin{aligned}
         \left(A\otimes_{\kk} \Sym V^* \hook F_{\otimes} \right)_{D-I} & \subset \ann_{D-I}^{\perp} \text{ and}
         \\
         \left(A\otimes_{\kk} \Sym V^* \hook F_{\otimes} \right)_{D-I}^{\perp} & \subset \ann_{D-I}.
      \end{aligned}
      \end{equation}
      Indeed, suppose
         $G \in \left(A\otimes_{\kk} \Sym V^* \hook F_{\otimes} \right)_{D-I}$,
         that is $G = \Theta \hook F_{\otimes}$ for some homogeneous form
         $\Theta \in \left(A\otimes_{\kk} \Sym V^*  \right)_{I}$.
      Then, for any $\Psi\in \ann_{D-I}$,
      \[
        \Psi(G) = \Psi\hook G =\Psi\hook (\Theta \hook F_{\otimes}) = (\Psi\cdot \Theta) \hook F_{\otimes} = \Theta \hook(\Psi \hook F_{\otimes})=0,
      \]
      so $G\in \ann_{D-I}^{\perp}$ as claimed in the first inclusion of \eqref{equ_apolar_inclusions_and_perps}.

      Now let $\Theta\in\left(A\otimes_{\kk}
        \Sym V^* \hook F_{\otimes} \right)_{D-I}^{\perp}$.
      Then for all $\Psi\in \left(A\otimes_{\kk} \Sym V^*  \right)_{I} $
        we have:
      \[
         0= \Theta(\Psi\hook F_{\otimes})
          = \Psi(\Theta \hook F_{\otimes}),
      \]
      so that $\Theta \hook F_{\otimes} 
      \in \left(A\otimes_{\kk} \Sym V^*  \right)_{I}^{\perp}$.
      Since $\Theta \hook F_{\otimes}$ is in the weight space of degree $I$,
        it follows that $\Theta \hook F_{\otimes}=0$, proving the second inclusion of \eqref{equ_apolar_inclusions_and_perps}.

      Now we combine the first inclusion of
       \eqref{equ_apolar_inclusions_and_perps}
       with the the perpendicular parts of the second one.
      By \eqref{equ_double_perp_and_reversing_inclussion} we obtain:
      \[
         \left(A\otimes_{\kk} \Sym V^* \hook F_{\otimes} \right)_{D-I}  \subset \ann_{D-I}^{\perp}
         \subset \left(A\otimes_{\kk} \Sym V^* \hook F_{\otimes} \right)_{D-I}^{\perp\perp}.
      \]

     To conclude, we must use the hypothesis
       on $Q$ or, more precisely, on $A$, that they finite local Gorenstein.
     We have to use
       Proposition~\ref{prop_perp_of_perp_for_Gorenstein}.
     The two containments \eqref{equ_apolar_inclusions_and_perps} imply:
      \begin{align*}
         \ann_{D-I}^{\perp} & \stackrel{\text{\ref{prop_perp_of_perp_for_Gorenstein}\ref{item_perp_of_perp_for_Gorenstein}}}{=} \left(A\otimes_{\kk} \Sym V^* \hook F_{\otimes} \right)_{D-I} \stackrel{\text{\eqref{equ_grading_reversed}}}{=} \apolar_{I}.
      \end{align*}
      To conclude, apply Proposition~\ref{prop_perp_of_perp_for_Gorenstein} to $M=\ann_{D-I}^{\perp}$:
      \begin{align*}
        \apolar_{I} &= \ann_{D-I}^{\perp}
        \stackrel{\text{\ref{prop_perp_of_perp_for_Gorenstein}\ref{item_dual_of_submodule_versus_perp}}}{=} \Hom\left(\bigl(A \otimes_{\kk} \Sym V^* / \ann^{\perp\perp}\bigr)_{D-I}, A\right)
        \\
        &\stackrel{\text{\ref{prop_perp_of_perp_for_Gorenstein}\ref{item_perp_of_perp_for_Gorenstein}}}{=} \Hom\left(\bigl(A \otimes_{\kk} \Sym V^* / \ann\bigr)_{D-I}, A\right)
        = \Hom\left(\apolar_{D-I}, A\right).
      \end{align*}
\end{proof}

The proof presented above is perhaps not the most direct one.
Instead, one could exploit from the very beginning the presentation of $A$ as a quotient of a polynomial ring by the annihilator of a polynomial in a dual ring. Then present $\apolar$ in the same way, but using polynomial rings with more variables, and with a grading in which some variables have degree $0$.
Then the statement would follow from the classical Gorenstein duality.
Our proof has the advantage that a large part of it holds for any base algebra.
The finite Gorenstein assumption is only required for the implication of Proposition~\ref{prop_perp_of_perp_for_Gorenstein}.
Even though the following example shows this assumption is necessary to get the conclusion as in the lemma,
we believe that the partial results
(such as inclusions of \eqref{equ_apolar_inclusions_and_perps}) have a potential to be applied in further research.

\begin{example}
   Suppose $A=\kk[s,t]/(s^2, st,t^2)$
     and define $F_{\otimes}= s x^{(d)} + t y^{(d)}\in A\otimes_{\kk} S^{(d)}V$ for a two dimensional $V$ spanned by $x, y$.
   Then $\ann = (s,t, \alpha \beta, \alpha^{d+1}, \beta^{d+1})$.
   In particular, $A\otimes_{\kk} \Sym V^*/\ann$
     in degree $0$ is isomorphic to $\kk$,
     while in degree $d$ is isomorphic to $\kk^{2}$.
   Therefore, unlike in the Gorenstein case,
     the two gradings fail to be dual to one another.
\end{example}

\subsection{Apolarity lemma}

We conlude this section with a proof of relative apolarity lemma.
In the classical case, the apolarity lemma is a foundation for many research projects,
such as \cite{ranestad_schreyer_VSP},
\cite{nisiabu_jabu_border_apolarity},
\cite{ranestad_voisin_VSP_and_divisors_in_the_moduli_of_cubic_fourfolds}, \cite{bernardi_taufer_Waring_tangential_cactus_decompositions}, to mention just a few.
Informally, apolarity lemma appears whenever one needs set-theoretic membership tests for points of secant or cactus varieties, or recover Waring or cactus decomposition.
The ultimate goal of the relative apolarity lemma is to upgrade the set-theoretic tests and equations to scheme-theoretic ones.

In this article, from now on we will only consider cases homogeneous with respect to the standard grading.
That is, we will consider only maps of schemes $F\colon Q \to \affinespaceof{S^{(d)} V}$ (without lower degree terms)
or its projective variant $F\colon Q\to \PP_{\kk}(S^{(d)} V)$.
Some more general versions of classical apolarity lemma have been studied,
 see \cite{galazka_mgr_publ} and references therein,
 and most likely they generalize to the relative setting analogous to
 the setting of this paper, but proving it would require to dive into the setting of toric varieties and their Cox rings, which is outside of the applications we have in mind here.

Throughout this subsection $Q$ is a scheme that has enough of $\kk$-points. Frequently, either in the statements or in the proofs, we will restrict to the local case, so that $Q=\Spec A$ for a local $\kk$-algebra $A$ with maximal ideal $\gotm$ and residue field
$A/\gotm = \kk$.
Recall, that $\nu_d$ denotes the Veronese map.
Similarly to Notation~\ref{not_skip_id},
by a  minor abuse  we will denote also by $\nu_d$ (instead of $\id_Q\times \nu_d$) the Veronese map $\PP_{Q} V
\to  \PP_{Q}(S^{(d)}V)$.

\begin{theorem}[Relative apolarity lemma]
   \label{thm_relative_apolarity_lemma}
   Suppose $F\colon Q \to \PP_{\kk}(S^{(d)}V)$ and
   $\ccR \subset \PP_{Q} V$ is a family of subschemes.
   Let $\overline{F} \subset \PP_{Q}(S^{(d)}V)$
   be the graph of $F$.
   Then
   $\overline{F} \subset \linspan{\nu_d(\ccR)}$ if and only if $\ccI(\ccR) \subset \ann$.
\end{theorem}

We state two lemmas.
The first one in the classical case is standard, see for instance \cite[Lem.~2.15]{iarrobino_kanev_book_Gorenstein_algebras},
or \cite[Prop.~3.4(ii)\&(iii)]{nisiabu_jabu_cactus}
and the proof in the relative case is no different.
We repeat the argument for the sake of completeness.

\begin{lemma}\label{lem_ideals_contained_in_annihilators}
   Suppose $F\colon Q \to \affinespaceof{S^{(d)}V}$ and $\ccI \subset \ccO_Q \otimes \Sym V^*$ is a homogeneous ideal sheaf.
   Then for any integers $i$ and $j$ such that
    $i \leqslant j \leqslant d$, if $\ccI_j \subset \ann_j$, then $\ccI_i \subset \ann_i$.
    In particular, if $\ccI_d \subset \ann_d$, then $\ccI \subset \ann$.
\end{lemma}
\begin{prf}
  The proof relies on the following property:
  \begin{equation}\label{equ_if_derivatives_zero_then_polynomial_zero}
     \text{if }
     F'_{\otimes}\in \ccO_Q\otimes S^{(i)}V
   \text{ is such that } S^{i}V^*\hook F'_{\otimes} \equiv 0, \text{ then } F'_{\otimes} = 0.
  \end{equation}

  First we prove the latter claim: if $\ccI_d \subset \ann_d$, then $\ccI \subset \ann$.
  Since $\ccI$ is homogeneous,
    it is enough to check the claim in each degree $k$.
   If $k>d$, then $\ann_k= \ccO_Q \otimes_{\kk} S^k V^*$ and there is nothing to check.
   If $k=d$, then the assumption and the claim coincide.
   If $k<d$ then $\Phi\in \ccI_{k}$ implies  $S^{d-k}V^*\cdot \Phi \subset \ccI_{d} \subset \ann_{d}$.
   Hence $S^{d-k}V^* \hook (\Phi \hook F_{\otimes}) =0$, and by \eqref{equ_if_derivatives_zero_then_polynomial_zero}  $\Phi \hook F_{\otimes} =0$, that is $\Phi \in \ann$ as claimed.

   In general, in order to prove the former claim of the lemma, it is enough to assume $\ccI$ is generated in degrees up to $j$, since $\ccI_i$ and $\ccI_j$ is not affected by the higher degree generators.
   Then $\ccI_d\subset \ann_d$ and we conclude that $\ccI\subset \ann$, in particular $\ccI_i\subset \ann_i$.
\end{prf}

\begin{lemma}
   \label{lem_containment_in_linear_span_in_terms_of_ideal}
   Suppose the notation for $F, \overline{F}, \ccR$
    is as in  Theorem~\ref{thm_relative_apolarity_lemma}
    and that the base scheme $Q=\Spec A$ is local.
   Then,
   $\overline{F} \subset \linspan{\nu_d(\ccR)}$ if and only if
   \begin{itemize}
     \item $I(\ccR)_0=0$, and
     \item for all $\Phi \in I(\ccR)_d$
           we have $\Phi(F_{\otimes}) =0$.
   \end{itemize}
\end{lemma}

\begin{prf}
 By definition,
   the linear span $\linspan{\nu_d(\ccR)}$
   is the subscheme of $\PP_A(S^{(d)}V)$
   defined by the ideal generated by
   $I(\ccR)_0+I(\nu_d(\ccR))_1
   \subset A\otimes_{\kk} \Sym (S^d V^*)$.
   Note that $I(\nu_d(\ccR))_1=I(\ccR)_d\subset S^d V^*$.
 Consider the affine subscheme
   $\ccW \subset \affinespaceof[A]{S^{(d)}V}$ defined by the same ideal in $\Sym (S^d V)$,
   that is the ideal generated by $I(\ccR)_0+I(\ccR)_d$.

Pick the lift of $F$ to the affine space,
    $\hat{F} \colon Q\to \affinespaceof{S^{(d)}V}$,
    that corresponds to the tensor $F_{\otimes}$.
Note that since $Q$ is local and the image $\hat{F}(Q)$ avoids the origin,
   we have $\overline{F} \subset \linspan{\nu_d(\ccR)}$ if and only if the graph of $\hat F$ is contained in $\ccW$.
The latter holds if and only if the generators of the ideal of $\ccW$ are all in the ideal of the graph of $\hat F$.
That is, by Proposition~\ref{prop_basic_properties_of_tensor_of_a_scheme}\ref{item_ideal_of_graph_in_terms_of_tensor},
this happens if and only if the equations of
 $I(\ccR)_0$ and of  $I(\ccR)_d$ at $F_{\otimes}$ are zero.
 For $I(\ccR)_d$ this is precisely the second item. But $I(\ccR)_0 \subset A$, and the evaluation does not do anything. So the condition for vanishing evaluations of $I(\ccR)_0$ is just $I(\ccR)_0=0$, as claimed in the first item of the lemma, concluding the proof.
\end{prf}

\begin{prf}[ of Theorem~\ref{thm_relative_apolarity_lemma}]
   First we note,
      that since $Q$ has enough of $\kk$-points,
      it is enough to prove the equivalence for local $Q=\Spec A$.
  Indeed, there are four objects involved in the equivalence:
  two of them are sheaves, and the other two are schemes.
  Replacing $Q$ with  local schemes around a closed points
   is standard for sheaves.
  For $\overline{F} \subset \linspan{\nu_d(\ccR)}$ it is also enough
   to check the inclusion locally.
  For $\linspan{\nu_d(\ccR)}$, the restriction of the linear span behaves continuously by Lemma~\ref{lem_flat_base_change_of_linear_spans}.

  By Lemma~\ref{lem_containment_in_linear_span_in_terms_of_ideal} we have to prove that
  $I(\ccR) \subset \ann$ if and only if:
   \begin{itemize}
     \item $I(\ccR)_0=0$, and
     \item for all $\Phi \in I(\ccR)_d$
           we have $\Phi(F_{\otimes}) =0$.
   \end{itemize}
   Since $F_{\otimes}$ is a lift from the projective space, it follows that $F_{\otimes} \modulo \gotm \ne 0$ (where $\gotm \subset A$ is the maximal ideal).
   Thus, by Proposition~\ref{prop_perp_of_perp_for_section}
   for $W= S^{(d)}V$ and $M\subset \freemodule[A]{W}$
    being the module generated by $F_{\otimes}$, we have
   $\Phi(F_{\otimes}) =0$ if and only if
   $\Phi \in M^{\perp} = \ann_d$.
   That is, the two items above are equivalent to
   \begin{itemize}
     \item $I(\ccR)_0=0$, and
     \item $I(\ccR)_d \subset \ann_d$.
   \end{itemize}

   By Lemma~\ref{lem_ideals_contained_in_annihilators}     the second item implies that $I(\ccR) \subset \ann$ as claimed.
   On the other hand if $I(\ccR) \subset \ann$, then the second item holds automatically,
   and $\ann_0=0$ by Lemma~\ref{lem_no_constants_in_the_annihilator} proving also the first item.
\end{prf}

\begin{proposition}\label{prop_continuity_for_flat_case}
   Suppose $F\colon Q \to \PP_{\kk}(S^{(d)}V)$ is a morphism such that its image is contained in $\set{\rk \Upsilon^{i, d-i}=r}$ for some $i$ and some $r$.
   Then $\ann_i$ is a locally free sheaf on $Q$.
   Moreover, for any morphism of schemes 
   $G\colon Q'\to Q$ if we denote by $F'$ the composition $F \circ G$, then $\ann[F']_i = G^*\ann_i$.
\end{proposition}

\begin{prf}
  To prove the first claim, covering $Q$ by sufficiently small open affine subsets, by Proposition~\ref{prop_continuity_of_annihilator_for_open_subsets}
  it suffices to prove it for $Q=\Spec A$  affine, 
    and for $F$ such that lifts to a morphism
    $\hat{F} \colon Q\to \affinespaceof[\kk]{S^{(d)} V}$.
  Then, the map
  \[
    \cdot \hook F_{\otimes}\colon A\otimes_{\kk} S^i V^* \to A\otimes_{\kk} S^{(d-i)}V
  \]
 is a homomorphism of $A$-modules that geometrically corresponds
    to the morphism of sheaves
  $F^*\Upsilon^{i, d-i} \colon \ccO_Q \otimes_{\kk} S^i V^* \to \ccO_Q(1) \otimes_{\kk} S^{(d-i)}V$.
  Note that since $F$ lifts to the affine space,
  $\ccO_Q(1) \simeq \ccO_Q$.
  Thus, $\ann_i = \ker F^*\Upsilon^{i, d-i}$
   and the locally free claim for $\ann_i$
   follows from the constant rank assumption for $F^*\Upsilon^{i, d-i}$.

  The second claim follows from the first as pulling back short exact sequences of locally free sheaves is exact.
\end{prf}

We note that Example~\ref{ex_locally_free_but_not_constant_rank_positive_dim} shows that the inverse of this statement is not true in general, even if $Q$ is local:
locally free graded part of the annihilator does not necessarily imply constant rank of catalecticant.
This observation and its vast generalizations lay out the foundations of border apolarity, 
\cite{nisiabu_jabu_border_apolarity}.

\begin{example}\label{ex_locally_free_but_not_constant_rank_positive_dim}
   Suppose $A = \kk[t]$ or $\kk[[t]]$, $Q=\Spec A$, and $F_{\otimes} = x^{(d)} + t y^{(d)}$.
   Then $\ann= (\alpha \beta, t\alpha^d-\beta^d, \alpha^{d+1})$ and $\ann_i$ is a free $A$-module for each $i$.
   However, $F(Q)$ is not contained in $\set{\rk\Upsilon^{i, d-i} =2}$ for any $1 \leqslant i \leqslant d-1$.
\end{example}
%
%
%

\section{Cactus scheme of Veronese varieties}\label{sec_cactus_of_Veronese}

In this concluding section we prove our main results:
about the structure of catalecticant minors (Theorem~\ref{thm_Upsilon_contained_in_cactus})
and about the comparison of the structure of the cactus scheme to the structure of Gorenstein Hilbert scheme of points (Theorem~\ref{thm_universal_linear_span_isomorphic_to_cactus}).
Combining these with earlier propositions and theorems leads to proofs of Theorems~\ref{thm_cactus_scheme_and_minors_for_PV} and \ref{thm_universal_linear_span_isomorphic_to_cactus_intro} stated in the introduction.

\subsection{Catalecticant minors produce flat families}

\begin{theorem}\label{thm_Upsilon_contained_in_cactus}
   Fix any positive integer $r$ and a $\kk$-vector space $V$.
   Then for all $d$ and $i$ such that $r \leqslant i \leqslant d-r$ (in particular $d\geqslant 2r$) we have
   \[
     \set{\rk \Upsilon^{i,d-i} =r}\subset
     \cactusflat{r}{\nu_d(\PP_{\kk} V)} \setminus \cactusflat{r-1}{\nu_d(\PP_{\kk} V)}.
   \]
\end{theorem}

A key technical step in the proof of the theorem is the following lemma.

\begin{lemma}\label{lem_Q_in_Upsilon_lifts_to_Hilb}
   Assume positive integers $r, d, i$ satisfy
   $r \leqslant i \leqslant d-r$
   and that
   $Q$ is a finite local Gorenstein scheme  with a morphism
   \[
     F\colon Q \to \set{\rk \Upsilon^{i,d-i} = r} \subset \PP_{\kk} (S^{(d)} V).
   \]
   Then there exists a flat \familiar{} $\ccR\to Q$ with $\ccR \subset \PP_{\kk} V\times Q$ such that $F(Q)\subset \rellinspan{\nu_d(\ccR)}$.
\end{lemma}

\begin{prf}
  Without loss of generality we may assume $i\leqslant d-i$. Suppose $Q = \Spec A$ and
  consider $\ann \subset A\otimes_{\kk} \Sym V^*$.
  Note that $\ann_i$
  is locally free (hence flat) over $A$ of rank $\dim_{\kk} S^iV^* - r$ by Proposition~\ref{prop_continuity_for_flat_case}.
  Therefore both $(A\otimes_{\kk} \Sym V^*/ \ann)_{i}$ and $(A\otimes_{\kk} \Sym V^*/ \ann)_{d-i}$ are locally free  over $A$ and of rank $r$ by Lemma~\ref{lem_duality_of_apolar_algebra}.
  By the Macaulay bound from Proposition~\ref{prop_classical_Macaulay_and_Gotzmann}
  \[
     \dim_{\kk}(\Sym V^*/ \ann_{\kk})_{k} \leqslant r \text{ and }
     \dim_{\kk}(\Sym V^*/ \ann_{\gots})_{k} \leqslant r
  \]
  for all $k > i$.
  Since $\Sym V^*/ \ann_{\gots} = \Sym V^*/ \ann[F_{\kk}]$
 by Lemma~\ref{lem_annihilator_of_a_point_from_socle}, and
  by the
  duality of the Hilbert function of finite graded Gorenstein algebras, 
  we also have
    \[
      \dim_{\kk}(\Sym V^*/ \ann_{\gots})_{k} \leqslant r \text{ for all }k.
    \]
  Thus, by Gotzmann's Persistence from Proposition~\ref{prop_classical_Macaulay_and_Gotzmann}\ref{item_classical_Gotzmann} (and the duality), for all $r-1 \leqslant k
  \leqslant d-r+1$ one has
  $\dim_{\kk}(\Sym V^*/ \ann_{\gots})_{k} = r$.

  Since $\ann_{\kk} \subset \ann_{\gots}$,
  we must also have
  $\dim_{\kk}(\Sym V^*/ \ann_{\kk})_{k} = r$ for $i\leqslant k \leqslant d-r+1$
  and thus in the same range
  $(\ann_{\kk})_{k}=(\ann_{\gots})_{k}$ and $\ann_{k}$ is a free $A$-module by Lemma~\ref{lem_criterion_for_free_submodule}.

  By Proposition~\ref{prop_relative_Macaulay_and_Gotzmann}
  the ideal $J$ generated by $(\ann_{\kk})_{\leq i}$ defines $\ccR \to Q$, a flat \familiar{} with $\ccR\subset \PP_A V$.
  Moreover $(J^{\sat})_i = J_i$
  and furthermore $J^{\sat} \subset \ann$
  by Lemma~\ref{lem_ideals_contained_in_annihilators}.
  Thus, by Theorem~\ref{thm_relative_apolarity_lemma} 
   (the relative apolarity lemma), 
   $F(Q)\subset \rellinspan{\nu_d(\ccR)}$
   as claimed.
\end{prf}

\begin{prf}[ of Theorem~\ref{thm_Upsilon_contained_in_cactus}]
Suppose $Q = \Spec A \subset \set{\rk \Upsilon^{i,d-i}=r}$
  is any finite local Gorenstein subscheme.
By Lemma~\ref{lem_Q_in_Upsilon_lifts_to_Hilb} and the definition of the cactus scheme
we have $Q \subset \cactusflat{r}{\nu_d(\PP_{\kk} V)}$.
By Proposition~\ref{prop_scheme_inclusion_equivalence},
 we also must have
  \[
    \set{\rk \Upsilon^{i,d-i}=r}
\subset
\cactusflat{r}{\nu_d(\PP_{\kk} V)}.
  \]

  It remains to check that for any
  closed $\kk$-point $F\in \set{\rk \Upsilon^{i,d-i}=r}$
  we cannot have $F\in \cactusflat{r-1}{\nu_d(\PP_{\kk} V)}$.
Indeed, if $F\in \cactusflat{r-1}{\nu_d(\PP_{\kk} V)}$, then  $F\in     \set{\rk \Upsilon^{i,d-i}\leqslant r-1}
$ by Theorem~\ref{thm_cactus_scheme_contained_in_vanishing_of_catalecticant_minors}, a contradiction.
\end{prf}

Recall that the statement of Theorem~\ref{thm_cactus_scheme_and_minors_for_PV} is about inclusions of schemes:
\[
   \cactusflat{r}{\nu_d(\PP_{\kk} V)} \subset \cactussch{r}{\nu_d(\PP_{\kk} V)} \subset \Upsilon^{i,d-i}_r(\PP_{\kk} V)
\]
for $r-1\leqslant i\leqslant d-1$ and
\[
 \Upsilon^{i,d-i}_r(\PP_{\kk} V) \setminus \Upsilon^{i,d-i}_{r-1}(\PP_{\kk} V)\subset \cactusflat{r}{\nu_d(\PP_{\kk} V)} \setminus \cactusflat{r-1}{\nu_d(\PP_{\kk} V)}
\]
for $r\leqslant i\leqslant d-r$.

\begin{prf}[ of Theorem~\ref{thm_cactus_scheme_and_minors_for_PV}]
  The first inclusion
  $\cactusflat{r}{\nu_d(\PP_{\kk} V)} \subset \cactussch{r}{\nu_d(\PP_{\kk} V)}$  holds by Proposition~\ref{prop_basic_properties_of_cactus_schemes}\ref{item_cactus_flat_in_cacactus_scheme_in_linspan}.
  The second one
  $\cactussch{r}{\nu_d(\PP_{\kk} V)} \subset \Upsilon^{i,d-i}_r(\PP_{\kk} V)$
   for $ r-1 \leqslant i \leqslant d-1$
   is proven in
   Theorem~\ref{thm_cactus_scheme_contained_in_vanishing_of_catalecticant_minors}.
  These two inclusions prove Item~\ref{item_inclusion_of_cactus_in_catalecticant}.

  Note that for $d\ge 2r$ and $ r-1 \leqslant i \leqslant d-r+1$ the supports of
  $\cactusflat{r-1}{\nu_d(\PP_{\kk} V)}$,
  $\cactussch{r-1}{\nu_d(\PP_{\kk} V)}$, and
  $\Upsilon^{i,d-i}_{r-1}(\PP_{\kk} V)$ coincide by
  Proposition~\ref{prop_cactus_scheme_supported_at_cactus_variety} and \cite[Thm~1.5]{nisiabu_jabu_cactus}, or
  \cite[Thm~7.5]{jabu_jelisiejew_finite_schemes_and_secants}.
  Thus, to conclude the proof of \ref{item_equality_of_cactus_in_catalecticant}
 it is enough to observe that the third inclusion of schemes,
  \[
 \Upsilon^{i,d-i}_r(\PP_{\kk} V) \setminus \Upsilon^{i,d-i}_{r-1}(\PP_{\kk} V)\subset \cactusflat{r}{\nu_d(\PP_{\kk} V)} \setminus \cactusflat{r-1}{\nu_d(\PP_{\kk} V)}
\]
  for $r\leqslant i\leqslant d-r$ follows from Theorem~\ref{thm_universal_linear_span_isomorphic_to_cactus}.
\end{prf}

\subsection{Relation to Hilbert scheme}

Consider
$\ccH:=\Hilb_r{\PP_{\kk} V}$, the Hilbert scheme of finite subschemes of $\PP_{\kk}V$ of length $r$.
As in Section~\ref{sec_cactus_scheme} define
 $\ccR^{\univ} \subset \ccH \times \PP_{\kk} V$
 to be the universal flat \familiar{} in $\PP_{\kk} V$,
 so that $\phi^{\univ}:\ccR^{\univ} \to \ccH$
 is flat and finite of degree $r$.
By a slight abuse of notation we denote by $\nu_d (\ccR^{\univ})\subset \ccH \times \PP_{\kk}(S^d V)$ the image of $\ccR^{\univ}$ under $\id_{\ccH} \times \nu_d$.
Thus, for $d\geqslant r-1$, 
the morphism $\nu_d (\ccR^{\univ}) \to \ccH$ is a flat independent \familiar{} in $\nu_d(\PP_{\kk} V)$ over $\ccH$.
Now consider the linear span and relative linear span of
$\nu_d(\ccR^{\univ})$:
\[
\begin{tikzcd}
\ccH \times
    \PP_{\kk}( S^{(d)}V)
 \arrow[rr, twoheadrightarrow]
 \arrow[ddr, twoheadrightarrow]
 & & \PP_{\kk}( S^{(d)}V) \\
 & \linspan{\nu_d(\ccR^{\univ})} \arrow[r, twoheadrightarrow]
 \arrow[d, twoheadrightarrow, "\linspan{\nu_d(\phi^{\univ})}"]
 \arrow[ul, hook']
 &\rellinspan{\nu_d(\ccR^{\univ})}
 \arrow[u, hook']
 \\
 & \ccH  &
\end{tikzcd}
\]
Recall from Lemma~\ref{lem_cactus_flat_is_determined_by_Hilb} that $\rellinspan{\nu_d(\ccR^{\univ})} = \cactusflat{r}{\nu_d(\PP V)}$. We claim in Theorem~\ref{thm_universal_linear_span_isomorphic_to_cactus} that the surjective map
$\linspan{\nu_d(\ccR^{\univ})} \to \cactusflat{r}{\nu_d(\PP V)}$ is in fact an isomorphism of schemes when restricted to  suitable open subschemes.
We proceed to define these open subschemes, or more precisely, their complements, which are closed subsets of $\linspan{\nu_d(\ccR^{\univ})}$ and $\cactusflat{r}{\nu_d(\PP V)}$ respectively.

First, we will construct $Z \subset \linspan{\nu_d(\ccR^{\univ})}$ as the union of all linear spans of strict subschemes of fibres of $\nu_d(\ccR^{\univ}) \to \ccH$.
More formally, let $\tilde{\ccH}\to \ccH$ be the flag Hilbert scheme
 parameterizing pairs $R'\subset R\subset \PP_{\kk} V$ with length of $R$ equal to $r$
 and length of $R'$ equal to $r-1$.
Then let $\tilde{\ccR}^{\univ} \subset \tilde{\ccH} \times_{\ccH} \ccR^{\univ}$ be the universal family of the smaller scheme $R'$.
Thus, the projection map 
$\tilde{\ccR}^{\univ}\to \ccR^{\univ}$ takes the scheme $R'$ over a pair $R'\subset R$ in $\tilde{\ccH}$ onto its image inside $R$. The linear span $\linspan{\nu_d(\tilde{\ccR}^{\univ})}$ of the \familiar[(r-1)]{} $\nu_d(\tilde{\ccR}^{\univ})\to \tilde{\ccH}$ maps into
$\linspan{\nu_d(\ccR^{\univ})}$.
Its image is our desired scheme $Z \subset \linspan{\nu_d(\ccR^{\univ})}$. Note that all the schemes $\tilde{\ccH}$, $\ccR^{\univ}$  and
$\linspan{\nu_d(\tilde{\ccR}^{\univ})}$ are projective, and thus $Z$ is a closed subscheme, whose all points are images of points in $\linspan{\nu_d(\tilde{\ccR}^{\univ})}$, that is they are all in the linear span of $\nu_d(R')$ for a subscheme $R'$ length $r-1$.

Secondly, the other open subscheme is simply
$\cactusflat{r}{\nu_d(\PP V)}\setminus \cactusflat{r-1}{\nu_d(\PP V)}$.
Note that $\cactusflat{r}{\nu_d(\PP V)}\setminus \cactusflat{r-1}{\nu_d(\PP V)}$ is open and dense in $\cactusflat{r}{\nu_d(\PP V)}$ while $\linspan{\nu_d(\tilde{\ccR}^{\univ})}\setminus Z$ is open, but not always dense in $\linspan{\nu_d(\tilde{\ccR}^{\univ})}$.

\begin{theorem}\label{thm_universal_linear_span_isomorphic_to_cactus}
   The following two quasiprojective schemes are isomorphic:
   \[
     \linspan{\nu_d(\ccR^{\univ})} \setminus Z \simeq \cactusflat{r}{\nu_d(\PP V)}\setminus \cactusflat{r-1}{\nu_d(\PP V)}
   \]
\end{theorem}

\begin{prf}
  Note that by construction of $Z$ and Proposition~\ref{prop_cactus_scheme_supported_at_cactus_variety} the image of the support of  $Z$ is in $\cactusflat{r-1}{\nu_d(\PP V)}$.
  Moreover, $\linspan{\nu_d(\ccR^{\univ})}$
   maps surjectively onto $\cactusflat{r}{\nu_d(\PP V)}$.

  Thus, it remains to prove:
  \renewcommand{\theenumi}{(\Alph{enumi})}
  \begin{enumerate}
   \item \label{item_well_defined}
   If $p\in \linspan{\nu_d(\ccR^{\univ})} \setminus Z $ then its image is not in $\cactusflat{r-1}{\nu_d(\PP V)}$.
   \item \label{item_1_to_1}
   If $p,q\in \linspan{\nu_d(\ccR^{\univ})}$
   are two different $\kk$-points that are mapped to the same $F\in \cactusflat{r}{\nu_d(\PP V)} \subset \PP_{\kk}(S^{(d)}V)$,
   then $p$ and $q$ are both in $Z$.
   \item \label{item_separates_tangent_directions}
   If
   $Q= \Spec \kk[t]/(t^2) \subset\linspan{\nu_d(\ccR^{\univ})}$ and
   $Q$ is contracted to a reduced point $F$ in
   $ \cactusflat{r}{\nu_d(\PP V)} \subset \PP_{\kk}(S^{(d)}V)$ under the projection,
   then the support of $Q$ is in $Z$.
   \item \label{item_surjective_on_tangent_directions}
   If $Q= \Spec\kk[t]/(t^2) \subset \cactusflat{r}{\nu_d(\PP V)}\setminus \cactusflat{r-1}{\nu_d(\PP V)}$,
    then $Q$ lifts to
    $\linspan{\nu_d(\ccR^{\univ})} \setminus Z$.
  \end{enumerate}
By Item \ref{item_well_defined}, the projection     $\linspan{\nu_d(\ccR^{\univ})} \setminus Z \to \cactusflat{r}{\nu_d(\PP V)}\setminus \cactusflat{r-1}{\nu_d(\PP V)}$ is well defined,
by Item~\ref{item_1_to_1} it is $1$ to $1$ on points, and by Item~\ref{item_separates_tangent_directions} it separates tangent directions, thus it is an injective morphism. Finally, Item~\ref{item_surjective_on_tangent_directions} shows that the morphism is surjective on tangent spaces.

Items~\ref{item_well_defined} and \ref{item_1_to_1} are set-theoretic and they have been treated in \cite[Thm~1.6(ii)]{nisiabu_jabu_cactus} and \cite[Cor.~4.17, Lem.~4.18]{jabu_jelisiejew_finite_schemes_and_secants}. Roughly, if there are two presentations of $F\in \cactusflat{r}{\nu_d(\PP V)}$: $F\in \linspan{\nu_d(R_1)} \cap \linspan{\nu_d(R_2)}$ with $\dim R_1=\dim R_2=0$, and $\deg(R_1), \deg(R_2) \le r$ and $d\ge 2r-1$, then $F\in \linspan{\nu_d(R_1 \cap R_2)}$ by \cite[Cor.~2.7]{jabu_ginensky_landsberg_Eisenbuds_conjecture},  \cite[Lem.~4.18]{jabu_jelisiejew_finite_schemes_and_secants}.
Therefore, we have $p, q\in Z$.

Now we prove Item~\ref{item_separates_tangent_directions}
following a similar idea.
Consider the map $\linspan{\nu_d(\phi^{\univ})}|_{Q}\colon Q\to \ccH$. Note that this map cannot contract $Q$ to a reduced point, as in such case $Q$ would be contained in the fibre of
$\linspan{\nu_d(\phi^{\univ})}$, and the fibre (which is a linear subspace of $\PP_{\kk}(S^{(d)}V)$) is embedded into $\PP_{\kk}(S^{(d)}V)$ under the second projection (so it cannot contract $Q$ to a point $F$).
We can use $\linspan{\nu_d(\phi^{\univ})}|_{Q}$ to pullback the universal family of subschemes and we obtain a non-constant flat \familiar{} $\ccR\to Q$ in $\PP_{\kk} V$.
Let $I(\ccR) \subset \kk[t]/(t^2)\otimes_{\kk} \Sym V^*$ be the ideal of $\ccR$.
Using the embedding $\kk\to \kk[t]/(t^2)$ we consider $F$ as an element of $\kk[t]/(t^2) \otimes_{\kk} S^{(d)} V$, so that $\ann \subset \kk[t]/(t^2)\otimes_{\kk} \Sym V^*$.
By the relative apolarity, Theorem~\ref{thm_relative_apolarity_lemma},
we must have $I(\ccR)\subset \ann$.

Analogously to \S\ref{sec_Gorenstein_duality},  for $\Theta \in \kk[t]/(t^2)\otimes_{\kk} \Sym V^*$ write it as  $\Theta= \Psi + t \Psi'$ with $\Psi, \Psi' \in \Sym V^*$ and define $\partial \Theta:=\Psi'$,
and define the ideal $\partial I(\ccR) \subset \Sym V^*$:
\[
  \partial I(\ccR) :=
  \set{\partial\Theta \mid \Theta\in I(\ccR)}.
\]
Let $\partial \ccR = Z(\partial I(\ccR))$ be the subscheme of $\PP_{\kk} V$ defined by the homogeneous ideal $\partial I(\ccR)$, so that
$I(\partial \ccR) = \partial I(\ccR)^{\sat}$.
This scheme plays here the role of the intersection of ``schemes in the family $\ccR \to Q$'', by analogy to the case of two distict points.
 \renewcommand{\theenumi}{($\mathfrak{\roman{enumi}}$)}
At this moment we do not even know if $\partial \ccR \neq \emptyset$, but we will prove it at the end of the argument.
Below we list and prove several properties of $\partial\ccR$ and $\partial \ccI(\ccR)$.
\begin{interruptedenumerate}
\begin{insideenumerate}
\item \label{item_I_R_k_subset_partial_I_R}
First, $I (\ccR)_{\kk} \subset \partial I(\ccR)$ --- by Lemma~\ref{lem_properties_of_partial}\ref{item_fibre_contained_in_partial}.
\item $\partial I(\ccR)\subset \ann$.
\label{item_partial_I_R_subset_Ann}
\end{insideenumerate}
For $\Theta=\Psi + t\Psi'\in I(\ccR) \subset \ann$ we have:
$
   0= \Theta\hook F =(\Psi + t\Psi') \hook F =
   \Psi \hook F + t(\Psi' \hook F),
$
hence both components $\Psi \hook F$ and $\Psi' \hook F$ are zero, proving~\ref{item_partial_I_R_subset_Ann}.
\begin{insideenumerate}
\item \label{item_partial_I_R_saturated}
The ideal $\partial I(\ccR)$ is saturated in degrees $i\geqslant 2r-1$, that is $\partial I(\ccR)_i=(\partial I(\ccR))^{\sat}_i$.
\end{insideenumerate}
Explicitly, $I(\ccR)$ is saturated,
and since $\ccR\to Q$ is flat and finite of degree $r$,
the $\kk$-codimension of
$I(\ccR)_i \subset \kk[t]/(t^2) \otimes_{\kk} S^i V^*$ is constantly equal to $2r$ for all $i\ge r-1$.
Moreover, $I(\ccR)\cap \Sym V^*$ is also saturated in $\Sym V^*$,
and its codimension in each degree is bounded from above by the same number,
that is:
for each $i$ we must have
 $\codim_{\kk} (I(\ccR)_i\cap S^i V^* \subset S^i V^*)\leqslant 2r$.
Thus, (by saturatedness) this latter codimension is
 constant (independent of $i$) at least for all $i\geqslant 2r-1$.
Moreover, by Lemma~\ref{lem_properties_of_partial}\ref{item_dimensions_of_partial_and_projection}  we  have (for $i\ge r-1$):
\begin{align*}
  \codim_{\kk} \left(\partial I(\ccR)_i  \subset S^i V^*\right)
+ \codim_{\kk} \left(I(\ccR) \cap S^i V^*\subset S^i V^*\right) &\stackrel{\text{Lem.~\ref{lem_properties_of_partial}\ref{item_dimensions_of_partial_and_projection}}}{=}\\
  \codim_{\kk} \left(I(\ccR)_i \subset \kk[t]/(t^2) \otimes_{\kk} S^i V^*\right) &= 2r.
\end{align*}
Therefore, for $i\geqslant 2r-1$ the codimension of $\partial I(\ccR)_i$ in $S^i V^*$
is constant and equal to $2r- \codim_{\kk} \left(I(\ccR) \cap S^i V^*\subset S^i V^*\right)$, which is at most $2r$ (in fact, it is at most $r$ by \ref{item_I_R_k_subset_partial_I_R}),
and thus the ideal $\partial I(\ccR)$ is saturated in these degrees proving \ref{item_partial_I_R_saturated}.
\begin{insideenumerate}
\item \label{item_I_R_neq_partial_I_R}
Finally, $\ccR_{\kk} \neq \partial \ccR$.
\end{insideenumerate}
Suppose by contradiction that $\partial \ccR = \ccR_{\kk}$.
Translating this into the language of ideals, 
$(I(\ccR)_{\kk})^{\sat}= I(\partial \ccR)$.
Since $I(\partial \ccR)  = (\partial I(\ccR))^{\sat}$,
thus for sufficiently large degrees $i$
we must have $\left(I (\ccR)_{\kk}\right)_i = \left(\partial I(\ccR)\right)_{i}$.
Then for each $\Theta=\Psi + t\Psi'\in I(\ccR)_{i}$
both $\Psi$ and $\Psi'$ are in $\left(I (\ccR)_{\kk}\right)_i$,
  thus $I(\ccR)_i=(I (\ccR)_{\kk})_i\otimes \kk[t]/(t^2)$.
This equality contradicts the nontriviality of the family (if this was the case, then the map to the Hilbert scheme $\ccH$ would contract $Q$ to a point), which proves \ref{item_I_R_neq_partial_I_R}.
\end{interruptedenumerate}

To conclude the proof of \ref{item_separates_tangent_directions},
since $d\geqslant 2r-1$, we have
\[
\partial I(\ccR)^{\sat}_d \stackrel{\text{by \ref{item_partial_I_R_saturated}}}{=} \partial I(\ccR)_d\stackrel{\text{by \ref{item_partial_I_R_subset_Ann}}}{\subset} \ann_d, \text{ and thus } I(\partial \ccR) = \partial I(\ccR)^{\sat} \stackrel{\text{by Lem.~\ref{lem_ideals_contained_in_annihilators}}}{\subset} \ann.
\]
Since $\ann_0=0$, so also $\partial I(\ccR)^{\sat}_0=0$,
it follows that $F\in \linspan{\nu_d(\partial \ccR)}$ by Theorem~\ref{thm_relative_apolarity_lemma}.
In particular, $\partial \ccR\neq \emptyset$.
We also know that $\partial \ccR \subsetneqq \ccR_{\kk} \subset \PP_{\kk} V$ by \ref{item_I_R_k_subset_partial_I_R} and \ref{item_I_R_neq_partial_I_R}, so that $\partial \ccR$ is a finite subscheme of degree less than $r$.
That is, $F\in Z$ as claimed in \ref{item_separates_tangent_directions}.

Finally, to prove \ref{item_surjective_on_tangent_directions}
suppose
\[
Q= \Spec\kk[t]/(t^2) \subset \cactusflat{r}{\nu_d(\PP V)}\setminus \cactusflat{r-1}{\nu_d(\PP V)}.
 \]
Then, by Theorem~\ref{thm_cactus_scheme_contained_in_vanishing_of_catalecticant_minors}
we also have
$Q\subset \set{\Upsilon^{i,d-i}=r}$.
By Lemma~\ref{lem_Q_in_Upsilon_lifts_to_Hilb}
  there is a flat \familiar{} $\ccR\to \Spec \kk[t]/(t^2) \simeq Q$ in $\PP_{\kk} V$
  such that $Q\subset \rellinspan{\nu_d(\ccR)} \subset \PP_{\kk} (S^{(d)}V)$.
Thus, $Q$ admits two maps: 
$Q\to \ccH$ coming from the universal property of the Hilbert scheme and $Q\to \PP_{\kk}(S^{(d)} V)$ coming from the original embedding.
Consequently, we obtain the map into the product $\ccH\times \PP_{\kk}(S^{(d)} V)$
and by construction the image is contained
in $\linspan{\nu_d(\ccR)} \subset\linspan{\nu_d(\ccR^{\univ})}$.
This produces the desired lifting concluding the proof of \ref{item_surjective_on_tangent_directions} and also of the theorem.
\end{prf}

\begin{prf}[ of Theorem~\ref{thm_universal_linear_span_isomorphic_to_cactus_intro}]
Let $\ccH^{\Gor}:=\HilbGor{r}{\PP_{\kk} V}$ be the Hilbert scheme of Gorenstein subchemes of length $r$.
By construction  \cite[Def.~4.33]{jelisiejew_PhD} the Gorenstein Hilbert scheme is an open (not necessarily dense) subscheme
$\ccH^{\Gor} \subset \ccH$. Consider the map
$\linspan{\nu_d(\ccR^{\univ})} \to \ccH$.
Each $\kk$-point  $[R]\in \ccH\setminus \ccH^{\Gor}$ corresponds to a finite subscheme $R\subset\PP_{\kk} V$ of degree $r$.
By Corollary~\ref{cor_base_change_of_linear_span_of_independent_familiar}
the fibre $\linspan{\nu_d(\ccR^{\univ})}_{[R]}$ is equal to
  $\linspan{\nu_d(R)}$. Considering the restricted map
$\linspan{\nu_d(\ccR^{\univ})} \setminus Z \to \ccH$
we claim that:
\begin{itemize}
 \item
    the map factorises through $\ccH^{\Gor}$,
    that is $\linspan{\nu_d(\ccR^{\univ})}\setminus Z  \to \ccH^{\Gor} \hookrightarrow \ccH$, and
 \item the map $\linspan{\nu_d(\ccR^{\univ})}\setminus Z  \to \ccH^{\Gor}$ is surjective.
\end{itemize}

The first claim (about factorization) is equivalent to equality of the fibres
 $Z_{[R]}=\linspan{\nu_d(\ccR^{\univ})}_{[R]} = \linspan{\nu_d(R)}$ over each $\kk$-point  $[R]\in \ccH\setminus \ccH^{\Gor}$.
The latter span is covered by the spans of the smaller subschemes $\nu_d(R')\subsetneqq \nu_d(R)$ by \cite[Lem.~2.3]{nisiabu_jabu_cactus} or
\cite[Lem.~6.19]{jabu_jelisiejew_finite_schemes_and_secants}.
Hence $\linspan{\nu_d(R)}\subset Z_{[R]}$.
Since
$Z\subset \linspan{\nu_d(\ccR^{\univ})}$,
the other inclusion
$Z_{[R]} \subset \linspan{\nu_d(\ccR^{\univ})}_{[R]}$ is automatic.

The second claim (about surjectivity) is also set theoretic,
hence pick a $\kk$-point
$[R]\in \ccH^{\Gor}$
and consider the fibre
$\linspan{\nu_d(R)}$.
We want to prove that
$Z\cap \left( \set{[R]}\times \linspan{\nu_d(R)}\right)$ is not dense.
By construction of $Z$ this intersection is equal to the union of linear spans of smaller schemes,
\[
  Z\cap \left( \set{[R]}\times \linspan{\nu_d(R)}\right) = \bigcup_{R'\subsetneq R, \ \deg R'=r-1}{\linspan{\nu_d(R')}}.
\]
There are only finitely many subschemes $R'\subsetneq R$ of degree $r-1$
by \cite[Lem.~5.16]{jabu_jelisiejew_finite_schemes_and_secants} and each span  $\linspan{\nu_d(R')}$ is of dimension $r-2$.
Thus, $Z\cap \left( \set{[R]}\times \linspan{\nu_d(R)}\right)$
  cannot cover $\linspan{\nu_d(R)} \simeq \PP^{r-1}$.
 That is $Z\cap \left( \set{[R]}\times \linspan{\nu_d(R)}\right)$ is not dense and there is a $\kk$-point in the complement
 $\linspan{\nu_d(\ccR^{\univ})} \setminus Z$ that is mapped to $[R]$, as claimed.

Let $\ccE$ be the vector bundle over $\ccH^{\Gor}$ such that
\[
\linspan{\nu_d(\ccR^{\univ})}|_{\ccH^{\Gor}} = \PP_{\kk} (\ccE)
\]
and $Z^{\Gor}: = Z\cap \PP_{\kk} (\ccE)$.
By the first item above, we have $\PP_{\kk} (\ccE) \setminus Z^{\Gor} = \linspan{\nu_d(\ccR^{\univ})} \setminus Z$, while the latter is isomorphic to $\cactusflat{r}{\nu_d(\PP V)}\setminus \cactusflat{r-1}{\nu_d(\PP V)}$
by Theorem~\ref{thm_universal_linear_span_isomorphic_to_cactus}.
Moreover, $\PP_{\kk} (\ccE) \setminus Z^{\Gor} \to \ccH^{\Gor}$ is surjective by the second item above. These statements complete the proof of the theorem.
\end{prf}

As a concluding remark we note that in all Theorems~\ref{thm_cactus_scheme_and_minors_for_PV},
\ref{thm_universal_linear_span_isomorphic_to_cactus_intro},
\ref{thm_Upsilon_contained_in_cactus}, and
\ref{thm_universal_linear_span_isomorphic_to_cactus} we assume $d\geqslant 2r$ (and $r\leqslant i \leqslant d-r$, where $i$ is relevant).
In the set theoretic variant of the same results, \cite[\S8.2]{nisiabu_jabu_cactus}, it is hinted that similar conclusions should also be true for $d=2r-1$ (and $i=r-1$ or $i=r$).
In fact, going carefully through the proofs above, the reader can check that the only problem with $d=2r-1$ is Lemma~\ref{lem_Q_in_Upsilon_lifts_to_Hilb}.
It should be possible to generalise this lemma to the case $d=2r-1$ following the same idea as in \cite[Prop.~8.2]{nisiabu_jabu_cactus}, but it requires much more care and considering many subcases.
We leave this as a project for future research.

\addcontentsline{toc}{section}{References}

\bibliographystyle{alpha}
\bibliography{cactus_scheme.bbl}

\def\polhk#1{\setbox0=\hbox{#1}%
  {\ooalign{\hidewidth\lower1.5ex\hbox{`}\hidewidth\crcr\unhbox0}}}\def\dbar{\leavevmode\hbox
  to 0pt{\hskip.2ex\accent"16\hss}d}
\begin{thebibliography}{BJMR18}

\bibitem[BB14]{nisiabu_jabu_cactus}
Weronika Buczy{\'n}ska and Jaros{\l}aw Buczy{\'n}ski.
\newblock Secant varieties to high degree {V}eronese reembeddings,
  catalecticant matrices and smoothable {G}orenstein schemes.
\newblock {\em J. Algebraic Geom.}, 23:63--90, 2014.

\bibitem[BB21]{nisiabu_jabu_border_apolarity}
Weronika Buczy{\'n}ska and Jaros{\l}aw Buczy{\'n}ski.
\newblock Apolarity, border rank, and multigraded {H}ilbert scheme.
\newblock {\em Duke Math. J.}, 170(16):3659--3702, 2021.

\bibitem[BGL13]{jabu_ginensky_landsberg_Eisenbuds_conjecture}
Jaros{\l}aw Buczy{\'n}ski, Adam Ginensky, and Joseph~M. Landsberg.
\newblock Determinantal equations for secant varieties and the
  {E}isenbud-{K}oh-{S}tillman conjecture.
\newblock {\em J. Lond. Math. Soc. (2)}, 88(1):1--24, 2013.

\bibitem[BJ17]{jabu_jelisiejew_finite_schemes_and_secants}
Jaros{\l}aw Buczy\'{n}ski and Joachim Jelisiejew.
\newblock Finite schemes and secant varieties over arbitrary characteristic.
\newblock {\em Differential Geom. Appl.}, 55:13--67, 2017.

\bibitem[BJMR18]{bernardi_jelisiejew_marques_ranestad_cactus_rank_of_a_general_form}
Alessandra Bernardi, Joachim Jelisiejew, Pedro~Macias Marques, and Kristian
  Ranestad.
\newblock On polynomials with given {H}ilbert function and applications.
\newblock {\em Collect. Math.}, 69(1):39--64, 2018.

\bibitem[BT20]{bernardi_taufer_Waring_tangential_cactus_decompositions}
Alessandra Bernardi and Daniele Taufer.
\newblock Waring, tangential and cactus decompositions.
\newblock {\em J. Math. Pures Appl. (9)}, 143:1--30, 2020.

\bibitem[CLPS24]{choi_lacini_park_sheridan_sings_and_syz_of_secant_vars}
Doyoung Choi, Justin Lacini, Jinhyung Park, and John Sheridan.
\newblock Singularities and syzygies of secant varieties of smooth projective
  varieties.
\newblock in preparation, 2024.

\bibitem[Eis95]{Eisenbud}
David Eisenbud.
\newblock {\em Commutative algebra}, volume 150 of {\em Graduate Texts in
  Mathematics}.
\newblock Springer-Verlag, New York, 1995.
\newblock With a view toward algebraic geometry.

\bibitem[Ems78]{emsalem}
Jacques Emsalem.
\newblock G{\'e}om{\'e}trie des points {\'e}pais.
\newblock {\em Bull. Soc. Math. Fr.}, 106:399--416, 1978.

\bibitem[FH21]{furukawa_han_sing_of_secant_vars_of_Veronese}
Katsuhisa Furukawa and Kangjin Han.
\newblock On the singular loci of higher secant varieties of {V}eronese
  embeddings.
\newblock arXiv:2111.03254, 2021.

\bibitem[Ga{\l}23a]{galazka_mgr_publ}
Maciej Ga{\l}{\k a}zka.
\newblock Multigraded apolarity.
\newblock {\em Mathematische Nachrichten}, 296(1):286--313, 2023.

\bibitem[Ga{\l}23b]{galazka_phd}
Maciej Ga{\l}{\k a}zka.
\newblock {\em Secant varieties, {W}aring rank and generalizations from
  algebraic geometry viewpoint}.
\newblock PhD thesis, University of Warsaw, Faculty of Mathematics, Informatics
  and Mechanics, 2023.
\newblock
  \url{https://www.mimuw.edu.pl/~jabu/teaching/Theses/PhD_thesis_Galazka.pdf}.

\bibitem[Ger99]{geramita_catalecticant_varieties}
Anthony~V. Geramita.
\newblock Catalecticant varieties.
\newblock In {\em Commutative algebra and algebraic geometry ({F}errara)},
  volume 206 of {\em Lecture Notes in Pure and Appl. Math.}, pages 143--156.
  Dekker, New York, 1999.

\bibitem[GMR22]{galazka_mandziuk_rupniewski_distinguishing}
Maciej Ga{\l}{\k{a}}zka, Tomasz Ma{\'n}dziuk, and Filip Rupniewski.
\newblock Distinguishing secant from cactus varieties.
\newblock {\em Foundations of Computational Mathematics}, pages 1--48, 2022.

\bibitem[Gre98]{green_generic_initial_ideals}
Mark~L. Green.
\newblock Generic initial ideals.
\newblock In {\em Six lectures on commutative algebra}, volume 166 of {\em
  Progress in Mathematics}, pages 119--186. Birkh\"auser Verlag, Basel, 1998.

\bibitem[Har77]{hartshorne}
Robin Hartshorne.
\newblock {\em Algebraic geometry}.
\newblock Springer-Verlag, New York, 1977.
\newblock Graduate Texts in Mathematics, No. 52.

\bibitem[HS04]{haiman_sturmfels_multigraded_Hilb}
Mark Haiman and Bernd Sturmfels.
\newblock Multigraded {H}ilbert schemes.
\newblock {\em J. Algebraic Geom.}, 13(4):725--769, 2004.

\bibitem[IK99]{iarrobino_kanev_book_Gorenstein_algebras}
Anthony Iarrobino and Vassil Kanev.
\newblock {\em Power sums, {G}orenstein algebras, and determinantal loci},
  volume 1721 of {\em Lecture Notes in Mathematics}.
\newblock Springer-Verlag, Berlin, 1999.
\newblock Appendix C by Iarrobino and Steven L. Kleiman.

\bibitem[Jel17]{jelisiejew_PhD}
Joachim Jelisiejew.
\newblock {\em Hilbert schemes of points and their applications}.
\newblock PhD thesis, University of Warsaw, 2017.
\newblock arXiv:2205.10584.

\bibitem[Jel18]{joachimapo}
Joachim Jelisiejew.
\newblock {VSPs} of cubic fourfolds and the {Gorenstein} locus of the {Hilbert}
  scheme of 14 points on {{\(\mathbb{A}^6\)}}.
\newblock {\em Linear Algebra Appl.}, 557:265--286, 2018.

\bibitem[KK22]{KlKlep}
Steven~L. Kleiman and Jan~O. Kleppe.
\newblock Macaulay duality and its geometry.
\newblock Preprint, {arXiv}:2210.10934 [math.{AG}] (2022), 2022.

\bibitem[Puc98]{pucci_Veronese_variety_and_catalecticant_matrices}
Mario Pucci.
\newblock The {V}eronese variety and catalecticant matrices.
\newblock {\em J. Algebra}, 202(1):72--95, 1998.

\bibitem[Rai13]{raicu_3_times_3_minors}
Claudiu Raicu.
\newblock {$3\times3$} minors of catalecticants.
\newblock {\em Math. Res. Lett.}, 20(4):745--756, 2013.

\bibitem[RS00]{ranestad_schreyer_VSP}
Kristian Ranestad and Frank-Olaf Schreyer.
\newblock Varieties of sums of powers.
\newblock {\em J. Reine Angew. Math.}, 525:147--181, 2000.

\bibitem[RV17]{ranestad_voisin_VSP_and_divisors_in_the_moduli_of_cubic_fourfolds}
Kristian Ranestad and Claire Voisin.
\newblock Variety of power sums and divisors in the moduli space of cubic
  fourfolds.
\newblock {\em Doc. Math.}, 22:455--504, 2017.

\bibitem[{Sta}24]{stacks_project}
The {Stacks Project Authors}.
\newblock {\itshape Stacks Project}.
\newblock \url{http://stacks.math.columbia.edu}, 2024.

\end{thebibliography}

\end{document}